\documentclass[10pt,reqno,a4paper,twoside]{amsart}
\usepackage{amssymb,amsmath,amsthm,mathrsfs,graphicx,color,cite}
\theoremstyle{plain}
\newtheorem{theorem}{Theorem}[section]
\newtheorem{conjecture}{Conjecture}[section]
\newtheorem{proposition}[theorem]{Proposition}
\newtheorem{definition}[theorem]{Definition}
\newtheorem{lemma}[theorem]{Lemma}
\newtheorem{corollary}[theorem]{Corollary}

\theoremstyle{remark}
\newtheorem{remark}[theorem]{Remark}

\usepackage
[bookmarks=true,
   bookmarksnumbered=false,
   bookmarkstype=toc]
{hyperref}

\numberwithin{equation}{section}

\newcommand{\C}{\mathbb{C}}
\newcommand{\R}{\mathbb{R}}
\newcommand{\Z}{\mathbb{Z}}
\newcommand{\N}{\mathbb{N}}

\renewcommand{\Im}{\operatorname{Im}}
\renewcommand{\Re}{\operatorname{Re}}
\newcommand{\I}{\infty}
\newcommand{\abs}[1]{\left\lvert #1\right\rvert}
\newcommand{\Jbr}[1]{\left\langle #1 \right\rangle}
\newcommand{\norm}[1]{\left\lVert #1\right\rVert}

\newcommand{\tnorm}[1]{\lVert #1\rVert}

\newcommand{\IN}{\quad\text{in }}
\newcommand{\wIN}{\quad\text{weakly in }}

\def\({\left(}
\def\){\right)}
\def\<{\left\langle}
\def\>{\right\rangle}
\def\le{\leqslant}
\def\ge{\geqslant}

\newcommand{\qtq}[1]{\quad\text{#1}\quad}

\def\d{{\partial}}
\def\l{\lambda}

\newcommand{\eps}{\varepsilon}

\DeclareMathOperator{\sign}{sign}

\begin{document}
\title[NLS with potential]{Global dynamics below excited solitons for the non-radial NLS with potential}
\author[S. Masaki]{Satoshi Masaki}
\address{Department of Systems Innovation\\
Graduate School of Engineering Science\\
Osaka University\\
Toyonaka Osaka, 560-8531, Japan}
\email{masaki@sigmath.es.osaka-u.ac.jp}

\author[J. Murphy]{Jason Murphy}
\address{Department of Mathematics \& Statistics \\ Missouri University of Science \& Technology \\ Rolla, MO, USA}
\email{jason.murphy@mst.edu}

\author[J. Segata]{Jun-ichi Segata}
\address{Faculty of Mathematics, Kyushu University, Fukuoka, 819-0395, Japan}
\email{segata@math.kyushu-u.ac.jp}

\begin{abstract} We consider the global dynamics of solutions to the $3d$ cubic nonlinear Schr\"odinger equation in the presence of an external potential, in the setting in which the equation admits both ground state solitons and excited solitons at small mass.  We prove that small mass solutions with energy below that of the excited solitons either scatter to the ground states or grow their $H^1$-norm in time.  In particular, we give an extension of the result of Nakanishi \cite{Nakanishi} from the radial to the non-radial setting.
\end{abstract}

%



\maketitle


\section{Introduction}
We consider the following cubic nonlinear Schr\"odinger equation (NLS) in three space dimensions:
\begin{equation}\label{e:nls}
\begin{cases}
i \d_t u+ H u =  \sigma |u|^2 u, \\
u|_{t=0} = u_0 \in H^1(\R^3),
\end{cases}
\end{equation}
where $H=-\Delta + V$ for some external potential $V:\R^3\to\R$, and the sign $\sigma\in\{\pm 1\}$ gives the \emph{focusing} and \emph{defocusing} equations, respectively.  We consider sufficiently smooth and decaying potentials supporting a single bound state.  In this setting, one finds that \eqref{e:nls} supports small solitary waves (ground states), along with excited solitons in the focusing setting.  

When $V\equiv0$ and the nonlinearity is defocusing, all $H^{1}$ solutions are global and scatter to free solutions as $t\to\pm\infty$ \cite{GV}; in the focusing case, $H^{1}$ solutions with mass-energy below the ground state either scatter or blow-up in finite time \cite{AkahoriNawa, DHR}. In the case that $V$ has negative eigenvalues, (\ref{e:nls}) has small solitary waves, leading to more complicated dynamics for solutions. When $V$ has a single negative eigenvalue, solutions with sufficiently small data in $H^1$ scatter to the family of ground states as $t\to\pm\infty$ \cite{GNT} (see also \cite{SW,Kirr1,Kirr2, Mizumachi1,Mizumachi2}). The case when $V$ has several negative eigenvalues has also been studied in several works (see e.g. \cite{CM,NPT,SW3, TY1,TY2,TY3,TY3}), although this case is not as well-understood in general.

In \cite{Nakanishi}, Nakanishi considered the dynamics of radial solutions with small mass and with energy below that of the first excited state, in the case that $V$ has a single negative eigenvalue.  He proved that solutions either blow up in finite time or scatter to the family of ground states (with a criterion in terms of the initial data to distinguish these two cases).  In further work, he classified the behavior of radial solutions with small mass and energy slightly above the first excited state \cite{Nakanishi2}.  In this paper, we prove a result analogous to the one in \cite{Nakanishi}, but in the \emph{non-radial} setting.  In particular, in the focusing case, we establish a scattering/grow-up dichotomy below the excited energy and under a small mass assumption.  We also obtain scattering to the ground states for small mass and arbitrary energy in the defocusing case (under the assumption that there are no excited solitons in this setting; cf. Conjecture~\ref{conjecture}). 

Before we can state our main result precisely, we need to introduce some notation.  First, we denote the conserved \emph{energy} for \eqref{e:nls} by
\[
\mathbb{E}_V:= \mathbb{H}_V - \mathbb{G},
\]
which is the sum of the linear and nonlinear components, namely, 
\begin{equation}\label{HVandG}
\begin{aligned}
\mathbb{H}_V(\varphi) & :=   \int_{\R^3}\tfrac12 |\nabla \varphi(x)|^2 +  \tfrac12 V(x) |\varphi(x)|^2\,dx,  \\
\mathbb{G} (\varphi) &:=  \tfrac{\sigma}4 \int_{\R^3} |\varphi(x)|^4  dx.
\end{aligned}
\end{equation}
We will write $\mathbb{H}_0$ and $\mathbb{E}_0$ to designate the functionals corresponding to the free case, i.e. with $V\equiv 0$.  The conserved \emph{mass} is denoted by
\[
\mathbb{M}(\varphi) :=  \tfrac12\int_{\R^3} |\varphi(x)|^2  dx.	
\]

The precise assumptions made on the potential will be discussed in Section~\ref{S:potential} below.  The essential requirements are that $V$ be sufficiently smooth and decaying, and that $H=-\Delta+V$ has one simple, negative eigenvalue. This assumption implies that $V$ has a nonzero negative part.  The ground states and excited states for \eqref{e:nls} are discussed in detail in Section~\ref{S:solitons}.  The set of $H^1$ soliton profiles will be denoted by $\mathscr{S}$.  The quantity $\mathscr{E}_0(\mu)$ denotes the ground state energy at mass $\mu$, with $\mathscr{S}_0(\mu)$ denoting the set of soliton profiles with mass $\mu$ and energy $\mathscr{E}_0(\mu)$.  Similarly, $\mathscr{E}_1(\mu)$ denotes the excited soliton energy at mass $\mu$, with $\mathscr{S}_1(\mu)$ denoting the set of excited soliton profiles. We will show that
\begin{equation}\label{e:me_comparison}
\mu \mathcal{E}_1(\mu) < \mathbb{M}(Q) \mathbb{E}_{0}(Q)
\end{equation}
for small $\mu$, where $Q$ is the ground state for the standard cubic NLS; that is, the mass-energy level of the excited state is below that of the NLS ground state (see Theorem~\ref{P:excited}).  This inequality is reasonable, as $V$ has a non-trivial negative part.  We remark that, however, it is not clear whether or not this holds under the assumption of radial symmetry.

We may parametrize the ground state solitons as $\Phi=\Phi[z]$, where $z\in\C$ belongs to a sufficiently small disk $D$ (see Lemma~\ref{L:small-solitons}). We say that a solution $u$ to \eqref{e:nls} \emph{scatters to $\mathscr{S}_0$} if there exists a $C^1$ function $z:\R\to D$, $z_\pm\in\C$, and  $u_\pm\in H^1$ such that
\[
\lim_{t\to\pm\infty}\|u(t)-\Phi[z(t)]-e^{-it\Delta}u_\pm\|_{H^1} = 0\qtq{and}\lim_{t\to\pm\infty}|z(t)|=z_{\pm}. 
\]

Our main result is the following theorem. 

\begin{theorem}\label{T}  There exists $\mu_0>0$ sufficiently small such that the following holds:
\begin{itemize}
\item[(i)] Suppose $\sigma=1$.  Let $u_0\in H^1(\R^3)$ satisfy
\[
\mathbb{M}(u_0)\leq \mu_0 \qtq{and} \mathbb{E}_V(u_0)<\mathscr{E}_1(\mathbb{M}(u_0)),
\]
and let $u:[T_{-},T_{+}]\times\R^3\to\C$ denote the maximal-lifespan solution to \eqref{e:nls}.  If
\begin{equation}\label{grow-up-condition}
\mathbb{K}_{V,2}(u_0)<0 \qtq{and} \|\nabla u_0\|_{L^2}>1,
\end{equation}
then $\lim_{t\to T_\pm} \|u(t)\|_{H^1}=\infty$. Otherwise, $u$ is global and scatters to $\mathscr{S}_0$.
\item[(ii)] Suppose $\sigma=-1$ and assume Conjecture~\ref{conjecture}.  If $u_0\in H^1(\R^3)$ satisfies $\mathbb{M}(u_0)\leq \mu_0$, then $u$ is global and scatters to $\mathscr{S}_0$. 
\end{itemize} 
\end{theorem}

Theorem~\ref{T} is a non-radial analogue of the main result in \cite{Nakanishi}.  It is not a complete extension of \cite{Nakanishi} in the sense that under the condition \eqref{grow-up-condition}, we have only established `grow-up' of the $H^1$-norm, as opposed to finite-time blowup. 

In the remainder of the introduction, we will briefly describe the strategy of the proof of Theorem~\ref{T}, focusing primarily on the scattering result.  The overarching strategy of the proof is the same as in the work of Nakanishi \cite{Nakanishi}, who proceeded by suitably adapting the concentration-compactness approach to induction on energy (also called the `Kenig--Merle roadmap').  Let us therefore briefly recall Nakanishi's argument for Theorem~\ref{T} in the radial setting.  We will then describe the main steps needed to extend this argument to the non-radial case.

The basic idea of the argument is as follows: first, one shows that if the theorem fails, then one can construct a non-scattering solution enjoying strong compactness properties in $H^1$.  One then derives a contradiction by showing that this solution is incompatible with some monotonicity formula satisfied by the nonlinear equation (in this case, the virial identity).  

Nakanishi's argument involved significant modifications to the scheme just described, due to the fact that one is now considering scattering to the family of ground states.  In particular, one decomposes the full solution to \eqref{e:nls} in the form $u(t)=\Phi[z(t)]+\eta(t)$ and essentially seeks to establish scattering for $\xi(t):=P_c\eta(t)$ and convergence of $|z(t)|$; here $P_c$ is the projection away from the underlying linear eigenfunction, and we may write $\eta=R[z]\xi$ for a suitable operator $R[\cdot]$ (cf. Section~\ref{S:solitons}).  The nonlinear evolution for $u(t)$ reduces to a coupled ODE-PDE system for $(z(t),\xi(t))$.  Proving scattering for $\xi$ is complicated by the fact that in the nonlinear equation the underlying linear part is now time-dependent (in particular, one obtains the linearized equation around the moving ground state $\Phi[z(t)]$).  In order to follow the `roadmap' in this setting, Nakanishi proceeded by developing concentration-compactness tools (in particular, a linear profile decomposition) adapted to the time-dependent \emph{linearized} equation.  In what follows, we describe a slight variation of Nakanishi's approach that aligns with the presentation in this paper; however, the fundamental ideas remain the same.

Turning to the details, under the assumption that the scattering theorem fails, one may construct a sequence of solutions $u_n=\Phi[z_n]+R[z_n]\xi_n$ approaching some `critical' mass-energy (below the excited energy) such that certain space-time norms of $\xi_n$ diverge as $n\to\infty$, both forward and backward in time.  The key is to prove strong $H^1$ convergence of $u_n(0)$ along some subsequence, for then the solution to \eqref{e:nls} with data given by $\lim_{n\to\infty} u_n(0)$ can be shown to have pre-compact orbit in $H^1$.  To achieve this, the first step is to decompose $u_n(0)$ in a linear profile decomposition.  This entails writing $u_n(0)$ as a sum of profiles of the form $\mathcal{L}(0,s_n^j;z_n)P_c \varphi^j$, where $\mathcal{L}(t,s;z)$ is the solution operator for the linearized equation around $\Phi[z(\cdot)]$ (see Definition~\ref{def:Ltsz}), plus an error whose `linearized evolution' is small in $L_t^\infty L_x^4$.  Under the assumption that more than one profile is present (so that convergence of $u_n(0)$ fails), one can then develop a corresponding \emph{nonlinear profile decomposition} for the sequence $u_n$, wherein one constructs scattering nonlinear solutions corresponding to each profile $\varphi^j$.  The fact that these solutions scatter relies on the `criticality' of the mass-energy of the sequence $u_n$.  By exploiting orthogonality properties of these profiles and stability theory for the nonlinear equation, one can then deduce uniform space-time bounds for the functions $\xi_n$, leading to the desired contradiction.  As mentioned above, with the pre-compact solution in hand, one can utilize the virial identity to obtain a contradiction and complete the proof; here the condition \eqref{grow-up-condition} plays a key role, as the failure of this condition helps to guarantee the coercivity of the virial identity in this setting. 

In \cite{Nakanishi}, Nakanishi relies on the radial assumption in an essential way, which can be traced back to the proof of the linearized profile decomposition.  The basic strategy for proving such a decomposition consists of iteratively identifying and removing `bubbles of concentration' from a given bounded sequence.  More precisely, these `bubbles' consist of fixed profiles together with sequences of parameters, which typically include space-time translations and scales, and the profiles themselves are obtained by applying these symmetries and passing to suitable weak limits.  Due to the subcritical nature of \eqref{e:nls}, no scaling parameters are needed in this setting.  In the radial case, space translations may also be avoided completely, as the compactness of the embedding $H^1_{\text{rad}}\hookrightarrow L^4$ allows one to obtain nontrivial weak limits directly.  Thus, in the linearized profile decomposition of \cite{Nakanishi}, one must contend only with time translations.  The same is then true in the nonlinear profile decomposition.  This ultimately implies that once one obtains the desired non-scattering solution, the solution automatically has pre-compact orbit in $H^1$; in particular, one does not need to incorporate any moving spatial center to obtain compactness.  Such solutions are then well-adapted to the localized virial argument. 

To address the non-radial case, we must first establish a linear profile decomposition that incorporates space translations when constructing the linear profiles.  A key observation in this setting is the fact that in the regime $|x|\to\infty$, the effects of both the external potential $V$ and the time-dependent potentials (involving the ground states) become weak.  In particular, at the level of the linear profile decomposition, we are able to utilize approximation by the linear propagator $e^{-it\Delta}$.  This observation plays an important role in the nonlinear profile decomposition, as well; in particular, in the regime $|x|\to\infty$, we can employ approximation by the standard cubic NLS.  Here we rely in an essential way on the inequality \eqref{e:me_comparison}; indeed, this estimate ensures that working below the excited solitons actually puts us in the scattering region for the underlying cubic NLS  (cf. \cite{DHR}).  In particular, this implies that nonlinear profiles in the regime $|x|\to\infty$ must scatter, which ultimately means that such profiles are not present in the construction of the non-scattering solution.  Consequently, the non-scattering solution once again has pre-compact orbit in $H^1$, without the need to incorporate a moving spatial center.  This puts us in the same position as \cite{Nakanishi}, and hence the rest of the argument (i.e. the localized virial argument) follows as in that work.  Thus, the primary contribution of the present work is to extend both the linear and nonlinear profile decompositions to the non-radial setting.  

The observation that the effect of the potential and ground states becomes weak in the regime $|x|\to\infty$, and the fact that the non-scattering solution is consequently compact in $H^1$ without any moving spatial center, is related to several other recent works on scattering for NLS in the presence of external potentials (see e.g. \cite{KMVZZ, KMVZ}).  In the present setting, the underlying linear and nonlinear problems are quite a bit more complicated (e.g. due to the presence of the time-dependent ground states), and the amount of technical work needed to implement this idea in a rigorous way is rather substantial.  For the details, see Sections~\ref{S:LPD}~and~\ref{S:NPD}.

The rest of this paper is organized as follows:
\begin{itemize}
\item In Section~\ref{S:notation}, we collect some basic notation and discuss the assumptions needed on the external potential in more detail.
\item In Section~\ref{S:solitons}, we discuss the existence and properties of ground state solitons and excited solitons.  We note that proving the existence of the excited solitons in the non-radial setting requires some new arguments.
\item In Section~\ref{S:dynamics}, we first discuss global existence for \eqref{e:nls}.  We then present the decomposition of small-mass solutions around the ground state solitons described above, deriving the coupled ODE-PDE system for $(z(t),\xi(t))$.  We then turn to the scattering result in Theorem~\ref{T}.  Taking for granted a key convergence result (Proposition~\ref{p:key}) that implies the existence of the compact non-scattering solution (Proposition~\ref{p:criticalelement}), we carry out the localized virial argument to complete the proof of scattering in Theorem~\ref{T}.  Finally, we prove the grow-up result in Theorem~\ref{T}.  The remainder of the paper is then dedicated to establishing Proposition~\ref{p:key}, as follows.
\item In Section~\ref{S:linear}, we discuss the linearized equation.  In particular, we collect some Strichartz and dispersive estimates for this equation, and we also establish some small-data scattering and stability results. 
\item In Section~\ref{S:LPD}, we prove the linear profile decomposition adapted to the linearized equation.
\item Finally, in Section~\ref{S:NPD}, we establish the key convergence result, Proposition~\ref{p:key}, by means of the nonlinear profile decomposition described above. 
\end{itemize}

\subsection*{Acknowledgements} S.M. was supported by JSPS KAKENHI Grant Numbers JP17K14219, JP17H02854,  JP18KK0386, JP21H00991, and JP21H00993.
J. M. was supported by a Simons Collaboration Grant. 
J.S was supported by JSPS KAKENHI Grant Numbers JP17H02851, JP19H05597, 
JP20H00118, JP21H00993, and  JP21K18588.

\subsection{Notation and preliminaries}\label{S:notation}
We write $A\lesssim B$ to denote the inequality $A\leq CB$ for some $C>0$. We write
\[
(f,g) = \int \bar f(x) g(x)\,dx
\]
for the standard $L^2$ inner product.  We define
\[
\|f\|_{H^s_r} = \| \langle \nabla \rangle^s f\|_{L^r}. 
\]
Space-time norms are denoted by $\|\cdot\|_{L_t^q L_x^r}$ or $\|\cdot\|_{L_t^q H^s_r}$. 

We write $Q$ for the cubic NLS ground state, that is, the unique nonnegative, radial, decaying solution to $-Q+\Delta Q + Q^3 =0$ (see e.g. \cite[Appendix~B]{Tao}). 

We also define here several important functionals, including the virial functional $\mathbb{K}_{V,2}$.  Recall that $\mathbb{H}_V$ and $\mathbb{G}$ were defined in \eqref{HVandG}. 

\begin{definition}\label{def:KI}
For any $\varphi \in H^1$, we define
\begin{align*}
\mathbb{K}_{V,2}(\varphi) &:=   2 \mathbb{H}_0(\varphi) - 3 \mathbb{G}(\varphi)- \int_{\R^3}  \tfrac12 (x\cdot \nabla V(x)) |\varphi(x)|^2 dx,\\	
\mathbb{I}_{V}(\varphi) &:= \mathbb{E}_V(\varphi) - \tfrac12 \mathbb{K}_{V,2}(\varphi) \\
&= \tfrac12 \mathbb{G}(\varphi) + \tfrac14 \int (x\cdot \nabla V + 2 V)( x)|\varphi(x)|^2  dx.
\end{align*}
\end{definition}

\subsection{Assumptions on the external potential} \label{S:potential}

For the potential $V$, we make the following assumptions: Let $L^\I_0$ be the completion of $C_0^\I$ with respect to the supremum norm. We assume the following: 
\begin{itemize}
\item[(A1)] $\lim_{|x|\to\I}|V(x)|=0$;
\item[(A2)] $V,$ $x\cdot \nabla V,$ $|x|^2 |\nabla^2 V| \in (L^2+L_0^\I)(\R^3)$. 
\item[(A3)] $-\Delta + V $ on $L^2 (\R^3)$ has only one negative eigenvalue $e_0<0$, which is simple;
\item[(A4)] The wave operator $W= \lim_{t\to\I} e^{itH} e^{it\Delta}$ and its adjoint $W^*$ are bounded on the Sobolev space$H^k_p(\R^3)$
for $k=1,2$ and some $p>8$\footnote{The condition (A4) is satisfied if $H$ admits no eigenfunction or resonance at zero energy, $V\in L^{p}$ and $\widehat{V}\in \dot{B}^{1/2}_{2,1}$ (cf. \cite{Beceanu}).}.
\end{itemize}
These are essentially the same assumptions as in \cite{Nakanishi}. Note that we do not assume radial symmetry on $V$.  Note also that in \cite{Nakanishi}, the additional assumption $V/|x| \in L^1(\R^3)$ is added to (A2), which is used in the proof of Conjecture~\ref{conjecture} in the radial setting.

In addition, we suppose 
\begin{itemize}
\item[(A5)] $V \in H^{1}_{3/2};$
\item[(A6)]
\begin{itemize}
\item $V$, $x\cdot \nabla V \in L^\I$;
\item $V$ and $x\cdot \nabla V$ are continuous at the origin;
\item $V(0)=\min_{x\in\R^3}V(x)<0.$
\end{itemize}
\end{itemize}

\begin{remark} Assumption (A6) warrants some discussion.  This assumption is used in the proof that $\mathcal{E}_1$ obeys the bound \eqref{e:me_comparison} and that there exists an excited state for small mass.  The assumption $V(0)=\min_{x\in\R^3}V(x)$  is rather a characterization of the origin.  If a potential $V$ satisfies $V(x_0)=\min_{x\in\R^3}V(x)$ for some $x_0 \in \R^3$, and if $V$ and $(x-x_0)\cdot \nabla V$ are continuous at $x=x_0$, then an application of a suitable space translation to \eqref{e:nls} enables us to assume $x_0=0$.  Note that the validity of all other assumptions holds under space translation.  Thus, we are essentially supposing that the existence of such a point $x_0 \in \R^3$.  We also remark that the inequality $\min_{x\in\R^3}V(x)  <0$
follows from (A3). Indeed, if $V$ does not have a negative part, we see that $H\ge0$ and there exists no negative eigenvalue.
\end{remark}

\section{Ground states and excited states}\label{S:solitons}

In this section, we discuss the existence of soliton solutions to \eqref{e:nls}.  We denote by $\mathscr{S}$ the set of functions $\varphi\in H^1$ satisfying
\begin{equation}\label{e:solitoneq}
(H+ \omega) \varphi = \sigma |\varphi|^2 \varphi\qtq{for some}\omega\in\R.
\end{equation}
In particular, for $\varphi\in\mathscr{S}$, we have that $u(t,x)=e^{-i\omega t}\varphi$ is a soliton solution to \eqref{e:nls}.  As we will see, solitons arise as energy minimizers.  In particular, given a fixed mass $\mu>0$, we first define 
\begin{align*}
\mathscr{E}_0(\mu) &:= \inf\{ \mathbb{E}_V (\varphi) 
\ |\  \varphi \in \mathscr{S},\ \mathbb{M} (\varphi)= \mu \}, \\
\mathscr{S}_0 (\mu) &:= \{ \varphi \in H^1(\R^3) \ |\  \varphi \in \mathscr{S},\, \mathbb{M}(\varphi)=\mu,\, \mathbb{E}_V (\varphi) = \mathscr{E}_0(\mu) \}.
\end{align*}
This minimization problem will yield the ground state solitons.  We also define 
\begin{align*}
\mathscr{E}_1(\mu) &:= \inf\{ \mathbb{E}_V (\varphi) \ |\  \varphi \in \mathscr{S},\, \mathbb{M} (\varphi)= \mu,\, \mathbb{E}_V (\varphi) > \mathscr{E}_0(\mu) \}, \\
\mathscr{S}_1 (\mu) &:= \{ \varphi \in H^1(\R^3) \ |\  \varphi \in \mathscr{S},\, \mathbb{M}(\varphi)=\mu,\, \mathbb{E}_V (\varphi) = \mathscr{E}_1(\mu) \}.
\end{align*}
In the focusing case, this minimization problem will yield the excited solitons. 

%
%

The existence of small solitary waves for \eqref{e:nls} was addressed in \cite{SW,SW2,GNT}.  We import the results we need here. We write $\phi_0 $ for the normalized eigenfunction of $H$ associated with the unique negative eigenvalue $e_0$, i.e.,
\[
H \phi_0 = e_0 \phi_0, \quad \mathbb{M}(\phi_0)=1.
\]

\begin{lemma}\label{L:small-solitons}
There exists $z_0>0$ and a $C^1$-map 
\[
(\Phi, \Omega) : \{ z \in \C \ |\ |z|^2 < 2 z_0 \} \to H^1 \times \R,
\]
such that $(\Phi[z],\Omega[z])$ satisfies
\begin{equation}\label{e:Phizeq}
(H + \Omega[z]) \Phi[z]  = \sigma |\Phi[z] |^2 \Phi[z].
\end{equation}
The parameters satisfy the following:
\begin{itemize}
\item $(\Phi[z],\phi_0)=z$ (so that $\Phi[z]-z\phi_0\perp\phi_0)$,
\item $\Phi[e^{i\theta}z]=e^{i\theta}\Phi[z]$ for any $\theta \in \R$, and
\item if $z$ is real, then $\Phi[z]$ is real-valued. 
\end{itemize}
Moreover, the following hold as $|z|\to0$:
\[
\Phi[z] = z \phi_0 + o(|z|^2) \IN H^1,\quad \Omega[z] = -e_0 + o(|z|) \IN \C.
\]
Finally,
\[
\mathbb{M}(\Phi[z]) = |z|^2 + o(|z|^4)
\]
is an increasing bijective function of $|z|$.
\end{lemma}

The small solitons are energy minimizers for fixed small mass (see \cite{SW, SW2, GNT}): 

\begin{proposition} Let $z_0$ be as in Lemma~\ref{L:small-solitons}.  There exists $z_1 \le z_0$ and $\mu_1>0$ such that
\[
\{ \varphi \in \mathscr{S}_0 \ |\ \mathbb{M}(\varphi)<\mu_1 \} = \{ \Phi[z] \ |\ |z| <z_1 \}.
\]
Moreover, if $0<\mu<\mu_1$ then 
\[
\mathscr{E}_0(\mu)= \mathbb{E}_V (\Phi[\zeta_0(\mu)])\sim e_0 \mu <0,
\]
where $\zeta_0$ is an inverse function of 
$z \mapsto \mathbb{M}(\Phi[z]) \in (0,\mu_1).$
\end{proposition}

We will also need the following important `trichotomy' for functions in $H^1$ of small mass.  This result appears as \cite[Lemma~2.3]{Nakanishi}.

\begin{proposition}[Trichotomy]\label{p:smt}
There exist $C,\mu_3>0$ such that
if $\varphi \in H^1$ satisfies 
\[
\mathbb{M}(\varphi) \le \mu_3\qtq{and}\mathbb{K}_{V,2} (\varphi) \le C \mathbb{M}(\varphi)^{-1},
\]
then one of the following hold:
\begin{enumerate}
\item $\mathbb{H}_0(\varphi) \lesssim \mathbb{M}(\varphi)$;
\item $\mathbb{K}_{V,2}(\varphi) \gtrsim \mathbb{M}(\varphi)$ and $\mathbb{K}_{V,2}(\varphi) \sim \mathbb{E}_V (\varphi)\sim \mathbb{H}_0(\varphi)$;
\item $\sigma=1$ and $\mathbb{G}(\varphi) \gtrsim \mathbb{H}_0 (\varphi)  \gtrsim \mathbb{M}(\varphi)^{-1}$.
\end{enumerate}
\end{proposition}

We turn to the existence of excited states.  We consider first the focusing equation (i.e. $\sigma=1$) and work under the assumptions on $V$ detailed in Section~\ref{S:potential}.  The analogue of the following result was established in \cite[Proposition~2.5]{Nakanishi} in the radial setting.  In particular, the argument in \cite{Nakanishi} relies on the radial symmetry via the compactness of the embedding $H_{\text{rad}}^1\hookrightarrow L^4$.  To address the non-radial case, we instead utilize a profile decomposition

\begin{theorem}[Excited states in the focusing case]\label{P:excited}
There exist
$\mu_2 \in (0, \mu_1]$ and a constant $C \in (0,|V(0)|)$
such that for $0<\mu<\mu_2$, $\mathscr{E}_1(\mu)$ is continuous and decreasing in $\mu$, and
\begin{align}
\mathscr{E}_1(\mu) ={}& \inf \{ \mathbb{E}_V(\varphi) \ |\ \varphi \in H^1 ,\, \mathbb{M}(\varphi)=\mu, \, \mathbb{K}_{V,2}(\varphi)=0, \, \mathbb{G}(\varphi) \ge 1 \} \nonumber\\
={}& \inf \{ \mathbb{I}_V(\varphi) \ |\ \varphi \in H^1 ,\, \mathbb{M}(\varphi)\le\mu, \,\mathbb{K}_{V,2}(\varphi)\le0,\, \mathbb{G}(\varphi) \ge 1 \} \nonumber\\
\le{}& \mu^{-1} \mathbb{E}_0(Q) \mathbb{M}(Q)-C\mu, \label{e:e1bound}
\end{align}
with $\mathscr{S}_1(\mu)$ nonempty. Moreover,
for any $c>0$ there exists $\mu_2'\in (0, \mu_2]$
such that
\begin{equation}\label{e:e1boundsharp}
	\mu^{-1} (\mathbb{E}_0(Q) \mathbb{M}(Q)-c)<
	\mathscr{E}_1(\mu) < \mu^{-1} \mathbb{E}_0(Q) \mathbb{M}(Q)+(V(0) + c)\mu
\end{equation}
holds for $\mu\in (0,\mu_2']$.
Furthermore, there exists a continuous map 
\begin{equation}\label{e:variational_characterization}\kappa : (0,\mu_2) \times \R_+ \to (0,\tfrac12)\qtq{such that}
	|\mathbb{K}_{V,2}(\varphi)| \ge \kappa (\mathbb{M}(\varphi), \delta )
\end{equation}
for any $\varphi \in H^1$ satisfying 
\[
\mathbb{M}(\varphi)<\mu_2,\quad\mathbb{E}_V (\varphi) \le \mathscr{E}_1 (\mu) - \delta,\qtq{and}\norm{\nabla \varphi}_{L^2} \ge 1.
\]
\end{theorem}  
\begin{remark}
In Nakanishi \cite{Nakanishi}, the radial version of $\mathcal{E}_1(\mu)$, which we denote $\mathcal{E}_{1,\mathrm{rad}}(\mu)$, is considered.  When $V$ is radial, both minimization quantities make sense, and by definition, we have $\mathcal{E}_{1}(\mu)\le \mathcal{E}_{1,\mathrm{rad}}(\mu)$.
For $\mathcal{E}_1(\mu)$, the estimate \eqref{e:e1boundsharp} holds.  On the other hand, we do not have a similar upper bound for $\mathcal{E}_{1,\mathrm{rad}}(\mu)$; Instead, we merely have $\lim_{\mu\to 0} \mu\mathcal{E}_{1,\mathrm{rad}}(\mu)=\mathbb{E}_0(Q) \mathbb{M}(Q)$.
\end{remark}

 To prove the theorem, we introduce the quantity
\begin{equation}\label{E1tildedef}
\tilde{\mathscr{E}}_1(\mu) := \inf \{ \mathbb{E}_V(\varphi) \ |\ \varphi \in H^1 ,\, \mathbb{M}(\varphi)=\mu, \, \mathbb{K}_{V,2}(\varphi)=0, \, \mathbb{G}(\varphi) \ge 1 \} .
\end{equation}
In fact, we will ultimately prove that $\tilde{\mathscr{E}}_1(\mu)={\mathscr{E}}_1(\mu)$.

\begin{proof}
The identity
\begin{equation}\label{E1tilde}
\tilde{\mathscr{E}}_1(\mu) = \inf \{ \mathbb{I}_V(\varphi) \ |\ \varphi \in H^1 ,\, \mathbb{M}(\varphi)\le\mu, \,\mathbb{K}_{V,2}(\varphi)\le0,\, \mathbb{G}(\varphi) \ge 1 \} 
\end{equation}
can be established by the same argument as in \cite{Nakanishi}.  Note that the set in question is nonempty, so that $\tilde{\mathscr{E}}_1(\mu)$ is finite.  It also follows that $\mu\mapsto\tilde{\mathscr{E}}_1(\mu)$ is decreasing.

The rest of the proof is divided into six steps.  The first four steps are devoted to the proof of the existence of a minimizer to ${\mathscr{E}}_1(\mu)$.  We also establish the existence of the map $\kappa$. Most of the proof of these two statements may be done in parallel.  To give a unified treatment, we introduce a parameter $E_\I \in (0, \tilde{\mathscr{E}_1} (\mu)]$.  In what follows, we will need the following inequality
for some $C>0$ and $\mu_*>0$: 
\begin{equation}\label{3.10.5}
E_\I	 \le{} \tilde{\mathscr{E}}_1(\mu) \le  {\tfrac1{\mu} \mathbb{M}(Q)  \mathbb{I}_0 (Q) - C \mu}
\end{equation}
for $\mu \in (0,\mu_*]$.
We will prove the inequality below (see Proposition~\ref{p:e1comparison1} and observe that $\mathbb{I}_0(Q)=\mathbb{E}_0(Q)$).
\medskip

\noindent
\underline{\bf Step 1}. 
Let $\{ \varphi_n \}_{n\in \N} \subset H^1$ be a sequence such that
\[
\mathbb{E}_V(\varphi_n) \to E_\I,\quad
\mathbb{M}(\varphi_n) \to \mu, \quad
\mathbb{K}_{V,2}(\varphi_n)\to 0\qtq{as}n\to\infty,
\]
with $\mathbb{H}_0(\varphi_n) \ge 1$ for all $n$.
By Gagliardo--Nirenberg and Young's inequality, 
\begin{align*}
	\mathbb{H}_0 (\varphi_n) ={}& 3\mathbb{E}_V (\varphi_n) - \mathbb{K}_{V,2} (\varphi_n) - \tfrac12 \int (x\cdot\nabla V + 3V)|\varphi_n|^2 dx \\
	\le {}& 3\mathbb{E}_V (\varphi_n) - \mathbb{K}_{V,2} (\varphi_n) + C \mathbb{M} (\varphi_n) + \tfrac12 \mathbb{H}_0 (\varphi_n),
\end{align*}
so that $\mathbb{H}_0 (\varphi_n)  \le 6 E_\I + 2C \mu + 1$ for large $n$; in particular, $\{\varphi_n\}$ is bounded in $H^1$. 

Now, as $\mathbb{K}_{V,2}(\varphi_n)\to 0$ as $n\to\I$, we deduce from $\mathbb{H}_0(\varphi_n) \ge 1$ and the trichotomy of Proposition~\ref{p:smt} that
\[
	\mu^{-1} \lesssim \mathbb{H}_0(\varphi_n) \lesssim  \mathbb{G}(\varphi_n)
\]
for large $n$. Furthermore, using the fact that $\mathbb{K}_{V,2} (\varphi_n) \to 0$ as $n\to\I$, we may also deduce that that $\mathbb{H}_0(\varphi_n) \sim \mathbb{G}(\varphi_n)$ by absorbing the potential term as above.  In particular, we obtain boundedness of $\mathbb{G}(\varphi_n)$. Passing to a subsequence if necessary, we may therefore assume that $\mathbb{H}_0(\varphi_n)$ and $\mathbb{G}(\varphi_n)$ converge as $n\to\I$ to some limits ${H}_\I$ and $G_\I$, which then satisfy $G_\I \sim H_\I \gtrsim \mu^{-1}$.

We now apply a profile decomposition to $\{\varphi_n\}$ with remainder vanishing in $L^4$ (see e.g. \cite{Hmidi}).  This takes the form
\[
\varphi_n = u_0 + \sum_{j=1}^J T_{y_n^j} u_j + r_n^J,
\]
where $(T_y f)(x) := f(x-y)$ denotes translation.  In particular,
\begin{align}
\lim_{n\to\infty} |y_n^j-y_n^k|\to \I\qtq{for any}j\neq k, &\label{jorthogonal} \\
T_{y_n^j}^{-1} r_n^J \rightharpoonup 0\qtq{weakly in }H^1\qtq{for each}j\leq J&, \nonumber \\
 \lim_{J\to\I} \varlimsup_{n\to\I} \norm{r_n^J}_{L^4} =0.& \label{e:L4error}
\end{align}

Furthermore, we have
\begin{equation}\label{e:excitemassdecomp}
\mu \ge \mathbb{M} (u_0) + \sum_{j=1}^\I \mathbb{M}(u_j), 
\end{equation}
\begin{equation}\label{e:exciteH0decomp}
H_\I \ge \mathbb{H}_0 (u_0) + \sum_{j=1}^\I \mathbb{H}_0 (u_j),
\end{equation}
and (making use of \eqref{e:L4error}) 
\begin{equation}\label{e:Gdecomp}
G_\I = \mathbb{G} (u_0) + \sum_{j=1}^\I \mathbb{G} (u_j).
\end{equation}
Note also that for any $f \in L^2 + L^\I$, we have
\begin{equation}\label{e:potentialdecomp}
	\lim_{n\to\I} \int f |\varphi_n|^2\,dx= \int f |u_0|^2\,dx.
\end{equation}

Using \eqref{e:potentialdecomp}, \eqref{e:exciteH0decomp}, and \eqref{e:Gdecomp}, 
 we obtain
\begin{equation}\label{e:K2decomp}
\begin{aligned}
	0={}& \lim_{n\to\infty} \mathbb{K}_{V,2} (\varphi_n)\\
={}&
2H_\infty -3G_\infty - \int_{\R^3}  \tfrac12 (x\cdot \nabla V(x)) |u_0(x)|^2 dx \\
\ge{}& \mathbb{K}_{V,2} (u_0) + \sum_{j=1}^\I  \mathbb{K}_{0,2} (u_j),
\end{aligned}
\end{equation}
and (recalling the definition of $\mathbb{I}_V$ from Definition~\ref{def:KI})
\begin{equation}\label{e:IVdecomp}
	E_\I = \lim_{n\to\I} \mathbb{I}_V (\varphi_n) = \mathbb{I}_V (u_0) + \sum_{j=1}^\I \mathbb{I}_0 (u_j).
\end{equation}
\medskip

\noindent
\underline{\bf Step 2}.
In this step, we prove
\begin{equation}\label{e:exciteclaim1}
\mathbb{K}_{V_,2}(u_0) \le 0.
\end{equation}
In view of \eqref{e:K2decomp}, it is a consequence of
\begin{equation}\label{e:exciteclaim1.5}
\mathbb{K}_{0,2} (u_j) \ge 0, \quad \forall j \ge 1.
\end{equation}
Let us prove \eqref{e:exciteclaim1.5}.
We fix $j_0 \ge 1$ and first observe that if $u_{j_0}=0$ then $\mathbb{K}_{0,2} (u_{j_0}) = 0$, so we suppose instead that $u_{j_0}\neq 0$.

Using 
\begin{equation}\label{e:e1existencepf1}
\mathbb{I}_V (u_0) \ge - \tfrac12\norm{(x\cdot \nabla V + 2 V)_-}_{L^\I} \mathbb{M}(u_0),
\end{equation}
we first observe that
\begin{equation}\label{e:Iujbound}
	 \mathbb{I}_0(u_{j_0}) < \tfrac1{\mu} \mathbb{M}(Q)  \mathbb{I}_0 (Q) - C \mu + \tfrac12\norm{(x\cdot \nabla V + 2 V)_-}_{L^\I} \mathbb{M}(u_0).
\end{equation}
As $\mathbb{M}(u_{j_0}) \le \mu - \mathbb{M}(u_0)$, we obtain
\begin{align*}
\mathbb{M}(u_{j_0}) \mathbb{I}_0(u_{j_0})<{}& \tfrac{ \mu - \mathbb{M}(u_0)}{\mu } \mathbb{M}(Q)  \mathbb{I}_0 (Q)  - C \mu \mathbb{M}(u_{j_0})\\
 &+ \tfrac12\norm{(x\cdot \nabla V + 2 V)_-}_{L^\I}(\mu - \mathbb{M}(u_0)) \mathbb{M}(u_0).
\end{align*}
Now observe that if $\mu \lesssim_{Q,V} 1$ then the right-hand side is decreasing in $\mathbb{M}(u_0) \in [0, \mu]$. Thus we obtain the bound
\begin{equation}\label{e:ujbound}
\mathbb{M}(u_{j_0}) \mathbb{I}_0(u_{j_0}) < \mathbb{M}(Q)  \mathbb{I}_0 (Q)  - C \mu \mathbb{M}(u_{j_0}).
\end{equation}
Recalling that 
\begin{equation}\label{e:QcharacteriztionI}
\mathbb{M}(Q)  \mathbb{I}_0 (Q) = \inf \{ \mathbb{M}(\varphi)  \mathbb{I}_0 (\varphi) \ |\ \mathbb{K}_{0,2} (\varphi) \le 0, \,\varphi \neq0\},
\end{equation}
we see from \eqref{e:ujbound} that $\mathbb{K}_{0,2}(u_{j_0}) >0$.   Thus $\mathbb{K}_{0,2} (u_j) \ge 0$ for all $j \ge 1$, as desired.
\medskip

\noindent
\underline{\bf Step 3}.
We next claim
\begin{equation}\label{e:exciteclaim2}
\mathbb{G} (u_0)\geq 1.
\end{equation}
Notice that \eqref{e:exciteclaim1} and Proposition~\ref{p:smt} imply one of the following two alternatives:
\begin{equation}\label{e:e1existence_falsecase}
\mathbb{H}_0(u_0) \lesssim \mathbb{M}(u_0) \le \mu,
\end{equation}
\begin{equation}\label{e:e1existence_truecase}
\mathbb{G}(u_0) \gtrsim \mathbb{H}_0 (u_0)  \gtrsim \mathbb{M}(u_0)^{-1} \ge \mu^{-1}.
\end{equation}
To obtain the desired conclusion, it suffices to prove that \eqref{e:e1existence_truecase} holds.

Now suppose towards a contradiction that \eqref{e:e1existence_falsecase} occurs.  In this case, the Gagliardo-Nirenberg inequality implies 
\begin{equation}\label{GNcase111}
\mathbb{G}(u_0) \lesssim \mathbb{M}(u_0)^2 
\end{equation}
which yields
\begin{equation}\label{e:e1existencepf2}
\mathbb{K}_{V,2} (u_0) \ge  - C \mathbb{M}(u_0)^2 - \norm{(x\cdot \nabla V)_+}_{L^\I} \mathbb{M}(u_0)\ge - C'\mathbb{M}(u_0).
\end{equation}

We now claim that there exists unique $j_1\ge1$ such that
\begin{equation}\label{e:e1existencepf2.5}
\mathbb{M}(u_{j_1})  \mathbb{I}_0 (u_{j_1}) \ge (1-\mu) \mathbb{M}(Q)  \mathbb{I}_0 (Q) .
\end{equation}
The uniqueness of such a number is proved as follows: by Cauchy--Schwarz, embedding of $\ell^p$ spaces, the mass decomposition, \eqref{e:IVdecomp}, and \eqref{3.10.5}, we have 
\begin{align*}
\sum_{j=1}^\I \mathbb{M}(u_j)  \mathbb{I}_0 (u_j) &\le \sum_{j=1}^\infty \mathbb{M}(u_j)\cdot\sum_{j=1}^\infty \mathbb{I}_0 (u_j) \\
&\le [\mu-\mathbb{M}(u_0)][E_\infty-\mathbb{I}_V(u_0)] \\
 &  \le [\mu-\mathbb{M}(u_0)][\tfrac{1}{\mu}\mathbb{M}(Q)\mathbb{I}_0(Q)-C\mu +C_V \mathbb{M}(u_0)]\\
&\le \mathbb{M}(Q)  \mathbb{I}_0 (Q) - C \mu^2+ f_\mu(\mathbb{M}(u_0)),
\end{align*}
where 
\[
f_\mu(x):=\mu C_Vx - \tfrac{1}{\mu}\mathbb{M}(Q)\mathbb{I}_0(Q)x + C\mu x - C_V x^2
\]
is negative for sufficiently small $\mu$ and $x\in[0,\mu]$.  Hence, \eqref{e:e1existencepf2.5} does not hold for more than one $j_1$.

For the existence of such $j_1$, we suppose instead that 
\begin{equation}\label{e:e1existencepf3}
	\mathbb{M}(u_{j})  \mathbb{I}_0 (u_{j}) < (1-\mu) \mathbb{M}(Q)  \mathbb{I}_0 (Q) 
\end{equation}
for all $j \ge 1$.  Then, by the sharp Gagliardo--Nirenberg inequality,
\begin{align*}
\mathbb{K}_{0,2} (u_j) &\ge 2 \mathbb{H}_0 (u_j) - 3 \mathbb{G} (u_j)^\frac13 \( \frac{\mathbb{G}(Q)}{\mathbb{M}(Q)^{1/2}\mathbb{H}_0(Q)^{3/2}} \mathbb{M}(u_j)^{\frac12}\mathbb{H}_0(u_j)^{\frac32} \)^\frac23 \\
&> 2 (1-(1-\mu)^\frac13) \mathbb{H}_0 (u_j)> \tfrac\mu3 \mathbb{H}_0 (u_j)
\end{align*}
for all $j\ge 1$, where we have used  $\mathbb{I}_0 = \frac12 \mathbb{G}$, $\mathbb{H}_0(Q) = \tfrac32 \mathbb{G}(Q)$, and
\eqref{e:e1existencepf3}.  Thus, observing that
\[
\mathbb{H}_0 (u_j)\ge \tfrac32 \mathbb{G}(u_j) \Leftrightarrow \mathbb{K}_{0,2}(u_j) \ge 0,
\]
we see from \eqref{e:K2decomp}, \eqref{e:exciteclaim1.5}, \eqref{e:Gdecomp}, \eqref{GNcase111}, 
\eqref{e:excitemassdecomp}, and the bound $G_\infty \gtrsim \mu^{-1}$ that
\begin{align*}
- \mathbb{K}_{V,2} (u_0) \ge{}& \sum_{j=1}^\I \mathbb{K}_{0,2} (u_j) > \tfrac\mu{3} \sum_{j=1}^\I \mathbb{H}_0 (u_j) \ge \tfrac\mu{2} \sum_{j=1}^\I \mathbb{G} (u_j) \\	
={}& \tfrac\mu{2}(G_\I -  \mathbb{G}(u_0)) > C(1 - \mu^3),
\end{align*}
which contradicts \eqref{e:e1existencepf2} if $\mu$ is small.  Thus we obtain $j_1$ such that \eqref{e:e1existencepf2.5} holds, and by relabeling indices we may assume $j_1=1$.  In particular, 
\begin{equation}\label{e:e1existencepf4}
	 \mathbb{I}_0 (u_{1}) \ge \tfrac1{2\mu} \mathbb{M}(Q)  \mathbb{I}_0 (Q).
\end{equation}

In view of \eqref{e:Iujbound} and \eqref{e:e1existencepf2.5},
\[
\mu - \mathbb{M}(u_0) \ge \mathbb{M}(u_1)= \tfrac{ \mathbb{M}(u_1)  \mathbb{I}_0 (u_{1})}{ \mathbb{I}_0 (u_{1})} \ge \mu(1-\mu) \tfrac{ \mathbb{M}(Q)  \mathbb{I}_0 (Q) }{ \mathbb{M}(Q)  \mathbb{I}_0 (Q) - C_1 \mu^2 + C_2 \mu \mathbb{M}(u_0) } ,
\]
where $C_1$ and $C_2$ are constants depending only on $V$.  Regarding this as a quadratic inequality with respect to $\mathbb{M}(u_0)$ and solving the inequality under the constraint $\mathbb{M}(u_0) \le \mu$, we obtain that in fact, 
\[
\mathbb{M}(u_0) \lesssim \mu^2. 
\]
Using the assumption \eqref{e:e1existence_falsecase} and Gagliardo--Nirenberg, we also see that $\mathbb{H}_0(u_0) \lesssim \mu^2$ and $\mathbb{G}(u_0) \lesssim \mu^4$.
In particular, together with \eqref{e:K2decomp} and 
\eqref{e:exciteclaim1.5},
\begin{equation}\label{e:e1existencepf5}
	0 \le \mathbb{K}_{0,2} (u_1) \le - \mathbb{K}_{V,2} (u_0) \lesssim \mathbb{G}(u_0) + \mathbb{M}(u_0) \lesssim \mu^2.
\end{equation}

Now, let 
\[
u_1^\tau=e^{\frac{3\tau}{2}}u_1(e^\tau x),
\]
so that
\[
	\mathbb{K}_{0,2} (u_1^\tau) = 2e^{2\tau} \mathbb{H}_0 (u_1) - 3e^{3\tau} \mathbb{G} (u_1)
	= 2e^{2\tau} \mathbb{K}_{0,2} (u_1) - 6e^{2\tau}(e^\tau-1) \mathbb{I}_0 (u_1).
\]
By \eqref{e:e1existencepf4} and \eqref{e:e1existencepf5}, there exists $0< \tau_0 \lesssim \mu^3$ such that
$\mathbb{K}_{0,2} (u_1^{\tau_0})=0$. Then $\mathbb{M}(u_1^{\tau_0})=\mathbb{M}(u_1) \le \mu$ and
\begin{align*}
\mathbb{I}_0 (u_1^{\tau_0}) = e^{3\tau_0} \mathbb{I}_0 (u_1) \le{}& (1+ O(\mu^3)) \( \tfrac1{\mu} \mathbb{M}(Q)  \mathbb{I}_0 (Q) - C_1' \mu   \)\\
\le {}& \tfrac1{\mu} \mathbb{M}(Q)  \mathbb{I}_0 (Q) - \tfrac12C_1' \mu
\end{align*}
if $\mu$ is small. Thus, $u_1^{\tau_0}$ is a nonzero function satisfying
\[
\mathbb{M}(u_1(\tau_0))\mathbb{I}_0 (u_1(\tau_0)) < \mathbb{M}(Q)  \mathbb{I}_0 (Q)\qtq{and} \mathbb{K}_{0,2} (u_1(\tau_0))=0,
\]
contradicting the characterization \eqref{e:QcharacteriztionI} of $Q$.  Thus we conclude that \eqref{e:e1existence_falsecase} fails, so that \eqref{e:e1existence_truecase} holds. 
\medskip

\noindent
\underline{\bf Step 4}.
Now, we see from \eqref{e:excitemassdecomp},
\eqref{e:exciteclaim1}, and \eqref{e:exciteclaim2} that $u_0$ satisfies 
\[
\mathbb{M}(u_0) \le \mu,\quad\mathbb{K}_{V_,2}(u_0) \le 0,\qtq{and} \mathbb{G} (u_0)\geq 1,
\]
One then deduces from the characterization \eqref{E1tilde} of $\tilde{\mathscr{E}}_1(\mu)$ that
\[
\mathbb{I}_V (u_0) \ge \tilde{\mathscr{E}}_1 (\mu).
\]
Furthermore, by \eqref{e:IVdecomp},
\[
\tilde{\mathscr{E}}_1 (\mu) \le \mathbb{I}_V (u_0) \le E_\I \le \tilde{\mathscr{E}}_1(\mu).
\]
Thus, we have $E_\I=\tilde{\mathscr{E}}_1 (\mu)$ and $u_0$ is a minimzer to $\tilde{\mathscr{E}}_1 (\mu)$.  In fact, this also shows that if $E_\I < \tilde{\mathscr{E}}_1 (\mu)$, then there is no sequence such that
\[
\mathbb{E}_V(\varphi_n) \to E_\I,\quad\mathbb{M}(\varphi_n) \to \mu, \qtq{and}\mathbb{K}_{V,2}(\varphi_n)\to 0 \qtq{as}n\to\infty
\]
and such that $\inf_n \mathbb{H}_0( \varphi_n)\ge 1 $. In particular, for small $\delta>0$, we have
\[
\inf \{ |\mathbb{K}_{V,2}(\varphi)| \ :\ \mathbb{E}_V(\varphi)< \tilde{\mathscr{E}}_1 (\mu) - \delta, \, \mathbb{M}(\varphi) \le \mu, \, \mathbb{H}_0( \varphi)\ge 1 \} \gtrsim_\delta 1.
\]
This will imply \eqref{e:variational_characterization} once we establish $\tilde{\mathscr{E}}_1 (\mu) = {\mathscr{E}}_1 (\mu)$.  In particular, we will complete the proof once we establish this identity and the existence of a minimizer to ${\mathscr{E}}_1 (\mu)$.

We first show $\tilde{\mathscr{E}}_1(\mu)\le\mathscr{E}_1(\mu)$.  We select $\varphi_0 \in \mathscr{S}$ such that $\mathbb{M}(\varphi_0)=\mu$ and $\mathbb{E}_V(\varphi_0) > \mathscr{E}_0(\mu)$.  If there is no such function then ${\mathscr{E}}_1 (\mu)=\I$ and the result follows.  Otherwise, $\varphi_0 \in \mathscr{S}$ implies that $\mathbb{K}_{V,2}(\varphi_0)=0$.
Then, by Proposition~\ref{p:smt}, we have either $\mathbb{H}_0 \lesssim \mathbb{M} \lesssim \mu$ or $\mu^{-1} \lesssim \mathbb{H}_0 \lesssim \mathbb{G} (\varphi_0)$. In the former case we see from Proposition \ref{p:coordinate} below that{\footnote{This requires $\mu_2 \le \mu_4$.}} $\mathbb{E}_V(\varphi_0)=\mathscr{E}_0(\mu)$, which is precluded.
Hence, we have the latter case. In particular, $\mathbb{G} (\varphi_0)\geq1$, so that
\[
\mathbb{E}_V(\varphi_0) \ge \tilde{\mathscr{E}}_1 (\mu).
\]
Taking the infimum over such $\varphi_0$, we obtain $\tilde{\mathscr{E}}_1 (\mu) \le {\mathscr{E}}_1 (\mu)$.

We next show $\tilde{\mathscr{E}}_1(\mu)\ge \mathscr{E}_1(\mu)$. We let $\varphi_1 $ be the minimizer to $\tilde{\mathscr{E}}_1 (\mu)$ constructed above.  Using a scaling argument, we may obtain that $\mathbb{M}(\varphi_1)=\mu$ and $\mathbb{K}_{V,2}(\varphi_1) =0$.  Furthermore, arguing as in \cite{Nakanishi}, we may obtain that $\varphi_1 \in \mathscr{S}$.
Noting that 
\[
\mathbb{E}_{V}(\varphi_1) \gtrsim \mu^{-1} \ge \mu \sim \mathscr{E}_0(\mu),
\]
we therefore obtain
\[
\mathbb{E}_{V} (\varphi_1) \ge {\mathscr{E}}_1 (\mu).
\]
Hence, the desired inequality follows from $\tilde{\mathscr{E}}_1(\mu)=\mathbb{E}_{V} (\varphi_1)$.
We may also observe that $\varphi_1$ is in fact a minimizer to $\mathscr{E}_1(\mu)$.
\medskip

\noindent
\underline{\bf Step 5}.
Let us next establish the bound \eqref{e:e1boundsharp}.
The upper bound follows from Proposition~\ref{p:e1comparison1} and $\mathcal{E}_1(\mu)=\tilde{\mathcal{E}}_1(\mu)$.
Hence, we consider the lower bound. 
It suffices to show $\mu \mathcal{E}_1(\mu)$ tends to $\mathbb{E}_0(Q) \mathbb{M}(Q)$ as $\mu\to0$.
To see this, we let $\varphi(\mu)$ be the minimizer for $\mathcal{E}_1(\mu)$ and set $Q_\mu(x):=\mu \varphi(\mu)(\mu x)$. 
A computation shows $\mathbb{M}(Q_\mu)=1$ and, thanks to \eqref{e:e1existence_truecase}, $\mathbb{G}(Q_\mu) \gtrsim 1$. Further,
\begin{align*}
	\mu \mathcal{E}_1(\mu)=\mu\mathbb{E}_V(\varphi(\mu))=
	 \mathbb{H}_0(Q_\mu)-\mathbb{G}(Q_\mu) + \tfrac{\mu^2}2 \int V(\mu x)
	 |Q_\mu|^2 dx
\end{align*}
and
\[
	0=\mu\mathbb{K}_{V,2}(\varphi(\mu))=
	 2\mathbb{H}_0(Q_\mu)-3\mathbb{G}(Q_\mu) + \tfrac{\mu^2}4 \int (x\cdot \nabla V +2V)(\mu x)
	 |Q_\mu|^2 dx
\]
follow. Combining these two identities and using the upper bound in \eqref{e:e1boundsharp}, $\mathbb{M}(Q_\mu)=1$, and the assumption $V, x\cdot \nabla V \in L^\infty$, one has
\[
	\mathbb{H}_0(Q_\mu) < 3 \mathbb{E}_0(Q) \mathbb{M}(Q)  + C \mu^2,
\]
where the constant $C$ depends only on $V$. Hence, $\{ Q_\mu \}_{0<\mu \le \mu_2}$ is a bounded sequence in $H^1$.
Further, 
\[
	1 \lesssim \varlimsup_{\mu \to 0} \mathbb{G}(Q_\mu) =
	\varlimsup_{\mu \to 0} \tfrac23 \mathbb{H}_0(Q_\mu) \le
	2 \mathbb{E}_0(Q) \mathbb{M}(Q)
\]

Now, pick a sequence $\mu_n \to 0$. Subtracting a subsequence if necessary,
we suppose that $\mathbb{H}_0(Q_{\mu_n})$ and $\mathbb{G}(Q_{\mu_n})$ converge.
We apply the profile decomposition to $\{ Q_{\mu_n} \}_n$ with remainder vanishing in $L^4$, as in Step 1. This now takes the form
\[
	Q_{\mu_n} = \sum_{j=0}^J T_{y_n^j} \tilde{Q}_j + \rho_n^J
\]
up to a subsequence. Note that we have the bounds
\begin{equation}\label{e:pfStep5_1}
	\sum_{j=0}^\infty \mathbb{M}(\tilde{Q}_j) \le 1, \quad
	\sum_{j=0}^\infty \mathbb{G}(\tilde{Q}_j) = \lim_{n\to\infty}
	\mathbb{G}(Q_{\mu_n}) \le	2 \mathbb{E}_0(Q) \mathbb{M}(Q).
\end{equation}
If $\tilde{Q}_j=0$ for all $j\ge0$ then we have $\sum_{j=0}^\infty \mathbb{G}(\tilde{Q}_j)=0$, a contradiction.
Hence,  $\tilde{Q}_j\neq0$ for some $j\ge0$.
Then, since
\[
	\sum_{j=0} \mathbb{K}_{0,2}(\tilde{Q}_j)
	\le 2 \lim_{n\to\infty} \mathbb{H}_0(Q_{\mu_n}) - 3
	 \lim_{n\to\infty} \mathbb{G}(Q_{\mu_n})=0,
\]
there exists $j_2\ge0$ such that $\tilde{Q}_{j_2}\neq0$ and
$\mathbb{K}_{0,2}(\tilde{Q}_{j_2})\le 0$.
Then, recalling $\mathbb{I}_0(Q)=\mathbb{E}_0(Q)=\mathbb{G}(Q)/2$, we see from \eqref{e:QcharacteriztionI} and \eqref{e:pfStep5_1} that
\[
	\mathbb{E}_0(Q) \mathbb{M}(Q) =\mathbb{I}_0(Q) \mathbb{M}(Q)\le \mathbb{I}_0(\tilde{Q}_{j_2}) \mathbb{M}(\tilde{Q}_{j_2}) = \tfrac12
	\mathbb{G}(\tilde{Q}_{j_2}) \mathbb{M}(\tilde{Q}_{j_2}) \le \mathbb{E}_0(Q) \mathbb{M}(Q).
\]
In particular, $\mathbb{G}(\tilde{Q}_{j_2}) \mathbb{M}(\tilde{Q}_{j_2})=2\mathbb{E}_0(Q) \mathbb{M}(Q)$ holds.
This also implies that $\mathbb{M}(\tilde{Q}_{j_2})=1$ in view of \eqref{e:pfStep5_1}.
Hence $\tilde{Q}_j=0$ for all $j\neq j_2$ and further $Q_{\mu_n}$ converges 
strongly to $\tilde{Q}_{j_2}$ in $H^1$.
Thus, 
\[
	\lim_{n\to\infty} \mu_n \mathcal{E}_1(\mu_n)=
	\tfrac12 \mathbb{G}(\tilde{Q}_{j_2}) = \mathbb{E}_0(Q) \mathbb{M}(Q).
\]
Since $\{\mu_n\}_n$ is arbitrary, we obtain the desired convergence.
\medskip

\noindent
\underline{\bf Step 6}.
Finally, let us prove continuity of $\mathscr{E}_1(\mu)$.
Let $\varphi(\mu) $ denote the minimizer for $\mathscr{E}_1(\mu)$.  We define 
\[
{\varphi}^\tau(x) = e^{6\tau/5} (\varphi(\mu))( e^{\tau} x).
\]
It follows that
\begin{align*}
&\tfrac{d}{d\tau}\mathbb{M} ({\varphi}^\tau)|_{\tau=0} = -\tfrac35 \mu, \\
&\tfrac{d}{d\tau}\mathbb{E}_V ({\varphi}^\tau)|_{\tau=0}  = \tfrac75 \mathbb{H}_0(\varphi(\mu)) - \tfrac95 \mathbb{G}(\varphi(\mu)) +O(1)= \tfrac15 \mathbb{H}_0(\varphi(\mu)) + O(1), \\
&\tfrac{d}{d\tau}\mathbb{K}_{V,2} ({\varphi}^\tau)|_{\tau=0}  = \tfrac{14}5 \mathbb{H}_0(\varphi(\mu)) - \tfrac{27}5 \mathbb{G}(\varphi(\mu)) +O(1)= -\tfrac45 \mathbb{H}_0(\varphi(\mu)) + O(1). 
\end{align*}
This shows that $\mathbb{E}_V ({\varphi}^\tau) \le \mathscr{E}_1(\mu) + O(\tau\mu^{-1})$ for small $\tau>0$. Thus, for any $\eps>0$ there exists $\delta_0=O(\mu^2 \eps)$ so that if $\mu' \in (\mu-\delta_0,\mu)$ then
\[
\mathscr{E}_1(\mu) \le \mathscr{E}_1(\mu') \le \mathscr{E}_1(\mu) + \eps.
\]
This implies left continuity of $\mathscr{E}_1(\mu)$.  The right continuity is shown similarly. \end{proof}

We turn now to the estimate on $\tilde{\mathscr{E}}_1 (\mu) $ that was needed in the proof above. 

\begin{proposition}\label{p:e1comparison1}  For any $\eps>0$ there exists $\mu_2' >0$ such that
\[
\tilde{\mathscr{E}}_1 (\mu) \le  \tfrac1{\mu} \mathbb{E}_0 (Q) \mathbb{M}(Q)  + (V(0)+\eps) \mu \qtq{for all}0<\mu<\mu_2',
\]
where $\tilde{\mathscr{E}}_1 (\mu)$ is defined in \eqref{E1tildedef}.
\end{proposition}

\begin{proof}
Recalling \eqref{E1tilde} and using the rescaling $\psi (x) = \mu \varphi(\mu x)$, we may write
\begin{equation}\label{e:e1comparisonpf1}
	\mu \tilde{\mathscr{E}}_1 (\mu) = \inf_{\psi \in A_\mu} \biggl[ \tfrac12 \mathbb{G}(\psi) + \tfrac14 \int \mu^2(x\cdot \nabla V + 2 V)(\mu x)|\psi(x)|^2  dx\biggr],
\end{equation}
where
\[
A_\mu = \bigl\{ \psi \in H^1 \ |\ \mathbb{M}(\psi)\le1 ,\, \mathbb{K}_{0,2}(\psi)\le \tfrac12 \int \mu^2(x\cdot \nabla V )(\mu x)|\psi(x)|^2  dx ,\, \mathbb{G}(\psi)\ge\mu \bigr\}.
\]

Now set $\psi_1^a = e^{a} Q(e^a x)$ for $a\in \R$, where $Q$ is the NLS ground state as above.   We recall that
\[
\mathbb{K}_{0,2} (Q) =0, \quad \mathbb{G}(Q) =  2\mathbb{E}_{0}(Q), 
\]
and thus
\begin{align*}
&\mathbb{M} (\psi_1^a) = e^{-a} \mathbb{M} (Q), \\
&\tfrac12\mathbb{G}(\psi_1^a) = \tfrac12 e^{a}	\mathbb{G}(Q)=  e^{a}\mathbb{E}_0(Q), \\
&\mathbb{K}_{0,2} (\psi_1^a) =e^{a}\mathbb{K}_{0,2} (Q)=0.
\end{align*}
We then let $\tilde{\psi}_1:= \psi_1^{a_0}$ with $a_0 = \log \mathbb{M}(Q)$. Thus
\begin{align*}
&\mathbb{M} (\tilde{\psi}_1) = 1, \\
&\tfrac12\mathbb{G}(\tilde{\psi}_1) = \mathbb{E}_0(Q)\mathbb{M} (Q), \\
&\mathbb{K}_{0,2} (\tilde{\psi}_1) = 2\mathbb{H}_0(\tilde{\psi}_1) - 3\mathbb{G}(\tilde{\psi}_1)  =0.
\end{align*}
Finally, set $\psi_2^\tau = e^{3\tau/2} \tilde{\psi}_1(e^\tau x)$ for $\tau\in \R$. Then
\begin{align}
&\mathbb{M} (\psi_2^\tau) =\mathbb{M} (\tilde{\psi}_1) = 1, \label{e:e1comparisonpf2}\\
&\tfrac12\mathbb{G}(\psi_2^\tau) =\tfrac12 e^{3\tau} \mathbb{G}(\tilde{\psi}_1)=  e^{3\tau} \mathbb{E}_0(Q)\mathbb{M} (Q), \label{e:e1comparisonpf22}\\
&\mathbb{K}_{0,2} (\psi_2^\tau) = 2e^{2\tau} \mathbb{H}_0(\tilde{\psi}_1) - 3e^{3\tau} \mathbb{G}(\tilde{\psi}_1) = 6e^{2\tau}(1-e^\tau) \mathbb{E}_0(Q)\mathbb{M} (Q).\label{e:e1comparisonpf3}
\end{align}

We will choose $\tau$ so that $\psi_2^\tau \in A_\mu$. To this end, we first observe that \eqref{e:e1comparisonpf2}-- \eqref{e:e1comparisonpf22} yield the mass constrant, while the constraint $\mathbb{G}(\psi_2^\tau) \ge \mu$ is also satisfied provided $|\tau| \lesssim_Q 1$  and $\mu \lesssim_Q 1 $.  To obtain the constraint on $\mathbb{K}_{0,2} (\psi_2^\tau)$, we claim that
\begin{equation}\label{e:e1comparisonpf4}
	\tfrac12 \int \mu^2(x\cdot \nabla V )(\mu x)|\psi_2^\tau(x)|^2  dx
	= o(\mu^2)
\end{equation}
as $\mu\to0$, locally uniformly in $\tau$. Indeed, since $x\cdot \nabla V $ is bounded and continuous at the origin, the dominated convergence theorem implies 
\[
	\lim_{\mu\to0} \tfrac12 \int (x\cdot \nabla V )(\mu x)|\psi_2^\tau(x)|^2  dx
	= \tfrac12 \int (x\cdot \nabla V )(0)|\psi_2^\tau(x)|^2  dx = 0
\]
locally uniformly in $\tau$, which yields \eqref{e:e1comparisonpf4}. Thus, for any $\eps>0$ there exists $\tilde{\mu}>0$ such that 
\[
	\abs{  \int (x\cdot \nabla V )(\mu x)|\psi_2^\tau(x)|^2} \le 12 \eps^2 \mathbb{E}_0(Q)\mathbb{M} (Q)\qtq{for}\mu \in (0,\tilde{\mu}). 
\]
The choice $\tau = \eps^2 \mu^2$ in \eqref{e:e1comparisonpf3} then yields
\[
	\mathbb{K}_{0,2} (\psi_2^\tau) \le - 6 \eps^2 \mu^2   \mathbb{E}^0(Q)\mathbb{M} (Q)
	\le - \mu^2\abs{ \tfrac12 \int (x\cdot \nabla V )(\mu x)|\psi_2^\tau(x)|^2  dx}\le0
\]
for $\mu \in (0,\tilde{\mu})$, which implies $\psi_2^\tau \in A_\mu$.

Next, arguing as in the proof of \eqref{e:e1comparisonpf4}, we observe that
\[
	\tfrac14 \int (x\cdot \nabla V + 2 V)(\mu x)|\psi_2^\tau(x)|^2\,dx \to V(0) \mathbb{M}(\psi_2^\tau)=V(0) <0
\]
as $\mu\to0$ by virtue of \eqref{e:e1comparisonpf2}.
We remark that this covergence also holds locally uniformly in $\tau$.
Thus, for any $\eps>0$ there exists $\tilde{\mu}' >0$ such that
\begin{equation}\label{e:e1comparisonpf5}
	\abs{ \tfrac14 \int (x\cdot \nabla V + 2 V)(\mu x)|\psi_2^\tau(x)|^2  dx	 -V(0) } <\frac{\eps}2
\end{equation}
if $\mu \in (0,\tilde{\mu}')$ and $|\tau|\le 1$.
We now take $\eps>0$ so small that
\[
\eps\le \frac{1}{6(e-1)\mathbb{E}_0(Q)\mathbb{M} (Q)}
\]
is satisfied
and fix the corresponding $\tilde{\mu}$ and $\tilde{\mu}'$.
Then, there exists $\mu_2' \in (0,\min (\tilde{\mu},\tilde{\mu}')]$ such that the choice $\tau = \eps^2 \mu^2$ yields $\tau \le 1$ and
\begin{equation}\label{e:e1comparisonpf6}
\begin{aligned}
		\tfrac12\mathbb{G}(\psi_2^\tau) ={}&
	e^{3\tau} \mathbb{E}_0(Q)\mathbb{M} (Q) \\
\le{}& (1+3(e-1)\tau) \mathbb{E}_0(Q)\mathbb{M} (Q)  
\\
= {}& \mathbb{E}_0(Q)\mathbb{M} (Q)  + (3(e-1)\eps \mathbb{E}_0(Q)\mathbb{M} (Q)) \eps \mu^2\\
\le {}& \mathbb{E}_0(Q)\mathbb{M} (Q)  + \tfrac{\eps}2  \mu^2
\end{aligned}
\end{equation}
for $\mu \in (0,\mu_2')$, where we have used \eqref{e:e1comparisonpf22}. 
Letting $\tau = \eps^2 \mu^2 $ and
plugging $\psi_2^\tau \in A_\mu$, \eqref{e:e1comparisonpf5}, and \eqref{e:e1comparisonpf6} 
to the formula \eqref{e:e1comparisonpf1}, one concludes that
\begin{align*}
\mu \tilde{\mathscr{E}}_1 (\mu) &{}\le \tfrac12\mathbb{G}(\psi_2^\tau) + \tfrac14 \int \mu^2(x\cdot \nabla V + 2 V)(\mu x)|\psi_2^\tau(x)|^2  dx \\
&{}\le  \mathbb{E}_0(Q)\mathbb{M} (Q) + \tfrac{ \eps}2 \mu^2 + \(V(0) +\tfrac{\eps}2\)\mu^2 \\
&{}= \mathbb{E}_0(Q)\mathbb{M} (Q) + ({V(0)}+\eps)\mu^2 
\end{align*}
for $\mu \in (0,\mu_2')$, which yields the desired result.  \end{proof}

We turn our attention now to the defocusing case, $\sigma=-1$, in which case we expect that the excited solitons are not present (see e.g. \cite[Proposition~2.6]{Nakanishi} for the radial case). In the present setting, we can immediately see that there are no solutions to \eqref{e:solitoneq} in the case $\omega\ge-e_0$.  Indeed, taking the $L^{2}$ inner product of (\ref{e:solitoneq}) with $\bar{\varphi}$, we have
\[
(H\varphi,\varphi)+\omega\|\varphi\|_{L^{2}}^{2}=-\|\varphi\|_{L^{4}}^{4}.
\]
Since, $H\ge e_{0}$, we obtain
\[
(e_{0}+\omega)\|\varphi\|_{L^{2}}^{2}\le-\|\varphi\|_{L^{4}}^{4},
\]
which implies the claim.  More generally, we conjecture the following:

\begin{conjecture}[Solitons in the defocusing case]\label{conjecture} Let $\sigma=-1$.  
The equation \eqref{e:solitoneq} has a unique positive solution $\varphi_\omega$ for $\omega \in (0,-e_0)$. The set $\mathscr{S}$ is characterized as
\[
	\mathscr{S} = \{ e^{i\theta} \varphi_\omega \ |\  \omega \in (0,-e_0), \, \theta \in \R \}.
\]
The map $(0,-e_0) \ni \omega \mapsto \mathbb{M} (\varphi_\omega) \in (0,\I)$ is monotone decreasing $C^1$ function. Denoting the inverse function by $\omega_0(\mu)$, we have $\omega_0'(\mu)<0$ and
\[
\mathscr{E}_0 (\mu) = \mathbb{E}_V (\varphi_{\omega_0(\mu)}) \in (e_0 \mu, 0)\qtq{for all}\mu>0.
\]
On the other hand, $\mathscr{E}_1(\mu)=\I$ for all $\mu>0$.
\end{conjecture}

In fact, most of Conjecture~\ref{conjecture} may be found in \cite{GNT}; the only claim that has not yet been established is that $\mathscr{E}_1(\mu)=\infty$ for all $\mu>0$ in the non-radial case.

\section{Global dynamics}\label{S:dynamics}

In this section, we prove the main result (Theorem~\ref{T}), taking for granted one key proposition (Proposition~\ref{p:key}).  The remaining sections will then be devoted to establishing Proposition~\ref{p:key}.

We begin with the following result concerning global existence for \eqref{e:nls}. 

\begin{proposition}[Global existence]\label{l:gwp}
Let $\mu_2$ be as in Proposition~\ref{P:excited}.  Suppose $u$ is a solution to \eqref{e:nls} such that
 \[
 \mathbb{M}(u) = \mu \le \mu_2\qtq{and}\mathbb{E}_V(u) < \mathscr{E}_1(\mathbb{M}(u)).
\]

If $\sigma=1$ and
\[
\|\nabla u(t)\|_{L^2}>1 \qtq{and} K_{V,2}(u(t))<0
\]
fails at some time $t\in\R$, then it fails for all $t\in\R$ and the solution $u$ is global.  Furthermore, the solution satisfies one of the following: 
\begin{itemize}
\item We have $\sup_{t\in\R}\norm{u(t)}_{H^1}^2 \lesssim \mu$ and scattering to $\mathscr{S}_0$ holds in both time directions.
\item For all $t\in\R$, we have  
\[
\mu \lesssim \norm{u(t)}_{H^1}^2 \lesssim \mathbb{E}_V(u) + \mu\qtq{and} \mathbb{K}_{V,2} (u(t)) \ge \kappa
\]
for some $\kappa>0$, with the implicit constants is independent of $\kappa$.
\end{itemize}

If $\sigma=-1$, then the solution is global and one of the two alternatives above holds. 
\end{proposition}

\begin{proof}
If $\sigma=-1$, then conservation of energy yields an \emph{a priori} bound on the $H^1$ norm of $u$ (note that the contribution of the potential is controlled by the small mass assumption).  In particular, the solution is global.  Furthermore, the dichotomy follows from the coercivity of $\mathbb{K}_{V,2}$ under the small mass assumption (cf. Proposition~\ref{p:smt}).

Suppose instead $\sigma=1$.  Using \eqref{e:variational_characterization} and the trichotomy of Proposition~\ref{p:smt}, we find that the validity of 
\[
\mathbb{H}_0(u(t)) >1\qtq{and}  \mathbb{K}_{V,2} (u(t)) <0
\]
is preserved in time.  Thus, the failure of this property is preserved as well. 

Now, if $\mathbb{H}_0(u(t)) \le 1$ and $\mathbb{K}_{V,2} (u(t)) <0$ then we see from the trichotomy of Proposition~\ref{p:smt} that
$\mathbb{H}_0(u(t)) \lesssim \mu$, which implies that the first alternative holds (see \cite{GNT} for the scattering result). 

If instead $\mathbb{K}_{V,2} (u(t)) \ge 0$ then Sobolev embedding and Young's inequality yield
\begin{align*}
	\mathbb{H}_0(u(t)) &= 3\mathbb{E}_V(u) - \mathbb{K}_{V,2} (u(t)) - \tfrac12\int_{\R^3} (x\cdot \nabla V + 3V)|u(t,x)|^2 dx \\
	&\le 3\mathbb{E}_V(u) + C \mathbb{M}(u) + \tfrac12 \mathbb{H}_0(u(t))
\end{align*}
for some $C>0$.  In particular, we have $\mathbb{H}_0(u(t)) \le 6\mathbb{E}_V(u) + 2C \mathbb{M}(u)$, which guarantees global existence via energy-subcriticality. The remaining bounds are again a consequence of Proposition~\ref{p:smt}. \end{proof}

For global solutions to \eqref{e:nls} with small mass, we will introduce a decomposition into a ground state part and a `radiation' part.  The equation \eqref{e:nls} then reduces to a coupled ODE-PDE system for the ground state parameter and the radiation (see \eqref{def:PDEODE} below). 

We begin by recalling the ground state map $\Phi$ introduced in Lemma~\ref{L:small-solitons}.  We introduce $D_j$ to denote the derivative with respect to the real and imaginary part of $z=z^1+iz^2$.  That is, we view $\Phi(z) = \Phi(z^1,z^2)$ and denote
\begin{align*}
D_1 \Phi(z) &{}= \tfrac{\partial}{\partial z^1} \Phi(z^1,z^2), &
D_2 \Phi(z) &{}= \tfrac{\partial}{\partial z^2} \Phi(z^1,z^2).
\end{align*}
For $z \in \{z \in \mathbb{C}: |z| < 2z_1 \}\backslash\{0\} $, we define
\[
P_c [z] H^1 := \{ \varphi \in H^1 \ |\ \Re (i \varphi, D_j \Phi[z])=0\text{ for }j=1,2 \}
\]
and for $z=0$:
\[
P_c [0] H^1 := \{ \varphi \in H^1 \ |\  ( \varphi, \phi_0 )=0 \}.
\]
To simplify notation, we write $P_c H^1 = P_c [0] H^1$.

We define $D\Phi(z):\R^2\to\R^2$ via 
\[
(D\Phi(z)) w := \begin{pmatrix} D_1 \Phi(z) &  D_2 \Phi(z) \end{pmatrix}
\begin{pmatrix} w^1 \\ w^2 \end{pmatrix},
\]
which we may also view as a map from $\C\to\C$ via the identification 
\[
z^1+iz^2\leftrightarrow \begin{pmatrix} z^1\\z^2\end{pmatrix}.
\] 
Observe that $D\Phi$ is $\R$-linear but not $\C$-linear.  Indeed,
\[
	(D\Phi(z)) (iw) = 
	\begin{pmatrix} D_1 \Phi(z) &  D_2 \Phi(z) \end{pmatrix}
	\begin{pmatrix} - \Im w \\ \Re w \end{pmatrix}
\]
and
\[
	i(D\Phi(z)) w = i \begin{pmatrix} D_1 \Phi(z) &  D_2 \Phi(z) \end{pmatrix}
	\begin{pmatrix} \Re w \\ \Im w \end{pmatrix}
\]
are not identical in general.

We note also that
\begin{equation}\label{e:derivativeDP}
	\tfrac{d}{dt} \Phi (z(t)) =\begin{pmatrix} D_1 \Phi(z) &  D_2 \Phi(z) \end{pmatrix}
	\begin{pmatrix} \dot{z}^1 \\ \dot{z}^2 \end{pmatrix}= (D\Phi(z)) \dot{z},
\end{equation}
and that by the gauge invariance $\Phi (r e^{i\theta}) = e^{i\theta}  \Phi(r)$, one has
\begin{equation}\label{e:gaugeDP}
(D\Phi(z))(iz) = i \Phi(z).
\end{equation}

In what follows, we use the notation
\[
	\mathcal{B}_\mu := \{ \varphi \in H^1(\R^3) \ |\ \mathbb{M} (\varphi) < \mu \}.
\]
The following result (appearing as \cite[Lemma~2.3]{GNT}) allows us to isolate the small soliton component for $H^1$ functions with small mass.
\begin{proposition}[Decomposition around ground states]\label{p:coordinate}
There exist $\mu_4\in (0, \mu_0]$ and a unique mapping
\[
	\mathcal{B}_{\mu_4} \ni \varphi \mapsto (z,\eta)
	\in \{z \in \mathbb{C} : |z| < 2z_1 \}  \times \mathcal{B}_{\mu_4}
\]
such that
\begin{align*}
	&\varphi = \Phi[z] + \eta, \qquad
	 \eta \in \mathcal{B}_{\mu_4} \cap P_c [z] H^1 . 
\end{align*}
Moreover,
\[
	\mathbb{M} (\varphi) = \mathbb{M} (\Phi[z]) + \mathbb{M} (\eta).
\]
The map $\varphi \mapsto (z,\eta)$ is smooth and injective from $\mathcal{B}_{\mu_4}$ to $\C \times H^1$.
\end{proposition}

We introduce the $\C$-linear, bounded projection operator $P_c:H^1\to P_c H^1$ via
\[
P_c \phi := \phi - (\phi,\phi_0)\phi_0 .
\] 
We then have the following lemma relating $P_c$ and $P_c[z]$. 

\begin{lemma}  There exists $\mu_5\in (0, \mu_4]$ and $z_5 \in (0, z_1]$ such that if $|z| \le 2z_5$, then $P_c|_{P_c[z]H^1}$ is invertible.
The inverse $R[z]$ is of the form
\[
	R[z] \phi = \phi - \rho(z,\phi)  \phi_0,\quad \phi\in P_c H^1,
\]
where $\rho\in\mathbb{C}$ is linear with respect to $\phi$ and satisfies
\begin{equation}\label{e:rhoest}
	 |\rho (z,\phi)| \lesssim \sum_{j=1,2} |\Re (i\phi, D_j \Phi[z])| = \sum_{j=1,2} |\Re(i\phi, D_j (\Phi[z]-z\phi_0))|.
\end{equation}
\end{lemma}

\begin{remark}  The operator $R[z]$ is a compact perturbation of the identity operator 
on any space between $H^{2}\cap H^{1}_{1}$ and $H^{-2}+L^{\infty}$. 
\end{remark}

\begin{proof} The map $\rho$ will be defined to impose that $R[z]\phi\in P_c[z]H^1$, which requires the orthogonality conditions
\[
0=\Re (i R[z] \phi, D_j \Phi[z]) = \Re i(\phi, D_j \Phi[z]) - \Re i \rho(z,\phi) (\phi_0, D_j \Phi[z])
\]
for $j=1,2.$ Using $D_1 (z\phi_0)=\phi_0$ and $D_2 (z\phi_0)=i \phi_0$, we see that for $\phi \in P_c$ these  conditions are equivalent to
\begin{align*}
&	\Re i \rho(z,\phi) (\phi_0, D_1 \Phi[z]) = 
	\Re i \rho(z,\phi) (1+(\phi_0, D_1(\Phi[z]-z\phi_0))),\\
&	\Re i \rho(z,\phi) (\phi_0, D_2 \Phi[z]) = 
	\Re i \rho(z,\phi) (i+(\phi_0, D_2(\Phi[z]-z\phi_0))).
\end{align*}
Thus, $\rho$ may be defined via
\begin{align*}
	&\begin{bmatrix}-\Im (\phi_0, D_1(\Phi[z]-z\phi_0)) & -1 -\Re (\phi_0, D_1(\Phi[z]-z\phi_0)) \\
	-1 - \Im(\phi_0, D_2(\Phi[z]-z\phi_0)) & -\Re (\phi_0, D_2(\Phi[z]-z\phi_0)) \end{bmatrix}
	\begin{bmatrix} \Re  \rho(z,\phi) \\ \Im  \rho(z,\phi) \end{bmatrix} \\
	&=
	\begin{bmatrix} \Re i(\phi, D_1 \Phi[z]) \\ \Re i(\phi, D_2 \Phi[z]) \end{bmatrix},
\end{align*}
provided the coefficient matrix above is invertible. In particular, invertibility from the fact that $\Phi[z]-z\phi_0 =O(z^2)$ in $H^1$ if $z_5$ is sufficiently small.  Furthermore, we obtain the estimate
\[
	|\rho (z,\phi)| \le 2 \sum_{j=1,2} |\Re (i\phi, D_j \Phi[z])|.
\]
As
\[
	\Re (i\phi, D_j \Phi[z])= \Re (i\phi, D_j (\Phi[z]-z\phi_0)),
\]
this completes the proof.\end{proof}

We now consider a small mass solution $u(t)$ to \eqref{e:nls}, which we write as
\[
u(t)=\Phi[z(t)]+\eta(t)
\]
according to Proposition~\ref{p:coordinate}.  We can then derive the evolution equations for $\eta(t)$ and $z(t)$ as follows.  We consider only the focusing case $\sigma=+1$, noting that trivial modifications handle the defocusing case. 

Using \eqref{e:nls}, \eqref{e:Phizeq}, \eqref{e:derivativeDP}, and \eqref{e:gaugeDP}, the equation for $\eta$ may be written
\begin{equation}\label{the-equation-for-eta}
	i \d_t \eta + H \eta = \tilde{B}[z] \eta + N(z,\eta) - i (D\Phi(z))( \dot{z} + i\Omega z) ,
\end{equation}
where $\tilde{B}[z]$ is the $\R$-linear operator
\[
	\tilde{B}[z] f := 2 |\Phi[z]|^2 f + \Phi[z]^2 \bar{f}
\]
and $N(z,\eta)$ collects the quadratic and cubic nonlinear terms as follows:
\[
N(z,\eta) :=   2 |\eta|^2 \Phi[z] + \eta^2 \overline{ \Phi[z] } + |\eta|^2 \eta.
\]

Writing $\xi:=P_c\eta$, we have $P_c \d_t \eta = \d_t \xi$, and hence $\d_t \eta = R[z] \d_t \xi.$ Furthermore,
\[
	H \xi = H P_c \eta = H \eta - H \phi_0 (\eta ,\phi_0) = H \eta - \phi_0 (H\eta , \phi_0) = P_c H \eta,
\]
so that $H \eta = R[z] H \xi.$  Thus we obtain the following equation for $\xi$: 
\begin{equation}\label{e:sysderivation1}
	i\d_t \xi + H \xi = B[z] \xi + P_c N(z,R[z]\xi) - iP_c (D\Phi(z)) ( \dot{z} + i\Omega z),
\end{equation}
where $B[z]\xi:= P_c \tilde B[z]R[z]\xi.$

We turn to the equation for $z$. We first differentiate the orthogonality conditions
\[
	\Re (i \eta, D_j \Phi(z)) =0
\]
with respect to $t$. Then, using \eqref{the-equation-for-eta} and the identities 
\[
H D_j \Phi + D_j (\Omega \Phi) = D_j (|\Phi|^2\Phi),
\]
we obtain
\begin{align*}
	\Re (i \eta, (D(D_j\Phi))\dot{z}) &+ \Re (\eta, \Omega D_j \Phi)  
	- \Re (i (D\Phi)(\dot{z} + i\Omega z), D_j \Phi) \\
	& = - \Re (N(z,\eta), D_j \Phi).
\end{align*}
Furthermore, since $\eta \in P_c[z] H^1$,
\[
\Re (\eta, \Omega D_j \Phi) = \Re (i \eta, \Omega D(D_j \Phi) iz )  .
\]
Thus,
\[
\Re (i (D\Phi)(\dot{z} + i\Omega z), D_j \Phi)-\Re (i \eta, (D(D_j\Phi))(\dot{z}+i \Omega z))=  \Re (N(z,\eta), D_j \Phi),
\]
we may rewrite in the form
\[
	\begin{pmatrix} M_{11} & M_{12} \\
	M_{13} & M_{14} \end{pmatrix}
	\begin{pmatrix} \Re (\dot{z} + i\Omega z) \\ \Im (\dot{z} + i\Omega z) \end{pmatrix}
	=
	\begin{pmatrix}  \Re (N(z,\eta), D_1 \Phi) \\  \Re (N(z,\eta), D_2 \Phi) \end{pmatrix}
\]
with
\[
M_{jk} =M_{jk}(z,\eta) = \Re (i D_k\Phi , D_j\Phi)- \Re (i \eta, D_jD_k\Phi).
\]
Thus, we obtain
\begin{equation}\label{e:sysderivation2}
	\dot{z} + i\Omega z
	=\underline{N}(z,\eta):= \begin{pmatrix}1 & i \end{pmatrix} M^{-1} \begin{pmatrix}  \Re (N(z,\eta), D_1 \Phi) \\  \Re (N(z,\eta), D_2 \Phi) \end{pmatrix}.
\end{equation}
Introducing the notation 
\begin{equation}\label{def:tildeN}
	\tilde{N} (z,\xi) := P_c (N(z,R[z]\xi) - i (D\Phi(z))\underline{N}(z,R[z]\xi))
\end{equation}
and combining \eqref{e:sysderivation2} and \eqref{e:sysderivation1}, we obtain the following PDE-ODE system for the pair $(\xi,z)$:
\begin{equation}\label{def:PDEODE}
	\left\{
	\begin{aligned}
	& i\d_t \xi + H \xi = B[z] \xi + \tilde{N} (z,\xi),\\
	& \dot{z} + i\Omega z=\underline{N}(z,R[z]\xi).
	\end{aligned}
	\right.
\end{equation}

We now turn to the scattering result of Theorem~\ref{T}, first in the focusing case.  For $\mu>0$ and $A \in \R$, we define $\mathrm{GS}(\mu,A)$ to be the set of global, $H^1$-bounded solutions $u$ to \eqref{e:nls} with $\mathbb{M}(u) \le \mu$ and $\mathbb{E}_V(u) \le A.$  We then define 
\begin{align*}
	\mathrm{ST}(\mu,A) :={}&  \sup\{ \norm{\xi}_{L^8 L^4 (\R) }\ |\ \Phi[z] + R[z] \xi \in \mathrm{GS}(\mu,A) \},  \\
	\mathcal{X} :={}& \{ (\mu,A) \in \R_+\times \R \ |\ \mathrm{ST}(\mu,A)<\I  \}.
\end{align*}
Observing that $\mathrm{ST}(\mu,A)<\infty$ yields scattering to $\mathscr{S}_0$ (see e.g. \cite[Lemma~6.2]{Nakanishi}), the scattering portion of Theorem~\ref{T} in the focusing case may be rewritten as follows: \label{check111}
\begin{theorem}[Scattering to the ground states]\label{Trewrite} For $\mu>0$ sufficiently small, 
\[
 (\mu,A) \in \mathcal{X} \Longleftrightarrow A < \mathcal{E}_1(\mu).
\]
\end{theorem}
The $\implies$ direction of Theorem~\ref{Trewrite} follows from the existence of an excited state, i.e. Proposition~\ref{P:excited}.  Thus it remains to prove the $\Longleftarrow$ direction.

By the small data theory, we have that for small enough $\mu>0$, we have that $(\mu,A)\in\mathcal{X}$ for all $A$ sufficiently small. To prove the theorem, we suppose towards a contradiction that there exists $(\mu_0,A_0)$ and $\delta_0>0$ such that 
\[
\mu_0 \ll1 ,\quad A_0 = \mathscr{E}_1(\mu_0)-\delta_0<\mathscr{E}_1(\mu_0),\qtq{and} \mathrm{ST}(\mu_0,A_0) =\I.
\]
We then define
\begin{equation}\label{MstarEstar}
\begin{aligned}
E_*&:= \sup\{ A< A_0 \ |\ (\mu_0,A) \in \mathcal{X} \} \le A_0,\\
M_*& := \sup\{ \mu< \mu_0 \ |\ (\mu,E_*) \in \mathcal{X} \} \le \mu_0.
\end{aligned}
\end{equation}
As $\mu \mapsto \mathscr{E}_1(\mu)$ is decreasing, we have
\begin{equation}\label{e:basicEMrelation}
E_* \le \mathscr{E}_1 (\mu_0) - \delta_0 \le \mathscr{E}_1 (M_*) - \delta_0 \le \tfrac1{M_*} \mathbb{M}(Q)\mathbb{E}(Q) -C M_* - \delta_0 .
\end{equation}

Next, we observe that 
\begin{equation}\label{e:MEinside}
\bigl((0, M_*] \times (0,E_*]\bigr)\setminus (M_*,E_*) \subset \mathcal{X}.
\end{equation}
Indeed, by definition of $E_*$, we have that $0<E<E_*\implies(\mu_0,E) \in \mathcal{X}$.  Furthermore, by definition of $\mathrm{ST}(\mu,A)$, if $(\mu_0,E) \in \mathcal{X}$ then $(\mu,E) \in \mathcal{X}$ for all $0< \mu \le \mu_0$. Thus $(0,\mu_0] \times (0,E_*) \subset \mathcal{X}.$  As the definition of $M_*$ implies $(0,M_*) \times \{E_*\} \subset \mathcal{X}$, we obtain \eqref{e:MEinside}.  

We also claim that 
\begin{equation}\label{e:MEoutside}
(M_*,\I) \times ( E_*,\I) \subset \mathcal{X}^c.
\end{equation}
To see this, we argue as follows:

First, if $(M_*,E_*)=(\mu_0,E_0) \not \in \mathcal{X}$, then \eqref{e:MEoutside} follows immediately. 

Second, if $M_*<\mu_0 $ then the definition of $M_*$ shows that for any $\eps>0$ there exists $M^*\in (M_*,M_*+\eps)$ such that $(M^*,E_*) \not\in \mathcal{X}$, which implies that $[M_*+\eps,\I) \times [E_*,\I) \subset \mathcal{X}^c$. As $\eps>0$ may be taken arbitrarily, \eqref{e:MEoutside} follows. 

Finally, if $M_*=\mu_0$ and $E_*<E_0$ then the definition of $E_*$ shows that for any $\eps>0$ there exists $E^*\in (E_*,E_*+\eps)$ such that
$(\mu_0,E^*) \not\in \mathcal{X}$, which implies that $[\mu_0,\I) \times [E_*+\eps,\I) \subset \mathcal{X}^c$.  Again, since $\eps>0$ may be taken arbitrarily, \eqref{e:MEoutside} follows.

Now, using the facts that $M_* \le \mu_0$ and $E_*\leq A_0 <\mathscr{E}_1(\mu_0)$, we see from the continuity of $\mathscr{E}_1(\mu)$ that 
$(M_* +\eps_0,E_* + \eps_0) \in \{ (\mu,A) \ |\ A < \mathscr{E}_1 (\mu)  \}
$
for sufficiently small $\eps_0$. Thus there exists a sequence of global solutions $u_n = \Phi[z_n] + R[z_n] \xi_n$ obeying
\begin{equation}\label{ME-contradiction-assumption}
\mathbb{M}(u_n) = M_*+o_n(1), \quad \mathbb{E}_V(u_n) = E_*+o_n(1),
\end{equation}
and
\begin{equation}\label{84-contradiction-assumption}
\min\bigl\{ \norm{\xi_n}_{L^8_tL^4_x(\R_+)}, \norm{\xi_n}_{L^8_tL^4_x(\R_-)}\bigr\} \to \I\qtq{as}n\to\infty,
\end{equation}
with
\[
\mathbb{E}_V(u_n) \le \mathscr{E}_1 (\mathbb{M}(u_n)) - \tfrac{\delta_0}2\qtq{for all}n.
\]
Furthermore, Proposition~\ref{l:gwp} implies 
\begin{equation}\label{Kbd-contradiction-assumption}
\inf_{n\geq1} \inf_{t\in\R} \mathbb{K}_{V,2}(u_n(t)) \ge \kappa(\delta_0)\qtq{and} \sup_n \norm{u_n}_{L^\I (\R,H^1 )}^2 \lesssim E_* + M_*.
\end{equation}

Continuing from above, we seek to obtain a contradiction under the assumption
\begin{equation}\label{e:fakeassumption}
	M_*  \le \mu_0, \quad
	E_*\le \mathscr{E}_1 (\mu_0)-\delta_0 .
\end{equation}

The key ingredient is the following: 

\begin{proposition}[Existence of a compact non-scattering solution]\label{p:criticalelement} Suppose that \eqref{e:fakeassumption} holds. Then there exists a global solution 
\[
u_*(t)=\Phi[z_*(t)] + R[z_*(t)]\xi_*(t) \in C(\R,H^1)
\]
to \eqref{e:nls} such that
\[
\mathbb{M}(u_*) = M_*, \quad \mathbb{E}_V(u_*) =  E_* ,\qtq{and}\norm{u_*}_{L^\I H^1 (\R)}^2 \lesssim E_*+M_*.
\]
Furthermore, $u_*$ does not scatter to $\mathscr{S}_0$ in either time direction, i.e. 
\[
\norm{\xi_*}_{L^8_t L^4_x(\R_+)} = \norm{\xi_*}_{L^8_t L^4_x(\R_-)}=\I.
\]
Finally,
\begin{equation}\label{COMPACT}
\{ u_*(t) \ |\ t \in \R\} \qtq{is pre-compact in} H^1.
\end{equation}
\end{proposition}

Proposition~\ref{p:criticalelement} is in turn a consequence of the following.

\begin{proposition}[Key convergence result]\label{p:key}
Suppose that \eqref{e:fakeassumption} holds. Let 
\[
u_n (t) =\Phi[z_n(t)] + R[z_n(t)] \xi_n(t)
\]
 be a sequence of global solutions to \eqref{e:nls} such that \eqref{ME-contradiction-assumption}, \eqref{84-contradiction-assumption}, and \eqref{Kbd-contradiction-assumption} hold.  Then, there exist a subsequence in $n$ and a function $u_{0,\I} \in H^1$ such that
\[
u_n(0) \to u_{0,\I} \IN H^1\qtq{as}n\to\infty
\]
along this subsequence. The solution 
\[
u_\I(t)=\Phi[z_\I(t)] + R[z_\I(t)]\xi_\I(t) \in C(\R,H^1)
\]
to \eqref{e:nls} with data $u_\I(0)=u_{0,\I}$ exists globally in time and satisfies the following:
\begin{align*}
&\mathbb{M}(u_\I) = M_*, \quad \mathbb{E}_V(u_\I) =  E_*,\quad
\inf_{t \in \R} \mathbb{K}_{V,2}(u_\I(t)) \ge \kappa(\delta_0),
\end{align*}
and $\norm{u_\I}_{L^\I H^1 (\R)}^2 \lesssim E_*+M_*.$  Moreover, $u_\I(t)$ does not scatter to $\mathscr{S}_0$ in either time direction, i.e. $
\norm{\xi_\I}_{L^8_t L^4_x(\R_+)} = \norm{\xi_\I}_{L^8_t L^4_x(\R_-)}=\I.$
\end{proposition}

\begin{proof}[Proof of Proposition \ref{p:criticalelement}, assuming Proposition~\ref{p:key}] We have already constructed solutions obeying the assumptions of Proposition~\ref{p:key}.  Thus we obtain the global solution $u_\infty$, and it remains only to verify the pre-compactness. In fact, given $\{t_n\}\subset\R$,  $u_n(t):= u_\infty (t+t_n)$ is a sequence of solutions obeying the assumptions of Proposition~\ref{p:key}.  Thus we we can extract a subsequence such that $u_n(0)=u_*(t_n)$ converges strongly in $H^1$. \end{proof}

The proof of Proposition~\ref{p:key} (given in Section~\ref{S:key}) in the non-radial setting is the key new ingredient in this paper.  For now, we take Proposition~\ref{p:key} (and hence Proposition~\ref{p:criticalelement}) for granted and complete the proof of the scattering part of Theorem~\ref{T} (cf. Theorem~\ref{Trewrite} above).  As in \cite{Nakanishi}, the proof relies on a localized virial argument, which shows that a solution as in Proposition~\ref{p:criticalelement} cannot exist.  In \cite{Nakanishi}, an essential role played by the radial assumption is to guarantee that the critical element is compact in $H^1$ without the need for a moving spatial center.  Proposition~\ref{p:criticalelement} asserts that the same is true in the non-radial setting; however, proving this is a non-trivial task, comprising the majority of the work in this paper. 

\begin{proof}[Proof of Theorem~\ref{Trewrite}] Continuing from above, our task is to reach a contradiction under the assumption \eqref{e:fakeassumption}.  To this end, we apply Proposition~\ref{p:criticalelement} to obtain the solution $u_*$, which satisfies
\[
	\inf_{t \in \R} \mathbb{K}_{V,2} (u_*(t)) \ge \kappa_* >0.
\]
Let $\zeta (x)$ be a non-negative radial smooth function such that
\begin{equation}\label{def-zeta}
\zeta (x) = \left\{\begin{aligned} &|x| && |x| \le 1,\\
& \tfrac32 && |x|\ge 2.  \end{aligned}\right. 
\end{equation}
and $\zeta_R (x)=\zeta(x/R)$ for $R>0$. We define
\begin{align}
\mathbb{K}_{V,2}^R(f) 
&= \int_{\R^3} \zeta_{0,R}|\nabla f|^2 + \zeta_{1,R}|\tfrac{x}{|x|}\cdot\nabla f|^2 + 
\left(\tfrac{1}{R}\right)^2\zeta_{2,R}|f|^2+\zeta_{3,R}|f|^4\nonumber \\
&\qquad- \tfrac12\zeta_{0,R}(x\cdot\nabla V)|f|^2\,dx, \label{Kv}
\end{align}
where
\begin{align*}
 \zeta_0 &{}= 1- |x|^{-1} \zeta, &
\zeta_1 &{}= |x|^{-1} \zeta - \d_r \zeta ,\\ 
\zeta_2 &{}= \tfrac14\Delta[\partial_r + \tfrac{2}{|x|}], & \zeta_3 &{}= -\tfrac14[-3+(\partial_r + \tfrac{2}{|x|})\zeta].
\end{align*}
Then, we have
\begin{align*}
-\d_t  \int_{\R^3} R \zeta_R \tfrac{x}{|x|} \cdot \Im (\bar u_* {\nabla u_*} ) dx 
= 2 \{\mathbb{K}_{V,2}(u_*) -\mathbb{K}_{V,2}^{R}(u_*)\}.
\end{align*}

Since $\zeta_{j,R} =0$ for $|x| \le R$ (for $j=0,1,2,3$), we see from pre-compactness in $H^1$ that there exists $R_*>0$ such that
\[
\sup_{t\in\R} \d_t  \int_{\R^3} R_* \zeta_{R_*} \tfrac{x}{|x|} \cdot \Im (\bar u_* {\nabla u_*} ) dx
\le -\inf_{t \in \R} \mathbb{K}_{V,2} (u_*(t)) \le -\kappa_* <0.
\]
However, this contradicts the fact that
\[
\sup_{t\in\R} \abs{  \int_{\R^3} R_* \zeta_{R_*} \tfrac{x}{|x|} \cdot \Im (\bar u_* {\nabla u_*} ) dx }\lesssim R_* M_* \sup_{t} \norm{\nabla u_*(t)}_{L^2} < \I.
\]
\end{proof}

In the rest of this section, we establish the remaining part of Theorem~\ref{T} in the focusing setting.  In particular, focusing on the forward time direction, we will prove the following:

\begin{proposition}[Growup result] 
Let $\sigma=1$. Let $u_0 \in H^1(\R^3)$ satisfy 
\[
\mathbb{M}(u_0) \le \mu_0\qtq{and} \mathbb{E}_V (u_0) < \mathscr{E}_1 (\mathbb{M}(u_0)).
\]
If $\mathbb{K}(u_0)<0$ {and} $\mathbb{H}_0(u_0)\ge 1,$ then the corresponding solution $u$ to \eqref{e:nls} satisfies
\[
\lim_{t \uparrow T_{\max}} \norm{u(t)}_{H^1} =\I,
\]
where $T_{\max}$ is the maximal time of existence.
\end{proposition}
\begin{proof} We argue as in \cite{AkahoriNawa}.  If $T_{\max}$ is finite we obtain the result by the standard blowup criterion.  We will prove that if $T_{\max}=\I$, 
\[
\inf_{R>0} \limsup_{t \to \I} \int_{|x| \ge R} |\nabla u(t)|^2 dx = \I,
\]
which implies the result.  Suppose instead that for some $R_0>0$, we have
\begin{equation}\label{e:growup_h0}
h_0:= \sup_{t \ge 0 } \int_{|x| \ge R_0 } |\nabla u(t) |^2dx <\I.
\end{equation}
As $u\in C([0,\infty); H^1)$, we see from \eqref{e:variational_characterization} that there exists $\kappa>0$ such that
\begin{equation}\label{e:growup_kappa}
\mathbb{K}_{V,2}(u(t)) \le - \kappa\qtq{for all}t\ge 0.
\end{equation}

We now define
\[
W_R (x) = R^2 \int_0^{|x|/R} \zeta (r) dr,
\]
regarding $\zeta$ as a function of $r$. We then have the virial identity 
\begin{align*}
(W_R, |u(t)|^2) ={}& 	(W_R, |u_0|^2) - 2t\int_{\R^3} R \zeta_R \frac{x}{|x|} \cdot \Im (\bar u_0 {\nabla u_0} ) dx\\
&+ 4\left\{\int_0^t \int_0^s \mathbb{K}_{V,2} (u(\sigma)) d\sigma ds - \int_0^t \int_0^s \mathbb{K}_{V,2}^R (u(\sigma)) d\sigma ds \right\},
\end{align*}
where $\mathbb{K}_{V,2}^R$ is given by (\ref{Kv}).

We first claim that there exists $m_0>0$ such that if $v \in H^1(\R^3)$ satisfies $\mathbb{K}_{V,2}^R (v) \le -\tfrac12\kappa$,  $\mathbb{M}(v) \le \mathbb{M}(u_0)$, and
\[
	\int_{|x|\ge R} |\nabla v|^2 dx \le h_0\qtq{for some}R\ge 1,
\]
then
\[
\int_{|x|\ge R} |v|^2 dx \ge 2m_0
\]
for the same $R$, where $h_0$ is the constant given in \eqref{e:growup_h0}.
Indeed, pick such $v \in H^1(\R^3)$.
Then, since $\zeta_{0,R} \ge0$ and $\zeta_{1,R} \ge0$,
\begin{align*}
\tfrac12\kappa \le - \mathbb{K}_{V,2}^R(v) \le{}& \int |\zeta_{3,R}| |v|^4 dx + (\norm{\zeta_2}_{L^\I} + \norm{x\cdot \nabla V}_{L^\I}) \norm{v}_{L^2(|x|\ge R)}^2\\
\lesssim{}& (\norm{\nabla v}_{L^2(|x|\ge R)}^3 + \norm{v}_{L^2}) \norm{v}_{L^2(|x|\ge R)} \\
\le{}& (h_0^{\frac32} + \norm{u_0}_{L^2}) \norm{v}_{L^2(|x|\ge R)},
\end{align*}
showing the desired lower bound.

We next claim that if $R$ is sufficently large then
\begin{equation}\label{e:growup_claim2}
	\int_{|x| \ge R} |u(t)|^2 dx \le m_0
\end{equation}
for all $t\ge 0$. To this end, we first observe that if $R>0$ is large then \eqref{e:growup_claim2} is true at $t=0$. Now, suppose towards a contradiction that
\[
T_R := \sup \{ T\ge0 \ |\ \eqref{e:growup_claim2}\text{ is true on }[0,T] \}<\infty.
\]
As $u\in C([0,\infty); H^1)$, we have $T_R>0$ and 
\begin{equation}\label{e:growup_claim2pf1}
\int_{|x| \ge R} |u(T_R)|^2 dx = m_0.
\end{equation}
Then, by the previous claim, we see that one of $\mathbb{K}_{V,2}^R (u(T_R)) \le -\tfrac12\kappa$, $\mathbb{M}(u(T_R)) \le \mathbb{M}(u_0)$, or $\int_{|x|\ge R} |\nabla u(T_R)|^2 dx \le h_0$ must fail. By mass conservation, the second of these is true. Furthermore, if $R\ge R_0$ then the third item holds. Thus, if $R \ge R_0$,
\[
\mathbb{K}_{V,2}^R (u(T_R)) > -\tfrac12\kappa.
\]

Now, by the virial identity, we have
\begin{align*}
(W_R, |u(T_R)|^2) \leq{}& 	(W_R, |u_0|^2) - 2T_R\int_{\R^3} R \zeta_R \frac{x}{|x|} \cdot \Im (\overline{u_0} {\nabla u_0} ) dx -  \kappa T_R^2 \\
={}& (W_R, |u_0|^2) 	- \kappa \(T_R + \frac{1}{\kappa} \int_{\R^3} R \zeta_R \frac{x}{|x|} \cdot \Im (\overline{u_0} {\nabla u_0} ) dx\)^2 \\
& + \tfrac1{\kappa} \(\int_{\R^3} R \zeta_R \frac{x}{|x|} \cdot \Im (\overline{u_0} {\nabla u_0} ) dx\)^2\\
\le{}& (W_R, |u_0|^2)+ \tfrac1{2\kappa} \norm{\nabla u_0}_{L^2}^2  \norm{ R\zeta_R u_0 }_{L^2}^2.
\end{align*}
We now claim that the right-hand side is $o(R^2)$ as $R\to\I$. Indeed, for any $\eps>0$, there exists $R_\eps$ such that
\[
\int_{|x|\ge R_\eps} |u_0|^2 dx \le \eps.
\]
Furthermore, as $\norm{W_R}_{L^\I} \lesssim R^2$ and $\norm{\zeta_R}_{L^\I} \le 3/2$, we have
\[
\int_{|x|\ge R_\eps} W_R |u_0|^2 dx  +  \int_{|x|\ge R_\eps} R^2 \zeta_R^2 |u_0|^2 dx \lesssim \eps R^2.
\] 
Finally, if $R \ge R_\eps$ then $W_R = |x|^2$ and $R\zeta_R = |x|$ for $|x| \le R_\eps$ and hence
\[
\int_{|x|\le R_\eps} W_R |u_0|^2 dx  +  \int_{|x|\le R_\eps} R^2\zeta_R^2 |u_0|^2 dx\lesssim R_\eps^2 \norm{u_0}_{L^2}^2.
\]
Combining these two estimate, we obtain the claim.

Choosing $R$ sufficiently large, we therefore have that
\[
(W_R, |u(T_R)|^2) \le \tfrac12 m_0 R^2.
\]
Consequently, 
\[
R^2 \int_{|x|\ge R}| u(T_R)|^2 dx  \le (W_R, |u(T_R)|^2) \le \frac12 m_0 R^2,
\]
which contradicts \eqref{e:growup_claim2pf1}.  Thus, we obtain \eqref{e:growup_claim2} for all $t\geq 0$.

Combining above two claims, we see that there exists $R>0$ such that
\[
	\sup_{t\ge 0} \int_{|x| \ge R} |u(t)|^2 dx \le m_0, \quad \inf_{t\ge 0} \mathbb{K}_{V,2}^R (u(t)) \ge -\tfrac12\kappa.
\]
However, by the virial identity, one then obtains 
\[
0 \le (W_R, |u(t)|^2) \le C(u_0) t - \kappa t^2
\]
for all $t\ge0$, which yields a contradiction for sufficiently large $t$. 
\end{proof}

So far, we have restricted our attention primarily to the focusing case of \eqref{e:nls}.  Straightforward modifications suffice to treat the defocusing case, as well.  In particular, due to the absence of excited solitons (cf. Conjecture~\ref{conjecture}), the condition $\mathbb{E}_V(u_0)<\mathscr{E}_1(\mathbb{M}(u_0))$ does not impose any restriction on the size of the energy.  Moreover, the condition \eqref{grow-up-condition} cannot be satisfied in the defocusing case (cf. Proposition~\ref{p:smt}); thus, in the defocusing case, only the scattering alternative occurs. 

\section{Strichartz estimates and stability for the linearized equation}\label{S:linear} 

In the previous section, we saw that any small mass solution to \eqref{e:nls} admits a decomposition of the form $u=\Phi[z]+R[z]\xi$.  In this section, our interest is in the equation satisfied by $\xi(t)$, as well as the underlying linear part of this equation.

We introduce the following set of functions from \cite{Nakanishi}. 

\begin{definition}
Let $\mathrm{SBC}$ denote the set of continuous functions $z(t):\R \to \C$ satisfying either of the following:
\begin{itemize}
\item $z(t)$ arises from a global solution $u(t)$ to \eqref{e:nls} such that
\[
\mathbb{M}(u) \le \mu_0\qtq{and}\sup_{t\in\R} \mathbb{H}_0(u(t))\lesssim \mu^{-1}
\]
via the decomposition $u(t) = \Phi[z(t)] + \eta(t)$.
\item $z(t)$ arises from a solution $u(t)$ to \eqref{e:nls} defined on an interval $I$ via the decomposition $u(t)=\Phi[z(t)]+\eta(t)$. Outside of $I$, $z(t)$ is constant. 
\end{itemize}
An element $z(t)$ of $\mathrm{SBC}$ is continuously differentiable on $\R$ with at most two exceptional points and obeys the bounds
\[
\norm{z}_{L^\I (\R)} \le \mu_0\qtq{and}\norm{\dot{z} }_{L^\I (\R)} \lesssim 1.
\]
\end{definition}

\begin{remark} By the Arzel\'a--Ascoli Theorem, any sequence $\{z_n\}_n \subset \mathrm{SBC}$ has a subsequence that converges locally uniformly to a continuous bounded function, $z_\I$ on $\R$.  Using the fact that a weak limit of bounded solutions to \eqref{e:nls} is again a solution, it follows that $z_\I\in \mathrm{SBC}$.
\end{remark}

We next introduce the solution operator for the linearized equation.  Recall that $B[z]:=P_c\tilde B[z]R[z]$, where 
\[
\tilde B[z]f = 2|\Phi[z]|^2 f + \Phi[z]^2 \bar f,
\]
with $z\mapsto\Phi[z]$ parametrizing the ground states and $R[z]$ giving the inverse of $P_c|_{P_c[z]H^1}$ (cf. Proposition~\ref{p:coordinate} and the subsequent discussion). 

\begin{definition}\label{def:Ltsz}
For $z \in \mathrm{SBC}$, we define $\mathcal{L}(t,s;z):P_cH^1 \to P_cH^1$ as follows:
\[
v (t) := \mathcal{L}(t,s;z) \varphi
\qtq{is the unique solution to}
\left\{\begin{aligned}
& i \d_t v + H v =B[z] v, \quad \\
& v(s) = \varphi.
\end{aligned}\right.
\]
Then $\mathcal{L}(t,s;z)$ is a bounded operator, which is $\R$-linear but not $\C$-linear (except for the case $\mathcal{L}(t,s;0)=e^{i(t-s)H}$).
\end{definition}

We then have the following: 

\begin{lemma}\label{L:properties-of-L} For $z \in \mathrm{SBC}$, $\mathcal{L}(t,s;z)$ obeys the following: for any $t,s,r \in \R$,
\begin{itemize}
\item $\mathcal{L} (t,t;z)= \mathrm{Id}$;
\item $\mathcal{L} (t,s;z)\mathcal{L} (s,r;z)=\mathcal{L} (t,r;z)$;
\item $\mathcal{L} (t,s;z)=\mathcal{L} (t-r,s-r;z(\cdot + r))$.
\end{itemize}
\end{lemma}

For the inhomogeneous problem, we have the following Duhamel formula. 

\begin{proposition}[Duhamel formula]
Let $z \in \mathrm{SBC}$, $I \subset \R$, and $N: I \to P_c H^1$. The $w(t)$ to the inhomogeneous linearized equation
\[
\left\{\begin{aligned}
& i \d_t w + H w =B[z] w + N, \quad \\
& w(s) = \varphi
\end{aligned}\right.
\]
is given by
\begin{equation}
w(t) = \mathcal{L} (t,s;z) \varphi + \mathcal{D} (t,s;z)N,
\end{equation}
where
\begin{equation}\label{def:D}
\mathcal{D} (t,s;z)N := -  \int_s^t \mathcal{L}(t,r;z) (i N(r)) dr
\end{equation}
\end{proposition}
In general we have $\mathcal{L}(t,r;z) (i N(r))\neq i \mathcal{L}(t,r;z)  N(r)$, as $\mathcal{L}(t,s;z)$ is not $\C$-linear.

For the usual Strichartz spaces, we follow \cite{Nakanishi} and define
\begin{equation}\label{4.26}
\mathrm{Str}^s := L^\I_t H^s_x \cap L^2_t B^{s}_{6,2}, \quad \mathrm{Str}^{*s} := L^1_t H^s_x + L^2_t B^{s}_{6/5,2}
\end{equation}
and observe the embeddings $\mathrm{Str}^{1/2} \hookrightarrow L^4_t B^{1/2}_{3,2}  \hookrightarrow L^4_t L^6_x$.  We record here the following uniform Strichartz estimate for the linearized flow from \cite[Lemma~4.3]{Nakanishi}. Here the uniformity is with respect to the choice of $z\in\mathrm{SBC}$. 

\begin{proposition}[Uniform Strichartz estimate]\label{Prop:Strichartz}
Let $z \in \mathrm{SBC}$ and $I\subset\R$.  Then for any $\varphi \in P_c H^1$ and $\theta\in[0,1]$, we have
\[
	\norm{\mathcal{L}(\cdot, 0;z)\varphi }_{\mathrm{Stz}^\theta (\R)} \lesssim  \norm{\varphi}_{H^\theta}.
\]
Furthermore, 
\[
\inf_{t\in \R}\norm{\mathcal{L}(t, 0;z)\varphi }_{H^\theta} \sim \norm{\varphi}_{H^\theta}\sim \sup_{t\in \R} 	\norm{\mathcal{L}(t, 0;z)\varphi }_{H^\theta}.
\]
Finally, for any $s_0 \in I$,
\[
\norm{ \mathcal{D}(\cdot,s_0;z)(P_c F) }_{\mathrm{Stz}^\theta (I)}	\lesssim \norm{F}_{\mathrm{Stz}^{\theta*}(I)}.
\]
In the above estimates, the implicit constants are independent of the choice of $z(t)$.
\end{proposition}

The next lemma also a consequence of \cite[Lemma~4.3]{Nakanishi}.  For the sake of completeness, we provide a proof here. 

\begin{lemma} For any $\varphi \in P_c H^1$,
\[
	\sup_{z \in \mathrm{SBC}} \norm{ \mathcal{L}(t,0;z) \varphi }_{L^2 H^1_6 ([0,\I))}
	\lesssim \inf_{z \in \mathrm{SBC}} \norm{ \mathcal{L}(t,0;z) \varphi }_{L^2 H^1_6 ([0,\I))}
\]
A similar estimate holds on the interval $(-\I,0]$.
\end{lemma}

\begin{proof} Let $z \in \mathrm{SBC}$ and $\varphi \in P_c H^1$. Writing $u(t) = \mathcal{L} (t,0;z) \varphi$, we use the Duhamel with respect to $e^{itH}$ to write
\[
	u(t) = e^{it H} \varphi - i \int_0^t e^{i(t-s)H} B[z(s)] u(s)\,ds.
\]
Thus, using the Strichartz estimate for $e^{itH}P_c$, we have
\[
	\norm{u- e^{it H} \varphi}_{L^2 H^1_6([0,\I))} \le 
	C \norm{z}_{L^\I_t(\R)}^2 \norm{u}_{L^2 H^1_6([0,\I))} \ll \norm{u}_{L^2 H^1_6([0,\I))},
\]
where the implicit small constant depends only on upper bound on $\norm{z}_{L^\I_t(\R)}$.  The result follows. \end{proof}

We also need the following vanishing lemma.

\begin{proposition}\label{p:linearunifdecaytau} For any $\varphi \in P_c H^1$, 
\[
	\lim_{\tau\to\I} \sup_{z \in \mathrm{SBC}} \norm{ \mathcal{L}(t,0;z) \varphi }_{L^2 H^1_6 ((-\I,-\tau]\cup[\tau,\I))}
	=0.
\]
\end{proposition}
\begin{proof}  It suffices to consider the forward time direction.  If the proposition fails, then there exists a sequence $\tau_n\to\I$, $z_n\in\mathrm{SBC}$, and $c_0>0$ such that
\begin{equation}\label{e:Lunifpf1}
	\norm{ \mathcal{L}(t,0;z_n) \varphi }_{L^2 H^1_6 ([\tau_n ,\I))}
	\ge c_0.
\end{equation}
Passing to a subsequence, $z_n$ converges locally uniformly to some $z_\I \in \mathrm{SBC}$.  We then choose $\tau^*>0$ so that
\[
	\norm{ \mathcal{L}(t,0;z_\I) \varphi }_{L^2 H^1_6 ([\tau^* ,\I))} \le \tfrac{c_0}{3C_0},
\]
where $C_0$ is the implicit constant in the preceding lemma. 

Now set $u_n(t) = \mathcal{L}(t,0;z_n) \varphi $. By the Duhamel formula with respect to
$\mathcal{L}(t,0;z_\I)$, one has
\[
	u_n(t) = \mathcal{L}(t,0;z_\I) \varphi
	- \int_0^t \mathcal{L}(t,s;z_\I) i((B[z_n(s)]-B[z_\I(s)])u_n(s) ) ds
\]
By the uniform Strichartz estimate, we have
\begin{align*}
	\norm{ u_n(t) - \mathcal{L}(t,0;z_\I) \varphi }_{L^\I H^1 ([0,\tau^*])}
	&\lesssim \norm{ (B[z_n]-B[z_\I])u_n  }_{L^2 H^1_{6/5} ([0,\tau^*])} \\
	&\lesssim \norm{z_n-z_\I}_{L^\I ([0,\tau^*])} \norm{u_n}_{L^2 H^1_6 (\R)} \\
	&\lesssim \norm{z_n-z_\I}_{L^\I ([0,\tau^*])} \norm{\varphi}_{H^1} \\
	&\to 0\qtq{as}n\to\infty.
\end{align*}
Thus, by the preceding lemma and the Strichartz estimate, we have 
\begin{align*}
	 \| \mathcal{L} &(t,0;z_n) \varphi \|_{L^2 H^1_6 ([\tau^*,\I))} \\
	&= \norm{ \mathcal{L}(t,0;z_n(\cdot+ \tau^*)) \mathcal{L}(\tau^*,0;z_n) \varphi }_{L^2 H^1_6 ([0,\I))} \\
	&\le C_0 \norm{ \mathcal{L}(t,0;z_\I(\cdot+ \tau^*)) \mathcal{L}(\tau^*,0;z_n) \varphi }_{L^2 H^1_6 ([0,\I))} \\
	&= C_0 \norm{ \mathcal{L}(t,\tau^*;z_\I) \mathcal{L}(\tau^*,0;z_n) \varphi }_{L^2 H^1_6 ([\tau^*,\I))} \\
	&\le  C_0 \norm{ \mathcal{L}(t,0;z_\I) \varphi }_{L^2 H^1_6 ([\tau^*,\I))} \\
	&\quad + C_0 \norm{ \mathcal{L}(t,\tau^*;z_\I) (u_n(\tau^*) - \mathcal{L}(\tau^*,0;
	{z_{\infty}}) \varphi) }_{L^2 H^1_6 ([\tau^*,\I))} \\
	&\le \tfrac{c_0}3 + o_n(1) \le \tfrac{2c_0}{3}
\end{align*}
for large $n$, which contradicts \eqref{e:Lunifpf1} since $\tau_n > \tau^*$ for large $n$. \end{proof}

Finally, we need a Strichartz estimate with non-admissible pairs from \cite{Nakanishi}.

\begin{lemma}[Non-admissible Strichartz estimate]\label{L:NAS}
Let $(q_0,r_0)$, $(q_1,r_1) \in (1,\I) \times (2,6]$ and $\sigma_j = \tfrac{2}{q_j} + 3(\tfrac{1}{r_j} - \tfrac{1}{2})$ satisfy
\[
	\sigma_0 + \sigma_1 =0,\quad   \max_{j=0,1}(\sigma_j-\tfrac{1}{q_j})<0, \quad |\sigma_j| \le \tfrac{2}{3}.
\]
Then, for any $z \in \mathrm{SBC}$ and $s_0 \in I \subset \R$,
\[
	\norm{ \mathcal{D}(\cdot,s_0;z)(P_c F)  }_{L^{q_0}L^{r_0} (I) }
	\lesssim \norm{F}_{ L^{q_1'}L^{r_1'} (I) }.
\]
\end{lemma}

\begin{proof} The proof is based on the argument in \cite{Y}. We first see that the estimate holds for the case $V\equiv0$ by \cite{Foc,Vi}. Then, by assumption (A4), we see that the wave operator $W=\lim_{t\to\infty}e^{itH}e^{it\Delta}$ and its adjoint $W^{\ast}$ are bounded in $L^{p}$, $2\le p\le6$. Combining these facts, we obtain the estimate.
\end{proof}

At a few points we will also make use of the following dispersive estimates for $\mathcal{L}$.

\begin{lemma}\label{L:weighted-dispersive}  
The following estimates hold uniformly for $z\in \mathrm{SBC}$:  For any $2\leq p_1<6<p_2
{\leq p}$, where $p$ is given by Assumption $(A4)$ and $t\neq s$,
\[
\|\mathcal{L}(t,s;z)P_c\varphi\|_{L^{p_1}+L^{p_2}} \lesssim f(t-s)\|\varphi\|_{L^{p_1'}\cap L^{p_2'}},\quad  f(t):=\min_{j=1,2}\bigl\{ |t|^{-(\frac32-\frac3{p_j})}\bigr\}.
\]
\end{lemma}

\begin{proof} The assumptions on $V$ imply that the standard dispersive estimates hold for $e^{-itH}$.  We will use the Duhamel formula and the smallness of $z\in\mathrm{SBC}$ to transfer the estimates to $\mathcal{L}(t,0;z)$.  Writing $u(t)=\mathcal{L}(t,0;z)P_c\varphi,$ we have
\[
u(t) = e^{itH}P_c\varphi -i\int_0^t e^{i(t-s)H}B[z(s)]u(s)\,ds.
\]
We note that by translating $z$, we may assume $s=0$. Then, noting that
\[
\|B[z(s)]u\|_{L^{p_1'}\cap L^{p_2'}}\lesssim \|z\|_{L^\infty}^2\|u\|_{L^{p_1}+L^{p_2}}, 
\]
we fix $T>0$ and $t\in[0,T]$ and use the dispersive estimate for $e^{itH}$ to obtain
\begin{align*}
\tfrac{1}{f(t)}&\|u(t)\|_{L^{p_1}+L^{p_2}} \\ & \lesssim \|\varphi\|_{L^{p_1'}\cap L^{p_2'}} + \tfrac{1}{f(t)}\int_0^t f(t-s)f(s)\,ds\cdot\|z\|_{L_t^\infty}^2 \sup_{t\in[0,T]} \tfrac{1}{f(t)}\|u(t)\|_{L^{p_1}+L^{p_2}} \\
& \lesssim \|\varphi\|_{L^{p_1'}\cap L^{p_2'}} + c\biggl[\tfrac{1}{f(t)}\int_0^t f(t-s)f(s)\,ds\biggr] \sup_{t\in[0,T]} \tfrac{1}{f(t)}\|u(t)\|_{L^{p_1}+L^{p_2}}
\end{align*}
for some $0<c\ll 1$.  Thus, the result follows as before, provided we have that
\[
\sup_{t>0} \tfrac{1}{f(t)}\int_0^t f(t-s)f(s)\,ds \lesssim 1. 
\]
Indeed, the assumption that $p_1<6<p_2$ guarantees that $f\in L_t^1$, and so the estimate follows by observing that
\begin{align*}
\tfrac{1}{f(t)} \int_0^t f(t-s)f(s)\,ds & \lesssim \tfrac{1}{f(t)}\biggl[\int_0^{t/2} f(t)f(s)\,ds +\int_{t/2}^t f(t-s)f(t)\,ds\biggr]\lesssim \|f\|_{L_t^1}
\end{align*}
uniformly in $t>0$.\end{proof}

We next consider the equation
\begin{equation}\label{e:6.1}
	 i\d_t \xi + H \xi = B[z] \xi + \tilde{N} (z,\xi)+\mathcal{E}
\end{equation}
where $\tilde N$ is as in \eqref{def:tildeN} and $\mathcal{E}=P_c\mathcal{E}$ is a perturbative term.  We regard $z \in \mathrm{SBC}$ as given and view this as an equation for $\xi$.  In what follows, we will utilize some semi-norms introduced in \cite{Nakanishi}.  Given $s_0 \in \R$ and $z \in C(\R;\C)$, we first define $\mathcal{L}_{>}(s_0,z)$: $C(\R;H^1)\to C(\R;H^1)$ by
\begin{equation}\label{e:u>}
[\mathcal{L}_{>}(s,z)u](t)=\begin{cases}	
u(t) & t < s_0, \\
\mathcal{L}(t,s_0;z )(u|_{t=s_0}) & t \ge s_0.
\end{cases}
\end{equation}
Thus $\mathcal{L}_{>}(s_0,z)$ maps a function to a solution to $i\partial_t v + Hv = B[z]v$ with $z(t)$ in the interval $(s_0,\I)$ in such a way that the resulting function is continuous at $t=s_0$.  We may omit $z$ if it is clear from the context.   We then define
\begin{equation}\label{nakanishi-seminorms} 
\begin{aligned}
&\| u \|_{[z;T_0,T_1;T_2]} := \sup_{T_0<S<T<T_1} \norm{ \mathcal{L}_{>}(T,z) u-\mathcal{L}_{>}(S,z) u } _{L^8_tL^4_x( (T_0,T_2)\times \R^3 )}, \\
&\| u \|_{[z;T_0,T_1;T_2]'} := \sup_{T_0<T<T_1} \norm{ \mathcal{L}_{>}(T,z) u-\mathcal{L}_{>}(T_0,z) u } _{L^8_tL^4_x( (T_0,T_2)\times \R^3 )},
\end{aligned}
\end{equation}
which are equivalent semi-norms.  In particular, these semi-norms are zero precisely when $u$ is a solution to $i\partial_t u + Hu=B[z]u$ on $(T_0,T_1)$, so that these semi-norms provide a measure of nonlinear contributions.

We turn to the stability results.  The first result concerns small solutions to the nonlinear problem. 

\begin{proposition}[Perturbation around zero]\label{l:sdt}
Let $-\I<T_0< T_1 \le \I$, $z \in \mathrm{SBC}$, and $\varphi \in H^1(\R^3)$.  Let $\mathcal{E} \in \mathrm{Stz}^{\theta_0 *}(T_0,T_1)$ for some $\theta_0 \in [\tfrac12,1]$, and set
\[
\mathcal{N}_\theta := \norm{z}_{L^\I ([T_0,T_1))} + \norm{\varphi}_{H^\theta}+ \norm{\mathcal{D}(\cdot, T_0) \mathcal{E}}_{\mathrm{Stz}^\theta ([T_0,T_1))}
\]
for $\theta \in [0,1]$, where $\mathcal{D}$ is as in \eqref{def:D}. Suppose that $\mathcal{N}_0 \ll 1$.

$\mathrm{(i)}$ If 
\begin{equation}\label{stability-small1}
	\norm{\mathcal{L} (\cdot, T_0) \varphi + \mathcal{D}(\cdot,T_0)\mathcal{E}}_{ L^8L^4{(T_0,T_1)} } \cdot\mathcal{N}_{\frac12} \ll 1,
\end{equation}
then \eqref{e:6.1} with $\xi (T_0) = \varphi$ has a unique solution $\xi \in C^1 ([T_0,T_1); H^{\theta_0})$ satisfying
\[	
\norm{\xi}_{L^8_t L^4_x ( T_0,T_1 )} \lesssim  \norm{\mathcal{L} (\cdot, T_0) \varphi + \mathcal{D}(\cdot, T_0) \mathcal{E}}_{L^8_t L^4_x ( T_0,T_1 )}
\]
and
\[
\norm{\xi}_{ \mathrm{Stz}^\theta ( T_0,T_1 )} \lesssim \mathcal{N}_\theta\qtq{for all}\theta \in [0,\theta_0].
\]

$\mathrm{(ii)}$ If the solution $\xi \in C^1 ([T_0,T_1); H^{\theta_0})$ to \eqref{e:6.1} with $\xi (T_0) = \varphi$ satisfies
\[
\norm{\xi}_{L^8_tL^4_x{(T_0,T_1)}}  \mathcal{N}_{\frac12} \ll 1
\]
then the following estimates hold: for any $T \in (T_0,T_1)$ and $\theta \in [0,\theta_0]$,
\begin{align*}
	\|\mathcal{L}_{>}(T,z)\xi-  \mathcal{L} (\cdot, T_0) \varphi - &\mathcal{D}(\cdot,T_0)\mathcal{E} \|_{\mathrm{Stz}^\theta (T_0,T_1)}\\
	&\lesssim \mathcal{N}_\theta \norm{\mathcal{L} (\cdot, T_0) \varphi +\mathcal{D}(\cdot,T_0)\mathcal{E}  }_{L^8_tL^4_x(T_0,T)}\\
	&\quad\times (\mathcal{N}_0 + \norm{\mathcal{L} (\cdot, T_0) \varphi +\mathcal{D}(\cdot,T_0)\mathcal{E}  }_{L^8_tL^4_x(T_0,T)}), \\
	\|\xi  - \mathcal{D}(\cdot,T_0)\mathcal{E}\|_{{[z;T_0,T,T_1]}} &\ll \norm{\mathcal{L} (\cdot, T_0) \varphi +\mathcal{D}(\cdot,T_0)\mathcal{E}  }_{L^8_tL^4_x(T_0,T)} \\
	& \sim \norm{\xi}_{L^8_tL^4_x(T_0,T)}.
\end{align*}
\end{proposition}

\begin{proof}
(i) As the equation is $H^{\frac12}$-critical, it is straightforward to obtain a local solution if $\theta_0\ge \tfrac12$.  We need to obtain a solution on the whole interval $[T_0,T_1]$ and prove the desired bound.

First, using the Strichartz estimate, we have
\[
	\norm{ \mathcal{L} (\cdot, T_0) \varphi }_{\mathrm{Stz}^{\theta} (\R)} \lesssim \mathcal{N}_{\theta}\qtq{for any}\theta\in[0,1],
\]
and (recalling the definition of $\tilde N$ from \eqref{def:tildeN})
\[
	\norm{\xi}_{\mathrm{Stz}^{\frac12}} \le C \mathcal{N}_{\frac12} + 
	C \norm{\xi}_{\mathrm{Stz}^{\frac12}} \norm{\xi}_{L^8L^4} (\norm{\xi}_{\mathrm{Stz}^{\frac12}} + \norm{z}_{L^\I_t}).
\]

Now, by the non-admissible Strichartz estimates of Lemma~\ref{L:NAS}, we have
\begin{align*}
\|&\xi\|_{L^8_tL^4_x} \\
&\le \norm{\mathcal{L} (\cdot, T_0) \varphi + {\mathcal{D}(\cdot,T_0)\mathcal{E}}}_{L^8_tL^4_x} +  C \norm{\tilde{N}(z,\xi)}_{L^{8/3}L^{4/3} + L^{4}L^{6/5}} \\
& \le \norm{\mathcal{L} (\cdot, T_0) \varphi + {\mathcal{D}(\cdot,T_0)\mathcal{E}}}_{L^8_tL^4_x} +C\norm{\xi}_{L^8L^4}^2 (\norm{\xi}_{L^8 L^4} + \norm{\Phi[z]}_{L^\I L^3} ) \\
& \le \norm{\mathcal{L} (\cdot, T_0) \varphi + {\mathcal{D}(\cdot,T_0)\mathcal{E}}}_{L^8_tL^4_x} +C\norm{\xi}_{L^8L^4}^2 (\norm{\xi}_{\mathrm{Stz}^{\frac12}} + \norm{z}_{L^\I} ).
\end{align*}
%
%
Thus, by continuity, there exists $\delta>0$ such that on $(T_0,T_0+\delta)$, 
\begin{align*}
	\norm{\xi}_{\mathrm{Stz}^{\frac12}} {}&\le 100 C\mathcal{N}_{\frac12}, \\
	\norm{\xi}_{L^8L^4} {}&\le  100 \norm{\mathcal{L} (\cdot, T_0) \varphi + \mathcal{D}(\cdot,T_0)\mathcal{E}}_{ L^8L^4 },
\end{align*}
where $C\ge 1$ is the same constant as in the above estimates. 

Recalling the estimates above and \eqref{stability-small1}, we have that on the same interval, 

\begin{align*}
\| \xi\|_{\mathrm{Stz}^{\frac12}} & \le C\mathcal{N}_{\frac12}+C\bigl\{\|\xi\|_{\mathrm{Stz}^{\frac12}} + \|z\|_{L_t^\infty}\bigr\}\|\xi\|_{L^8L^4}\|\xi\|_{\mathrm{Stz}^{\frac12}} \\
& \le C\mathcal{N}_{\frac12} + 100^2C^2\cdot(100C+1)\cdot \mathcal{N}_{\frac12}\norm{\mathcal{L} (\cdot, T_0) \varphi + \mathcal{D}(\cdot,T_0)\mathcal{E}}_{ L^8L^4 }\mathcal{N}_{\frac12}\\
& \leq 2C\mathcal{N}_{\frac12}.
\end{align*}
%
Similarly, one has
\begin{align*}
	\norm{\xi}_{L^8L^4} \le {}&   \norm{\mathcal{L} (\cdot, T_0) \varphi + \mathcal{D}(\cdot,T_0)\eps}_{ L^8L^4 }\\
	&{} + 100^2 C[100C+1] \mathcal{N}_{\frac12}\cdot \norm{\mathcal{L} (\cdot, T_0) \varphi + \mathcal{D}(\cdot,T_0)\mathcal{E}}_{ L^8L^4 }\cdot \norm{\xi}_{L^8L^4},
\end{align*}
which yields
\[
	\norm{\xi}_{L^8L^4} \le  2 \norm{\mathcal{L} (\cdot, T_0) \varphi+\mathcal{D}(\cdot,T_0)\mathcal{E} }_{ L^8L^4 } .
\]
By a continuity argument, we deduce the desired bounds on $(T_0,T_1)$.

Finally, for any $\theta\in[0,1]$, we may obtain the estimate
\begin{align*}
	\norm{\xi}_{\mathrm{Stz}^{\theta}} \le C \mathcal{N}_{\theta}  
	+ 
	 4^2 C^2  \norm{\mathcal{L} (\cdot, T_0) \varphi + \mathcal{D}(\cdot,T_0)\mathcal{E}}_{ L^8L^4 }\mathcal{N}_{\frac12}\cdot \norm{\xi}_{\mathrm{Stz}^{\theta}},
\end{align*}
which implies the desired bound $\norm{\xi}_{\mathrm{Stz}^{\theta}} \le 2C \mathcal{N}_{\theta}.$

(ii) To begin, an application of Strichartz yields
\[
	\norm{ \mathcal{L} (\cdot, T_0) \varphi }_{\mathrm{Stz}^{\frac12} (\R)} \le C \mathcal{N}_{\frac12}.
\]
Further,  there exists $\delta>0$ such that
\[
	\norm{\xi}_{\mathrm{Stz}^{1/2} (T_0,T_0+\delta)} \le 100 C  \mathcal{N}_{1/2} 
\]
with the same constant as above.
Hence. on the same interval,
\begin{align*}
	\norm{\xi}_{\mathrm{Stz}^{1/2}} 
	&{}\le C \mathcal{N}_{1/2}  +
	C \norm{\xi}_{\mathrm{Stz}^{1/2}} \norm{\xi}_{L^8L^4} (\norm{\xi}_{\mathrm{Stz}^{1/2}} + \norm{z}_{L^\I_t}) \\
	&{}\le C \mathcal{N}_{1/2} + 
	100 C^2  \norm{\xi}_{L^8L^4} \mathcal{N}_{1/2} \times \norm{\xi}_{\mathrm{Stz}^{1/2}} .
\end{align*}
The second term on the right can be absorbed by the left by assumption.  Thus
\[
\norm{\xi}_{\mathrm{Stz}^{1/2}} \le 2C \mathcal{N}_{\frac12}.
\]
By a continuation argument, the last estimate is true on the interval $(T_0,T_1)$.
Then, we have
\begin{align*}
 \norm{\xi - \mathcal{L} (\cdot, T_0) \varphi - \mathcal{D}(\cdot,T_0)\mathcal{E}}_{L^8_tL^4_x} &{}\lesssim \norm{\xi}_{L^8L^4}^2 (\norm{\xi}_{L^8L^4} + \norm{z}_{L^\I} ) \\
&{}\lesssim  \norm{\xi}_{L^8L^4} (\norm{\xi}_{L^8L^4} \mathcal{N}_{1/2})\ll  \norm{\xi}_{L^8L^4}
\end{align*}
by assumption. Consequently, 
\[
 \norm{ \mathcal{L} (\cdot, T_0) \varphi +  \mathcal{D}(\cdot,T_0)\mathcal{E}}_{L^8_tL^4_x(T_0,T_1)}\sim   \norm{\xi}_{L^8_tL^4_x(T_0,T_1)}.
\]
Note that the same estimate is valid in any subinterval of the form $(T_0,T)$.

Now fix $T \in (T_0,T_1)$.
By Strichartz,
\begin{align*}
&\norm{\xi - \mathcal{L} (\cdot, T_0) \varphi -  \mathcal{D}(\cdot,T_0)\mathcal{E}}_{\mathrm{Stz}^{\theta} (T_0,T) }  \\
&{}\lesssim  \norm{\xi}_{\mathrm{Stz}^{\theta}(T_0,T)} \norm{\xi}_{L^8L^4(T_0,T)} (\norm{\xi}_{L^8L^4(T_0,T)}+ \norm{z}_{L^\I_t})\\
&{}\lesssim  \norm{\xi}_{\mathrm{Stz}^{\theta}(T_0,T)} \norm{ \mathcal{L} (\cdot, T_0) \varphi +  \mathcal{D}(\cdot,T_0)\mathcal{E}}_{L^8_tL^4_x(T_0,T_1)}\\
&\quad\times (\norm{ \mathcal{L} (\cdot, T_0) \varphi +  \mathcal{D}(\cdot,T_0)\mathcal{E}}_{L^8_tL^4_x(T_0,T_1)}+\mathcal{N}_0).
\end{align*}
Furthermore,
\begin{align*}
\| \mathcal{L}(\cdot,T;z) (\xi(T) - & \mathcal{L} (T, T_0) \varphi - \mathcal{D}(T,T_0)\mathcal{E} )\|_{\mathrm{Stz}^{\theta}(T,T_1)} \\
&\lesssim \norm{ \xi(T) - \mathcal{L} (T, T_0) \varphi  - \mathcal{D}(T,T_0)\mathcal{E}}_{H^\theta} \\
&\le \norm{\xi - \mathcal{L} (\cdot, T_0) \varphi  - \mathcal{D}(T,T_0)\mathcal{E}}_{\mathrm{Stz}^{\theta} (T_0,T) }.
\end{align*}
Thus we obtain
\begin{align*}
&\norm{\mathcal{L}_{>}(z,T)\xi -  \mathcal{L} (\cdot, T_0) \varphi - \mathcal{D}(\cdot,T_0)( {\bf 1}_{(-\I,T)}\mathcal{E} )}_{\mathrm{Stz}^\theta (T_0,,T_1)}\\
&\lesssim \mathcal{N}_\theta \norm{ \mathcal{L} (\cdot, T_0) \varphi +  \mathcal{D}(\cdot,T_0)\mathcal{E}}_{L^8_tL^4_x(T_0,T_1)}
\\ &\quad\times (\norm{ \mathcal{L} (\cdot, T_0) \varphi +  \mathcal{D}(\cdot,T_0)\mathcal{E}}_{L^8_tL^4_x(T_0,T_1)}+\mathcal{N}_0).
\end{align*}
Now let $T' \in (T_0,T)$.  We then have 
\begin{align*}
&\norm{ \mathcal{L}_{>}(z,T')(\xi- \mathcal{D}(\cdot,T_0)\mathcal{E} ) - \mathcal{L}_{>}(z,T_0)(\xi- \mathcal{D}(\cdot,T_0)\mathcal{E}) }_{L^8L^4 (T_0,T_1)}  \\
&\quad\le \norm{  \mathcal{D} (\cdot, T_0) ({\bf 1}_{[T_0,T')} \tilde{N} )}_{L^8L^4 (T_0,T_1)}  \\
&\quad\lesssim \norm{\xi}_{L^8L^4[T_0,T')} (\norm{\xi}_{L^8L^4[T_0,T_1)} \mathcal{N}_{1/2})\\
&\quad \ll \norm{\xi}_{L^8L^4(T_0,T)} \\
&\quad\lesssim \norm{\mathcal{L} (\cdot, T_0) \varphi +\mathcal{D}(\cdot,T_0)\mathcal{E}  }_{L^8L^4(T_0,T)}.
\end{align*}
Taking the supremum with respect to $T' \in (T_0,T)$, we obtain
\[
\norm{\xi - \mathcal{D}(\cdot,T_0)\mathcal{E}}_{[z;T_0,T,T_1]} \ll \norm{\mathcal{L} (\cdot, T_0) \varphi +\mathcal{D}(\cdot,T_0)\mathcal{E}  }_{L^8L^4(T_0,T)}.
\]
\end{proof}

We turn to stability for the nonlinear equation.

\begin{lemma}[Stability, I]\label{l:6.3}
Let $-\I < T_0 < T_1 <  \I$ and let $z_0,z_1  \in \mathrm{SBC}$.  Suppose $\xi_0,\xi_1 \in C_t P_c H^1 ([T_0,T_1]) $ solve
\[
(	i \d_t + H ) \xi_j = B[z_j] \xi_j + \tilde{N}(z_j,\xi_j) + P_c \mathcal{E}_j
\]
on $[T_0,T_1]$, respectively, where $\mathcal{E}_j \in L^{\frac83} L^{\frac43} ([T_0,T_1])$ is given. Suppose further that
\[
	\norm{ \mathcal{L} (\cdot, T_0,z_0) \xi_0(T_0) + \mathcal{D}(\cdot,T_0;z_0) \mathcal{E}_0 }_{L^8 L^4 ([T_0,T_1])} \le \tilde{\delta}
\]
for some small $\tilde{\delta}>0$.  If
\begin{align*}
	\|\mathcal{L} (\cdot, T_0,z_0) (\xi_0(T_0)  -  \xi_1(T_0)) + \mathcal{D}(\cdot,T_0;z_0) (\mathcal{E}_0&  - \mathcal{E}_1) \|_{L^8L^4 ([T_0,T_1])}
	\\ & +  \norm{ z_0 - z_1 }_{L^8 (T_0,T_1) } \le {\delta}	,
\end{align*}
then the following hold:
\begin{align*}
&\norm{ \xi_0 - \xi_1 }_{L^8L^4([T_0,T_1])} \lesssim \delta,\\
&\norm{ \xi_0 - \xi_1 - \mathcal{D}(\cdot,T_0;z_0) (\mathcal{E}_0 - \mathcal{E}_1) }_{ [z,T_0,T_1;T_2] } \lesssim  \tilde{\delta} \delta.
\end{align*}
\end{lemma}
\begin{proof}
Set
\[
\mathrm{Dist}(t) =\sum_{j=0}^1 (-1)^j (\xi_j(t) - \mathcal{L} (t, T_0,z_0) \xi_j(T_0)  - \mathcal{D}(t,T_0;z_0) \mathcal{E}_j).
\]
By the Duhamel formula for the linearized equation, we obtain
\[
\xi_0 = \mathcal{L}(\cdot,T_0;z_0) \xi_0(T_0) + \mathcal{D}(\cdot, T_0,z_0)  \tilde{N}(z_0,\xi_0) + \mathcal{D}(\cdot, T_0,z_0) \mathcal{E}_0
\]
and
\[
\xi_1 = \mathcal{L}(\cdot,T_0;z_0) \xi_1(T_0) + [\mathcal{D}(T_0,z_0) ( \tilde{N}(z_1,\xi_1) +(B[z_1]-B[z_0])\xi_1)]+ \mathcal{D}(\cdot, T_0,z_0) \mathcal{E}_1.
\]
By the non-admissible Strichartz estimate (Lemma~\ref{L:NAS}),
\[
\norm{ \mathrm{Dist} }_{ L^8 L^{4} }\lesssim \norm{ (\tilde{N}(z_0,\xi_0)-\tilde{N}(z_1,\xi_1) ) +  (B[z_0]-B[z_1]) \xi_1 }_{ L^{8/3}L^{4/3}+ L^{4}L^{6/5} }.
\]
Using the estimate
\begin{align*}
\norm{ (\tilde{N}(z_0,\xi_0)-\tilde{N}(z_1,\xi_1) ) +  (B[z_0]-B[z_1]) \xi_1 }_{ L^{8/3}L^{4/3}+ L^{4}L^{6/5}} \\
\lesssim (\norm{\xi_0 -\xi_1 }_{ L^8 L^{4} } + \norm{z_0-z_1}_{L^8})( \norm{\xi_0}_{L^8L^4}+\norm{\xi_1}_{L^8L^4})\\
\quad\times(1+ \norm{\xi_0}_{L^8L^4}+\norm{\xi_1}_{L^8L^4}),
\end{align*}
we obtain 
\[
\norm{ \mathrm{Dist} }_{ L^8 L^{4} (T_0,T_1) }\lesssim  \tilde{\delta}\norm{\xi_0 -\xi_1 }_{ L^8 L^{4} (T_0,T_1) } +  \delta \tilde{\delta}.
\]
If $\tilde{\delta} $ is sufficiently small, then we have
\begin{equation}\label{e:6.3pf1}
\begin{aligned}
\|\xi_0 &-\xi_1\|_{L^8L^{4} (T_0,T_1)}\\
& \lesssim \norm{\mathcal{L} (t, T_0,z_0) (\xi_0-\xi_1)(T_0)+ \mathcal{D}(\cdot, T_0;z_0)(\mathcal{E}_0-\mathcal{E}_1) }_{L^4L^{24/7}(T_0,T_1)} +\delta \tilde{\delta} \lesssim \delta.
\end{aligned}
\end{equation}

Next, note that
\begin{multline*}
	\mathcal{L}(t,T_1;z_0) \xi_0(T_1) = \mathcal{L}(t,T_0;z_0) \xi_0(T_0)+\mathcal{D}(\cdot, T_0;z_0)({\bf 1}_{(-\I,T_1]}\mathcal{E}_0) \\
	+ \mathcal{L}(t,T_1;z_0)\mathcal{D}(T_1, T_0,z_0)  \tilde{N}(z_0,\xi_0)
\end{multline*}
and
\begin{multline*}
	\mathcal{L}(t,T_1;z_0) \xi_1(T_1) = \mathcal{L}(t,T_0;z_0) \xi_1(T_0)+\mathcal{D}(t, T_0;z_0)({\bf 1}_{(-\I,T_1]}\mathcal{E}_1)\\
	+ \mathcal{L}(t,T_1;z_0) \mathcal{D}(T_1,T_0,z_0) ( \tilde{N}(z_1,\xi_1) +(B[z_1]-B[z_0])\xi_1).
\end{multline*}
Hence,
\begin{align*}
	&\mathcal{L}(t,T_1;z_0) (\xi_0(T_1)-\xi_1(T_1)) \\
 	&= \mathcal{L}(t,T_0;z_0) (\xi_0(T_0)-\xi_1(T_0)) + \mathcal{D}(\cdot, T_0;z_0)({\bf 1}_{(-\I,T_1]}(\mathcal{E}_0-\mathcal{E}_1))\\
 	&{}\quad +  \mathcal{D}(t,T_0,z_0) ({\bf 1}_{[T_0,T_1]} ( \tilde{N}(z_1,\xi_1) +(B[z_1]-B[z_0])\xi_1)).
\end{align*}
We now define 
\begin{align*}
\mathrm{Diff}&:= \mathcal{L}_{>}(t,T_1;z_0) (\xi_0-\xi_1)-\mathcal{L}(t,T_0;z_0) (\xi_0-\xi_1)(T_0) - \\
& \quad - \mathcal{D}(\cdot, T_0;z_0)({\bf 1}_{(-\I,T_1]}(\mathcal{E}_0-\mathcal{E}_1)) \\
&= \mathcal{L}_{>}(t,T_1;z_0) (\xi_0-\xi_1 - \mathcal{D}(\cdot, T_0;z_0)(\mathcal{E}_0-\mathcal{E}_1)\\
&\quad - \mathcal{L}_{>}(t,T_0;z_0) (\xi_0-\xi_1 - \mathcal{D}(\cdot, T_0;z_0)(\mathcal{E}_0-\mathcal{E}_1))
\end{align*}
for $t\ge T_0$.
We use this form for $t\ge T_1$ and the above formulas for $\xi_0$ and $\xi_1$ for $t\in[T_0,T_1]$ to get
\[
	\mathrm{Diff}
	=    \mathcal{D}(T_0,z_0) ({\bf 1}_{[T_0,T_1]} ( \tilde{N}(z_1,\xi_1) +(B[z_1]-B[z_0])\xi_1)).
\]
for $t\ge T_0$.
Thus, by the same nonlinear estimate as for $\norm{ \mathrm{Dist} }_{ L^8 L^{4}(T_0,T_1) }$, we may obtain
\[
\norm{\mathrm{Diff}}_{L^8L^{4}([T_0,T_2])} \lesssim  \tilde{\delta}\norm{\xi_0 -\xi_1 }_{ L^8 L^{4}  (T_0,T_1)} +  \delta \tilde{\delta} \lesssim \delta \tilde{\delta}.
\]
Taking the supremum with respect to $T_1$, we obtain the result.
\end{proof}

We upgrade Lemma~\ref{l:6.3} as follows:

\begin{proposition}[Stability, II] \label{l:6.4}
Let $-\I < T_0 < T_1 < T_2 \le \I$, $z_0 ,z_1 \in \mathrm{SBC}$,
and $\xi_0,\xi_1 \in C([T_0,T_1]; H^{3/4})$ solve
\[
i \d_t\xi_j + H  \xi_j = B[z_j] \xi_j + \tilde{N}(z_j,\xi_j) + \mathcal{E}_j
\]
on $[T_0,T_1]$, respectively, where $\mathcal{E}_j \in L^{8/3}L^{4/3} ([T_0,T_1])$ are given. Let 
\[
	 \mathcal{M} := \norm{\xi_0}_{L^8L^4 (T_0,T_1)} .
\] 
There exists $\delta_* ( \mathcal{M})>0$ and $C_* (\mathcal{M})>0$
such that if 
\begin{align*}
	\delta:=&\norm{ \mathcal{L} (\cdot, T_0,z_0) (\xi_0-\xi_1)(T_0)  + \mathcal{D} (\cdot, T_0,z_0) (\mathcal{E}_0-\mathcal{E}_1)  }_{L^8L^4 ([T_0,T_1])} \\
	&+  \norm{ z_0 - z_1 }_{L^8 (T_0,T_1) } \le \delta_*,
\end{align*}
then we have
\[
	\norm{ \xi_0 - \xi_1 - \mathcal{D} (\cdot, T_0,z_0) (\mathcal{E}_0-\mathcal{E}_1)}_{ [z_0;T_0,T_1;T_2] } \le C_* \delta.
\]
\end{proposition}
\begin{proof}
For any $N \in \N$, $(T_0,T_1)$ may be decomposed into subintervals $I_0,\dots, I_N$ of the  form $I_j = (S_j,S_{j+1})$
such that
\[
\sup_{j\in [1,N]} \norm{\xi_0}_{L^4L^6(I_j)} \le 2 N^{-\frac18} \mathcal{M} =: \tilde{\delta}.
\]
If $N$ is chosen sufficiently large, we see that
\[
	\norm{ \mathcal{L}(\cdot, S_j;z_0)\xi_0(S_j) + \mathcal{D}(\cdot,S_j;z_0) \mathcal{E}_0 }_{L^8 L^4(I_j)} \sim \norm{\xi_0}_{L^8L^4 (I_j)} \le \tilde{\delta}.
\]
Now let $ \delta_0 \in (0,\tilde\delta]$ to be chosen later. Suppose that
\begin{align*}
&\norm{ \mathcal{L} (\cdot, S_0,z_0) (\xi_0-\xi_1)(T_0)  + \mathcal{D} (\cdot, S_0,z_0) (\mathcal{E}_0-\mathcal{E}_1)  }_{L^8 L^4 ([S_0,T_1])} \\
&+  \norm{ z_0 - z_1 }_{L^8 (S_0,T_1) } \le \delta_0.
\end{align*}
Then we have
\[
\norm{ \mathcal{L}(\cdot, S_0, z_0) \xi_1(T_0) +\mathcal{D} (\cdot, S_0,z_0)\mathcal{E}_1 }_{L^8 L^4(I_0)} \lesssim \tilde{\delta}
\]
and we can apply Lemma~\ref{l:6.3} to obtain
\[
\norm{ \xi_0 - \xi_1- \mathcal{D} (\cdot, S_0,z_0) (\mathcal{E}_0-\mathcal{E}_1) }_{[z_0;S_0,S_1;T_2]} \le C_0\delta_0 \tilde{\delta}. 
\]
Now observe that
\begin{align*}
\| \mathcal{L} (\cdot, S_1,z_0) (\xi_0-\xi_1)(S_1)  + \mathcal{D} (\cdot, S_1,z_0) (\mathcal{E}_0&-\mathcal{E}_1)  \|_{L^8L^4 ([S_1,T_1])} \\
	&+  \norm{ z_0 - z_1 }_{L^8 (S_1,T_1) } \le \delta_1,
\end{align*}
where $\delta_1 := \delta_0 + C_0\delta_0 \tilde{\delta}.$  Now if $\delta_1 \le \tilde{\delta}$, then we repeat this argument to obtain
\[
\norm{ \xi_0 - \xi_1- \mathcal{D} (\cdot, S_1,z_0) (\mathcal{E}_0-\mathcal{E}_1) }_{[z_0;S_1,S_2;T_2]} \le C_0\delta_1 \tilde{\delta}.
\]
If $\delta_0$ is sufficiently small then this argument can be repeated up to $I_N$.  Setting $\delta_j = (1+ C_0 \tilde{\delta})^j \delta_0$, we find
\begin{align*}
	&\norm{ \xi_0 - \xi_1- \mathcal{D} (\cdot, T_0,z_0) (\mathcal{E}_0-\mathcal{E}_1) }_{[z_0;T_0,T_1;T_2]} \\
	&\le \sum_{j=0}^N \norm{ \xi_0 - \xi_1 - \mathcal{D} (\cdot, S_j,z_0) (\mathcal{E}_0-\mathcal{E}_1) }_{[z_0;S_j,S_{j+1};T_2]} \\
	&\le C_0 \tilde{\delta} \sum_{j=0}^N \delta_j = ((1+ C_0 \tilde{\delta})^{N+1}-1) \delta_0
\end{align*}
for sufficiently small $\delta_0$. 
\end{proof}

\section{Linear profile decomposition}\label{S:LPD}

In this section, we introduce a linear profile decomposition adapted to the linearized equation 
\[
\begin{cases} i\partial_t v + Hv = B[z]v, \\ v(s)=\varphi.\end{cases}
\]
Recall that we denote the solution to this equation by $v(t)=\mathcal{L}(t,s;z)$. 

We first introduce some notation and a few definitions. For $y \in \R^3$, let $T_y$ be the translation operator
\[
	(T_y f)(x) = f(x-y),\quad f:\R^3\to\C.
\]
Note that $T_y $ is not a map from $P_cH^1$ to itself when $y\neq0$.

Next, we say a sequence $ \{(y_n , s_n)\} \subset \R^3 \times \R$ is a \emph{parameter of shifts} 
\[
\qtq{either} y_n\equiv 0 \qtq{or}\lim_{n\to\I} |y_n|=\infty,\quad\text{and}
\]
\[
\qtq{either} s_n\equiv 0 \qtq{or}\lim_{n\to\infty}|s_n|=\infty.
\]

Finally, let $\{  z_n\}\subset\mathrm{SBC}$. We say a sequence of functions $\{\l_n\}\subset P_c H^1$ is a \emph{linear profile associated with $\{  z_n\}$} if
\[
	\lambda_n  =  \mathcal{L} (0, s_n ;z_n)P_c T_{y_n} \varphi
\]
for some $\varphi \in H^1$  and some parameter of shifts
$\{ (y_n,s_n)\}$.

The main result of this section is the following theorem: 

\begin{theorem}[Linear Profile Decomposition]\label{t:lpd}
Let $\{z_n\}_n\subset \mathrm{SBC}$ and let $\{ \psi_n \}_n \subset P_cH^1$ be a bounded sequence. Passing to a subsequence, there exists $J^* \in \N \cup \{0,\I\}$ and linear profiles $\{ \l^j_n \}_n$ associated with $\{z_n\}_n$ with shifts $\{(y_n^j, s_n^j)\}_n$ such that the following decomposition holds for each finite $0\leq J \leq J^*$:
\begin{equation}\label{e:lpd}
\psi_n  = \sum_{j=0}^{J} \lambda_n^j  + \gamma_n^J,\qtq{with} \lambda_n^j  :=  \mathcal{L} (0, s_n^j; z_n)P_c T_{y_n^j} \varphi_\I^j.
\end{equation}
This decomposition satisfies the following properties:
\begin{itemize}\item The profiles and shift parameters obey the following: 
\begin{enumerate}
\item  Writing 
\[
y^j := \lim_{n\to\I} |y^j_n| \in \{0,\I \}\qtq{and} s^j := \lim_{n\to\I} |s^j_n| \in \{ 0,\I \},
\]
we have $(y^0,s^0)=(0,0)$ and $(y^j,s^j)\neq (0,0)$ for $j\ge1$.

\smallskip

\item For each $j\neq k$, we have the following orthogonality of parameters:
\[
\lim_{n\to\infty} \bigl[|s_n^j-s_n^k|+ |y_n^j-y_n^k|\bigr]=\I.
\]

\smallskip

\item For each $j$, $T_{y_n}^{-1} \mathcal{L}(s_n^j,0;z_n) \psi_n \rightharpoonup \varphi_\I^j$ weakly in $H^1$.

\smallskip

\item We have $\varphi_\I^j \neq 0$ for $0<j \le J^*$ and $\varphi_\I^j = 0$ for $j > J^*$.  Furthermore, if $y^j=0$ then $\varphi_\I^j \in P_c H^1$.

\end{enumerate} 
\item We have the following mass/energy decoupling:
\begin{align}\label{e:Mdecomp}
	&\text{for each }J,\quad \mathbb{M}(\psi_n)
	= \sum_{j=0}^{J} \mathbb{M}(\lambda_n^j)
	+ \mathbb{M} (\gamma_n^J)  + o(1)\qtq{as}n\to\infty.\\
	\label{e:H0decomp}
	&\text{for each }J,\quad \mathbb{H}_0(\psi_n)
	= \sum_{j=0}^{J} \mathbb{H}_0(\lambda_n^j)
	+ \mathbb{H}_0 (\gamma_n^J)  + o(1)\qtq{as}n\to\infty.
\end{align}
\begin{equation}\label{e:Gdecomp2}
	\lim_{n\to\I} \mathbb{G}(\psi_n)
	= \mathbb{G}(\varphi^0_\I) + \sum_{j=1}^\infty \mathbb{G}(\varphi^{j}_\I) {\bf 1}_{\{ s^j=0 \}} .
\end{equation}
In addition, we have
\begin{equation}\label{e:Vdecomp}
	\lim_{n\to\I} \int V(x) |\psi_n(x)|^2 dx = \int V(x)|\varphi_\I^0(x)|^2 dx,
\end{equation}
so that
\begin{equation}\label{e:energydecomp}
	\varlimsup_{n\to\I} \mathbb{E}_V (\psi_n)
	\ge \sum_{j=0}^{\I} \mathbb{E} (\varphi^j_\I ;y^j, s^j),
\end{equation}
where
\[
		\mathbb{E}(\psi; y , s) := \left\{
			\begin{aligned}
			& \mathbb{E}_V(\psi) && (y,s)=(0,0),\\
			& \lim_{n\to\I} \mathbb{H}_0(\mathcal{L}(0,s_n^j;z_n) \psi) && (y,s)=(0,\infty),\\
			& \mathbb{E}_0(\psi) && (y,s)=(\infty,0),\\
			& \mathbb{H}_0(\psi) && (y,s)=(\infty,\infty).
			\end{aligned}
		\right.
\]
\item Finally, the remainder satisfies the following:
\[
	\sup_J \varlimsup_{n \to \I} \norm{  \mathcal{L}(\cdot,0;z_n) \gamma_n^J}_{ \mathrm{Stz}^1(\R) } <\I,
\]
\begin{equation}\label{e:remainder}
	\lim_{J\to \infty} \varlimsup_{n \to \I} \norm{ \mathcal{L}(\cdot,0;z_n) \gamma_n^J}_{L^\I_t L^4_x (\R) } =0.
\end{equation}
\end{itemize}
\end{theorem}

Before beginning the proof of Theorem~\ref{t:lpd}, we collect several preliminary results related to the effect of space translations on the linear propagation of the radiation. These results will be used both for the proof of Theorem~\ref{t:lpd} and in later sections.

\begin{lemma}\label{l:spacetranslation}
Let $\{z_n\}_n \subset \mathrm{SBC}$, and let $\{y_n\}$ be a sequence such that $|y_n|\to\I$ as $n\to\I$.  Then, for any $\varphi \in P_c H^1$ and any $\tau>0$,
\[
	T_{y_n}^{-1} \mathcal{L} (t,0;z_n) P_c T_{y_n}  \varphi \to e^{-it\Delta} \varphi \qtq{in}\mathrm{Stz}^1([-\tau,\tau])\qtq{as}n\to\infty.
\]
\end{lemma}
\begin{proof} By a density argument, we may assume that $\varphi \in \mathcal{S}$.  We fix $\tau>1$ and set
\begin{align*}
v_n(t) & := \mathcal{L} (t,0;z_n) P_c T_{y_n}\varphi
\end{align*}
for $t\in[-\tau,\tau]$. Observe that
\begin{equation}\label{e:spacepf1}
	\norm{v_n}_{L^\I_t H^1} \lesssim \norm{P_c T_{y_n} \varphi}_{H^1} \lesssim \norm{ \varphi}_{H^1}
\end{equation}
uniformly in $n$.  

The main step in the proof will be to establish $H^1$-smallness for $v_n$ far from $x=y_n$.  Combining this with the fact that $|y_n|\to\infty$, we can show that the contribution of the potential terms (i.e. the terms with $V$ and $B[z_n]$) are ultimately negligible.  To this end, we let $\chi_{n,r}(x)=\chi(\tfrac{x-y_n}{r})$ be a smooth function obeying
\[
\chi_{n,r}(x) = \begin{cases} 0 & |x-y_n|<r \\ 1 & |x-y_n|>2r \end{cases}\qtq{and} \|\nabla \chi_{n,r}\|_{L^\infty}\lesssim \tfrac{1}{r}.
\]

%

Using the equation for $v_n$, we have 
\begin{align*}
	\d_t |v_n|^2 = 2\nabla \cdot \Im (\bar v_n \nabla v_n) + 2\Im ( \bar v_n B[z_n] v_n ),
\end{align*}
so that
\begin{equation}\label{e:spacepf2}
	\d_t \int \chi_{n,r}  |v_n|^2\,dx = -2\int \nabla \chi_{n,R} \cdot \Im (\bar v_n \nabla v_n)\,dx
	+2\int \chi_{n,r} \Im ( \bar v_n B[z_n] v_n )\,dx.
\end{equation}

By \eqref{e:spacepf1}, we first have
\[
	\sup_{t\in[ -\tau,\tau]}\abs{\int \nabla \chi_{n,r} \cdot \Im (\bar v_n \nabla v_n)\,dx} \lesssim \tfrac1r \norm{\varphi}_{H^1}^2.
\]

To estimate the second term, we first note that for each $t \in [-\tau,\tau]$,
\[
	R[z_n] v_n = v_n +  c_n \phi_0
\]
for $c_n\in\C$ obeying
\[
	|c_n| \lesssim \norm{v_n D_1\Phi[z_n] }_{L^1} + \norm{v_n D_2\Phi[z_n] }_{L^1}
	\lesssim \( \int \chi_{n,r} |v_n|^2\, dx \)^\frac12 + o_n(1)
\]
for any fixed $r>0$ as $n\to\I$.  As $|y_n|\to\infty$, we obtain 
\begin{align*}
	\norm{   |\Phi[z_n]|^2 R[z_n]v_n  }_{L^2} 
	&{}\le \norm{   |\Phi[z_n]|^2 \chi_{n,r} R[z_n]v_n  }_{L^2} +
	\norm{   |\Phi[z_n]|^2 (1-\chi_{n,r})R[z_n]v_n}_{L^2} \\
	&{}\lesssim \( \int \chi_{n,r} |v_n|^2\,dx \)^\frac12 + |c_n| + \norm{|\Phi[z_n]|^2 (1-\chi_{n,r})}_{L^\I}\\
	&{}\lesssim \( \int \chi_{n,r} |v_n|^2\,dx \)^\frac12 + o_n(1)\qtq{as}n\to\infty,
\end{align*}
which in turn yields 
\[
\norm{  B[z_n(t)]v_n(t) }_{L^2} \lesssim \( \int \chi_{n,r} |v_n|^2\,dx \)^\frac12 + o_n(1)\qtq{as}n\to\infty
\]
for each fixed $r>0$. We thus obtain
\begin{equation}\label{setup-gronwall1}
\biggl| \partial_t \int \chi_{n,r}|v_n(t)|^2\,dx\biggr| \lesssim \tfrac{1}{r}+o_n(1)+\int \chi_{n,r}|v_n(t)|^2\,dx.
\end{equation}

We will estimate $\d_j v_n$ for $j\in\{1,2,3\}$ in a similar fashion, relying on the equation
\begin{equation}\label{pde-for-djv}
i\partial_t \partial_j v_n - \Delta \partial_j v_n = \partial_j[-Vv_n + B[z_n]v_n].
\end{equation}

We first claim that
\begin{equation}\label{jfirstlyclaim}
\|v_n\|_{L_t^\infty H^2} \lesssim \|\varphi\|_{H^2}. 
\end{equation}
To see this, we use the following Duhamel formula for $v_n$:
\[
	v_n (t) = e^{itH} v_n(0) -i \int_0^t e^{i(t-s)H} B[z_n(s)] v_n(s) ds,
\]
which firstly implies
\[
\|Hv_n\|_{L_t^\infty L_x^2\cap L_t^2 L_x^6} \lesssim \|\varphi\|_{H^2} + \|HB[z_n]v_n\|_{L_t^2 L_x^{\frac65}}.
\]
As we may use the definition of $B[\cdot]$ and Proposition~\ref{Prop:Strichartz} to obtain 
\[
\|HB[z_n]v_n\|_{L_t^2 L_x^{\frac65}} \lesssim c\|Hv_n\|_{L_t^2 L_x^6} + \|v_n\|_{\mathrm{Str}^1} \lesssim c\|Hv_n\|_{L_t^2 L_x^6} + \|\varphi\|_{H^1}
\]
for some $0<c\ll 1$ (depending on the upper bound of $|z_n|$), the claim \eqref{jfirstlyclaim} follows. 

We now use \eqref{pde-for-djv} to write
\[
	\d_t | \d_j v_n |^2 = 2 \nabla \cdot \Im (\overline{ \d_j v_n}\nabla \d_j v_n)
	+ 2\Im (\overline{\d_j v_n} (\d_j V) v_n) - 2 \Im ( \overline{\d_j v_n} \d_j (B[z_n] v_n) ),
\]
which implies
\begin{align*}
	\d_t \int \chi_{n,r}  | \d_j v_n |^2\,dx &= -2 \int \nabla \chi_{n,r} \cdot \Im (\overline{ \d_j v_n}\nabla \d_j v_n)\,dx \\
&\quad + 2 \int \chi_{n,r} \Im (\overline{\d_j v_n} (\d_j V) v_n)\,dx \\
&\quad - 2 \int \chi_{n,r}\Im ( \overline{\d_j v_n} \d_j (B[z_n] v_n) )\,dx.
\end{align*}
Arguing as above, we find that
\begin{equation}\label{setup-gronwall2}
\abs{\d_t \int \chi_{n,r}  | \d_j v_n |^2 \,dx }\lesssim \tfrac{1}{r} \norm{\varphi}_{H^2}^2 + \int \chi_{n,r}  | \d_j v_n |^2 \,dx + \int \chi_{n,r}  |  v_n |^2 \,dx + o_n(1).
\end{equation}

If we now define
\[
f_n(t)=\int \chi_{n,r}\bigl[|v_n(t,x)|^2+|\nabla v_n(t,x)|^2\bigr]\,dx,
\]
then \eqref{setup-gronwall1} and \eqref{setup-gronwall2} imply
\[
f_n(t) \le f_n(0) + t[\tfrac{C}{r}+o_n(1)]+C\int_0^t f_n(s)\,ds
\]
for some constant $C>0$ (independent of $t,\tau,r,n$). Thus, Gronwall's inequality implies 
\[
\sup_{t\in [-\tau,\tau]} f_n(t) \le e^{C\tau} \( \int \chi_{n,r} |\varphi|^2\,dx  + \tau[\tfrac{C}{r}+o_n(1)]\).
\]
In particular, given $\eps>0$ and $\tau>0$, if we choose $r=r(\tau,\eps)$ sufficiently large, then we obtain that 
\begin{equation}\label{gronwall-small}
	\sup_{t\in [-\tau,\tau]} \int \chi_{n,r} \bigl[|v_n(t,x)|^2+|\nabla v_n(t,x)|^2\bigr]\,dx \le \eps\qtq{for all}n\qtq{sufficiently large.}
\end{equation}

Using \eqref{gronwall-small} and the fact that $|y_n|\to\infty$, we can therefore deduce that
\begin{equation}\label{e:spacepf3}
\begin{aligned}
\lim_{n\to\I}\norm{ V v_n }_{L^1_t H^1_x ([-\tau,\tau])} &= 0, \\
\lim_{n\to\I}\norm{ B[z_n] v_n }_{L^1_t H^1_x ([-\tau,\tau])} &=0.
\end{aligned}
\end{equation}

We are now in a position to establish the desired convergence.  In particular, our goal is to show that $T_{y_n}^{-1} v_n(t)\to e^{-it\Delta}\varphi$ in $\mathrm{Stz}^1([-\tau,\tau])$.  We begin with the following Duhamel formula:
\begin{equation}\label{e:spacepf4}
T_{y_n}^{-1}v_n(t) = e^{-it \Delta}  T_{y_n}^{-1}P_c T_{y_n}\varphi  - i \int_0^t e^{-i(t-s) \Delta } T_{y_n}^{-1} (-V v_n+ B[z_n]v_n)(s)\,ds.
\end{equation}
We now observe that
\[
	\norm{T_{y_n}^{-1} P_c T_{y_n}\varphi -  \varphi }_{H^1} =  \norm{\phi_0 }_{H^1} |(\phi_0, T_{y_n} \varphi)| \to 0\qtq{as}n\to\infty.
\]
Combining this with Strichartz estimates for the free Schr\"odinger equation and \eqref{e:spacepf3}, we deduce
\[
	\lim_{n\to\infty}\norm{T_{y_n}^{-1}v_n(t) - e^{-it \Delta}\varphi }_{\mathrm{Stz}^1 ([-\tau,\tau])} = 0.
\]
\end{proof}

We use the assumption that  $V\in H_{\frac32}^1$ in the following lemma.

\begin{lemma}\label{l:spacelargetime}
Suppose that $\{z_n\}\subset \mathrm{SBC}$ and $\{y_n\}\subset\R^3$.  For any $\varphi \in H^1$, 
\begin{equation}\label{e:spacelargetimebound12}
\sup_n \|\mathcal{L}(t,0;z_n)P_c T_{y_n}\varphi\|_{X([0,\infty))} \lesssim \|e^{-it\Delta}\varphi\|_{X([0,\infty))},
\end{equation}
where
\[
X\in\{L_t^8 L_x^4\cap L_t^2 H_6^{\frac12}, L_t^2 H_6^1\}. 
\]
Moreover, the implicit constants are independent of the choice of $z_n$.  The analogous bounds hold in the negative time direction.
%
\end{lemma}
\begin{proof} Throughout the proof, we take space-time norms over $[0,\infty)\times\R^3$.  We set
\[
v_n(t) = \mathcal{L}(t,0;z_n) P_c T_{y_n} \varphi\qtq{and} w(t) = e^{-it \Delta} \varphi,
\]
and we denote
\[
d_n(t) = v_n(t) - P_c T_{y_n} w(t),
\]
which satisfies $d_n(0)=0$ and
\[
(i\d_t - H) d_n - B[z_n] d_n =- (i \d_t - H) P_c T_{y_n} w + B[z_n] (P_c T_{y_n} w).
\]
Using the equation
\begin{align*}
- (i  \d_t - H) P_c T_{y_n} w &= - (i \d_t - H) T_{y_n} w + (i \d_t - H) \phi_0 ( T_{y_n} w  ,\phi_0) \\
&= - V T_{y_n} w  - e_0 \phi_0 ( T_{y_n} w  ,\phi_0) + \phi_0 (T_{y_n}i\d_t w, \phi_0 )\\
&= - V T_{y_n} w  - e_0 \phi_0 ( T_{y_n} w  ,\phi_0) + \phi_0 (T_{y_n} w, \Delta \phi_0 ),
\end{align*}
we obtain the Duhamel formula
\begin{align*}
d_n(t)  = -i \int_0^t & \mathcal{L} (t,s;z_n) \bigl\{- V T_{y_n} w  - e_0 \phi_0 ( T_{y_n} w  ,\phi_0) \\
	&\quad + \phi_0 (T_{y_n} w, \Delta \phi_0 )+B[z_n] (P_c T_{y_n} w)\bigr\}\, ds.
\end{align*}

We now turn to the $H^{1/2}$-regularity estimates in \eqref{e:spacelargetimebound12}.  We first observe that 
\begin{align*}
\norm{ V T_{y_n} w}_{L^2H^{1/2}_{6/5} } &\lesssim \norm{V}_{H^{1/2}_{3/2}} \norm{w}_{L^2 H^{1/2}_{6}}\lesssim \norm{  w }_{L^2 H^{1/2}_6},\\
\norm{ e_0 \phi_0 ( T_{y_n} w  ,\phi_0) }_{L^2H^{1/2}_{6/5}} & \lesssim \norm{\phi_0}_{H^{1/2}_{6/5}} \norm{\phi_0}_{L^{6/5}}\norm{w}_{L^2L^6} \lesssim \norm{  w }_{L^2 H^{1/2}_6}, \\
\norm{ \phi_0 ( T_{y_n} w, \Delta\phi_0) }_{L^2H^{1/2}_{6/5}} &  \lesssim \norm{  w }_{L^2 H^{1/2}_6},\\
\norm{ B[z_n] (P_c T_{y_n} w) }_{L^2H^{1/2}_{6/5}} & \lesssim \norm{ \Phi[z_n] }_{L^\I H^1}^2 \norm{ P_c T_{y_n} w }_{L^2 H^{1/2}_6} \lesssim \norm{  w }_{L^2 H^{1/2}_6}.
\end{align*}
In particular, by Sobolev and Strichartz, we have
\[
\norm{d_n}_{L^8L^4 \cap L^2 H^{1/2}_6} \lesssim \norm{d_n}_{\mathrm{Stz}^{1/2}} \lesssim \norm{w}_{L^2 H^{1/2}_6}.
\]
Noting that 
\[
\norm{(1-P_c)T_{y_n} w}_{ L^8L^4 \cap L^2 H^{1/2}_6 }  \lesssim \norm{ w}_{ L^8L^4 \cap L^2 H^{1/2}_6},
\]
we thus obtain
\[
\norm{v_n}_{L^8L^4 \cap L^2 H^{1/2}_6} \lesssim \norm{w}_{L^8L^4 \cap L^2 H^{1/2}_6},
\]
which yields the first bound in \eqref{e:spacelargetimebound12}.  The other estimate is shown similarly. \end{proof} 
%

The next result pertains to the fact that in general, the operator $\mathcal{L}$ is not unitary.  Nonetheless, in the regime $|y_n|\to\infty$, we can recover the isometry property. 

\begin{lemma}
Let $\{z_n\}_n \subset \mathrm{SBC}$, and suppose $|y_n|\to\infty$. Then for any $\varphi \in H^1$, we have the following:
\begin{align}\label{e:spacenormflat1}
	\lim_{n\to\I} \sup_{t \in \R } \abs{ \norm{\mathcal{L}(t,0;z_n) P_c T_{y_n} \varphi}_{L^2} - \norm{\varphi}_{L^2}  }
	&=0,\\
\label{e:spacenormflat2}
	\lim_{n\to\I} \sup_{t \in \R } \abs{ \norm{\nabla \mathcal{L}(t,0;z_n) P_c T_{y_n} \varphi}_{L^2} - \norm{\nabla \varphi}_{L^2}  }
	&=0.
\end{align}\end{lemma}

\begin{proof} We set 
\[
v_n(t) =\mathcal{L}(t,0;z_n) P_c T_{y_n} \varphi
\]
and let $\eps>0$. As $e^{-it \Delta} \varphi \in L^2 H^1_6(\R\times\R^3)$ (by Strichartz), there exists $\tau>0$ such that
\begin{equation}\label{e:tau-dispersed}
\norm{ e^{-it\Delta} e^{i\tau\Delta}\varphi }_{L^2 H^1_6 ([0,\I))}= \norm{ e^{-it\Delta} \varphi }_{L^2 H^1_6 ([\tau,\I))} \le \eps.
\end{equation}

Using Lemma~\ref{L:properties-of-L}, the triangle inequality, the uniform Strichartz estimate (Proposition~\ref{Prop:Strichartz}), Lemma~\ref{l:spacelargetime}, \eqref{e:tau-dispersed}, and Lemma~\ref{l:spacetranslation} (with the $L_t^\infty H_x^1([-\tau,\tau])$ norm), we first estimate for all large $n$:
\begin{equation}\label{e:spacelargetimeboundpf2}
\begin{aligned}
\|v_n\|_{L^2 H_6^1([\tau,\infty))}& = \|\mathcal{L}(t,0,z_n) P_c T_{y_n}\varphi\|_{L^2 H_6^1([\tau,\infty)} \\
&= \|\mathcal{L}(t+\tau,0,z_n)P_cT_{y_n}\varphi\|_{L^2 H_6^1([0,\infty)} \\ 
 &= \| \mathcal{L}(t,0;z_n(\cdot+\tau))\mathcal{L}(\tau,0;z_n) P_c T_{y_n}\varphi\|_{L^2 H_6^1([0,\infty))} \\
& \lesssim \|\mathcal{L}(t,0;z_n(\cdot+\tau))P_c T_{y_n}e^{i\tau\Delta}\phi\|_{L^2 H_6^1([0,\infty))} \\
&\quad + \| \mathcal{L}(\tau,0;z_n)P_c T_{y_n}\varphi - T_{y_n} e^{i\tau\Delta}\varphi \|_{H^1} \\
& \lesssim \|e^{-it\Delta}e^{i\tau\Delta}\phi\|_{L^2 H_6^1([0,\infty))} + \eps \lesssim \eps.
\end{aligned}
\end{equation}

Next, we observe that we have the following Duhamel formula for $v_n$:
\[
v_n(t) = e^{i(t-\tau) \Delta }v_n(\tau) - i \int_\tau^t e^{-i(t-s)\Delta} (- V v_n + B[z_n] v_n)\,ds.
\]
Thus for $k\in\{0,1\}$ Strichartz and \eqref{e:spacelargetimeboundpf2} imply
\begin{align*}
	\|& v_n(t) - e^{i(t-\tau) \Delta }v_n(\tau) \|_{L^\I \dot{H}^k([\tau,\I))} 
	\\&\lesssim \norm{ V v_n }_{L^2 {H}^1_{6/5}([\tau,\I))} +  \norm{ B[z_n] v_n }_{L^2 H^1_{6/5}([\tau,\I))} \lesssim \norm{v_n}_{L^2H^1_6([\tau,\I))}\lesssim \eps.
\end{align*}
As $e^{it\Delta}$ is unitary on $\dot{H}^k$, this implies
\[
\sup_{t\ge \tau} \abs{ \norm{v_n(t)}_{\dot{H}^k} - \norm{v_n(\tau)}_{\dot{H}^k}} \lesssim \eps
\]
for all large $n$.  Thus the result follows provided we can establish
\[
\sup_{s\in[0,\tau]} \abs{ \|v_n(s)\|_{\dot H^k} - \|\varphi\|_{\dot H^k} } \lesssim \eps
\]
for all large $n$.  In fact, this follows from Lemma~\ref{l:spacetranslation} and the unitary property of the free propagator and space translation; indeed, we have
\begin{align*}
\| v_n(s) - T_{y_n} e^{-is\Delta}\varphi\|_{L_t^\infty H^1([0,\tau])}\to 0 \qtq{as}n\to\infty. 
\end{align*}
As similar estimates hold in the negative time direction, we complete the proof. \end{proof}

We next prove a result that collects the ways in which we may consider two parameters of shifts to be `orthogonal'.

\begin{lemma}[Asymptotic orthogonality]\label{l:orth}
Let $\{  z_n\}_n \subset \mathrm{SBC}$, and let  $\{ (y_n,s_n) \}_n$ and $\{ (\tilde{y}_n,\tilde{s}_n) \}_n$ be two parameters of shifts.

The following are equivalent:
\begin{enumerate}
\item $|s_n-\tilde{s}_n | + |y_n - \tilde{y}_n|\to \I$
as $n\to\I$;
\item
For any $\varphi \in H^1$
\[
T_{\tilde{y}_n}^{-1} \mathcal{L}(\tilde{s}_n,s_n ;z_n) P_c T_{y_n} \varphi  \rightharpoonup 0 \qtq{as}n\to\infty
\]
weakly in $H^1$ up to a subsequence.
\item For any subsequence $\{n_k\}_k$, there exists a bounded sequence $\{ u_k \}_k \subset P_c  H^1$ and a subsequence in $k$ such that
\[
T_{y_{n_k} }^{-1} \mathcal{L}(s_{n_k},0 ;z_{n_k})  u_k  \rightharpoonup \psi \neq 0\qtq{and} T_{\tilde{y}_{n_k} }^{-1}  \mathcal{L}(\tilde{s}_{n_k},0 ;z_{n_k})  u_k  \rightharpoonup 0
\]
weakly in $H^1$ as $k\to\infty$ along this subsequence.
\end{enumerate}
\end{lemma}

We first need the following: 
\begin{lemma}\label{l:missing}
Let $\{  z_n\}_n \subset \mathrm{SBC}$ and suppose that two parameters of shifts $\{ (y_n,s_n) \}_n$ and $\{ (\tilde{y}_n,\tilde{s}_n) \}_n$ satisfy
\[
\sup_n( |s_n-\tilde{s_n}| + |y_n-\tilde{y}_n| ) < \I.
\]
Then, up to a subsequence, the sequence of operators
\[
T_{\tilde{y}_n}^{-1} \mathcal{L} (\tilde{s}_n,s_n;z_n) P_c T_{y_n}  
\]
converges strongly to a continuous invertible operator on $H^1$.
\end{lemma}
\begin{proof} If $y_n\equiv 0$, the theorem follows from the local-in-time convergence of $z_n(\cdot+s_n)$ (cf. \cite[Equation (5.13)]{Nakanishi}).  Thus we suppose instead that $|y_n|\to\infty$.  Passing to a subsequence, we may suppose $s_n-\tilde s_n \to s_\infty\in\R$ and $y_n-\tilde y_n\to y_\infty\in\R^3$.  Without loss of generality, we may then assume $y_\infty=0$ and $y_n\equiv \tilde y_n$.  Under these conditions, Lemma~\ref{l:spacetranslation} implies that for any $\varphi\in H^1$,
\[
\lim_{n\to\infty} T_{{y}_n}^{-1} \mathcal{L} (\tilde{s}_n-s_n,0;z_n(\cdot+s_n)) P_c T_{y_n}\varphi
	= e^{i s_\I \Delta} \varphi,
\]
where the limit is taken strongly in $H^1$.  This implies the result.  
\end{proof}

\begin{proof}[Proof of Lemma \ref{l:orth}]  We will show that (i)$\implies$(ii)$\implies$(iii)$\implies$(i). 

\underline{(i) $\Rightarrow$ (ii):} By a density argument, we may assume $\varphi\in \mathcal{S}$. Passing to a subsequence, we may suppose that either
\[
	\lim_{n\to\I} |\tilde{s}_n -s_n| =\I \quad \text{ or } \quad \sup_{n}  |\tilde{s}_n -s_n| <\I.
\]

If $|s_n-\tilde s_n|\to\infty$, then the dispersive estimate of Lemma~\ref{L:weighted-dispersive} and the fact that $\varphi\in\mathcal{S}$ implies
\[
\| \mathcal{L}(\tilde s_n,s_n;z_n)P_c T_{y_n}\varphi\|_{L^4+L^8} \lesssim |\tilde s_n - s_n|^{-\frac34}\|\varphi\|_{L^{4/3}\cap L^{8/7}} \to 0 \qtq{as}n\to\infty,
\]
which implies the result.

If instead $|s_n-\tilde{s}_n |$ is bounded, then we must have $|y_n-\tilde{y}_n|\to\I$ as $n\to\I$.  In this case, Lemma~\ref{l:missing} implies
\[
\varphi_n:= T_{y_n}^{-1} \mathcal{L}(\tilde{s}_n, s_n ; z_n)P_c T_{y_n} \varphi\to \varphi_\I
\]
strongly in $H^1$ along a subsequence. Thus
\[
	|y_n-\tilde y_n|\to\infty\implies T_{y_n-\tilde{y}_n} \varphi_n \rightharpoonup 0 \wIN H^1\qtq{as}n\to\infty.
\]

\underline{(ii) $\Rightarrow$ (iii):} Assuming (ii), item (iii) follows with the choice
\[
u_k = \mathcal{L}(0,s_{n_k};z_{n_k}) P_c T_{y_{n_k}} \psi \in P_c H^1
\]
for some nonzero $\psi \in P_c H^1$.

\underline{(iii) $\Rightarrow$ (i):} Suppose (iii) holds but (i) fails, so that in particular
\[
\sup_n( |s_n-\tilde{s_n}| + |y_n-\tilde{y}_n| ) < \I.
\]
By Lemma \ref{l:missing}, one may choose a subsequence $\{n_k\}_k$ such that
\[
	D_k:= T_{\tilde{y}_{n_k}}^{-1} \mathcal{L} (\tilde{s}_{n_k},s_{n_k};z_{n_k})P_c T_{y_{n_k}}  
\]
converges strongly to a continuous invertible operator on $H^1$, say $D_\I.$ Using (iii), we can find a bounded sequence $\{ u_k\}_k \subset P_c H^1$ and a subsequence in $k$ such that
\[
T_{y_{n_k} }^{-1} \mathcal{L}(s_{n_k},0 ;z_{n_k})  u_k  \rightharpoonup \psi \neq 0\qtq{and} T_{\tilde{y}_{n_k} }^{-1}  \mathcal{L}(\tilde{s}_{n_k},0 ;z_{n_k})  u_k  \rightharpoonup 0.
\]
Now fix $\varphi \in \mathcal{S}$.  Then
\[
(T_{\tilde{y}_{n_k} }^{-1}  \mathcal{L}(\tilde{s}_{n_k},0 ;z_{n_k})  u_k,\varphi) \to 0\qtq{as}k\to\infty,
\]
while  
\[
(T_{\tilde{y}_{n_k} }^{-1}  \mathcal{L}(\tilde{s}_{n_k},0 ;z_{n_k})  u_k,\varphi)=( D_k T_{y_{n_k} }^{-1} \mathcal{L}(s_{n_k},0 ;z_{n_k})  u_k,\varphi) \to (D_\I \psi, \varphi)\qtq{as}k\to\infty.
\] 
As $\varphi$ was arbitrary and $D_\I$ is invertible, we deduce that $\psi =0$, a contradiction. \end{proof}

The next lemma is again related to the lack of unitarity of $\mathcal{L}$.  In particular, this lemma is needed to establish a vanishing property for orthogonal profiles. 

\begin{lemma}\label{l:orthremedy}
Let $\{z_n\} \subset \mathrm{SBC}$. For any parameter of shifts $(s_n,y_n)$, the following sequence of operators converges strongly along a subsequence:
\[
T_{y_n}^{-1}e^{i s_n H} \mathcal{L}(0,s_n;z_n) P_c T_{y_n}:H^1 \to H^1.
\]
Furthermore, if $|y_n|\to\I$, then then the limit is identity.
\end{lemma}

\begin{proof} If $s_n\equiv 0$, then the result is straightforward.  We will therefore focus on the case $|s_n|\to\infty$ and $|y_n|\to\infty$ (and note that the case $|s_n|\to\infty$ and $y_n\equiv 0$ may be handled similarly, cf. \cite[Lemma 5.2]{Nakanishi}). 

We first claim that
\begin{equation}\label{13claim}
T_{y_n}^{-1} e^{- i s_n \Delta} \mathcal{L}(0,s_n;z_n) P_c T_{y_n} \to \text{Id}\qtq{as}n\to\infty
\end{equation}
strongly as an operator on $H^1$.  To see this, first let $\varphi \in H^1$ and set 
\[
v_n(t)=\mathcal{L}(t,s_n;z_n) P_c T_{y_n}\varphi.
\] 
As in \eqref{e:spacepf4}, we may write 
\[
v_n(0) = e^{is_n \Delta} v_n(s_n) -i \int_{s_n}^0 e^{is \Delta} (-V v_n+ B[z_n]v_n)(s)\,ds,
\]
so that 
\begin{equation}\label{e:tyn111}
T_{y_n}^{-1} e^{-is_n \Delta} v_n(0) = T_{y_n}^{-1} P_c T_{y_n} \varphi-i \int_{s_n}^{0} e^{-i(s_n-s) \Delta} T_{y_n}^{-1}(-V v_n+ B[z_n]v_n)(s) \,ds.
\end{equation}
We now show that the right-hand side converges in $H^1$.  We have seen before that
\[
T_{y_n}^{-1}P_c T_{y_n}\varphi\to \varphi \qtq{strongly in}H^1,
\]
so we turn to the second term and establish convergence to zero. 

We suppose that $s_n\to\I$ as $n\to\I$ (the case $s_n\to-\infty$ is handled similarly).  We let $\eps>0$ and fix $\tau_0\ge 0$ to be chosen below.  For large $n$, we then have $s_n \ge \tau_0$ and thus, by Strichartz,
\begin{align*}
\biggl\|\int_{s_n-\tau_0 }^{0} & e^{-i(s_n-s) \Delta} T_{y_n}^{-1}(-V v_n+ B[z_n]v_n)(s)\,ds\biggr\|_{H^1} \\
	&\le \sup_{t \le -\tau_0} \norm{\int^{-\tau_0}_{t} e^{is \Delta} T_{y_n}^{-1}(-V v_n+ B[z_n(s+s_n)]v_n)(s+s_n)\,ds}_{H^1}\\
	&\lesssim \norm{T_{y_n}^{-1}(-V v_n+ B[z_n(\cdot+s_n)]v_n(\cdot+s_n))}_{L^2 H^1_{6/5} ((-\I,-\tau_0])}\\
	&\lesssim \norm{v_n(s_n+\cdot)}_{L^2 H^1_{6} ((-\I,-\tau_0])}.
\end{align*}
By \eqref{e:spacelargetimebound12}, Lemma~\ref{L:properties-of-L}, Strichartz, and Lemma~\ref{l:spacetranslation}, we have 
\begin{align*}
	\|v_n&(s_n+\cdot)\|_{L^2 H^1_{6} ((-\I,-\tau_0])} \\
	&= 	\norm{\mathcal{L}(\cdot  ,-\tau_0 ;z_n(\cdot +\tau_0+s_n)) P_c T_{y_n}\varphi }_{L^2 H^1_{6} ((-\I,0])}\\
	&\lesssim \norm{e^{-it\Delta} T_{y_n}^{-1} \mathcal{L}(0  ,-\tau_0 ;z_n(\cdot+ \tau_0+s_n)) P_c T_{y_n}\varphi }_{L^2 H^1_{6} ((-\I,0])} \\
	&\lesssim \norm{e^{-it\Delta}e^{i\tau_0 \Delta}\varphi }_{L^2 H^1_{6} ((-\I,0])} \\
	&\quad + \norm{T_{y_n}^{-1} \mathcal{L}(0  ,-\tau_0 ;z_n(\cdot+ \tau_0+s_n)) P_c T_{y_n}\varphi -e^{i\tau_0 \Delta}\varphi }_{H^1} \\
	&\lesssim \norm{e^{-it\Delta}\varphi }_{L^2 H^1_{6} ((-\I,-\tau_0])} + o_n(1)\qtq{as}n\to\infty.
\end{align*}
Thus, choosing $\tau_0$ sufficiently large, we find that for all large $n$,
\[
\norm{\int_{s_n-\tau_0 }^{0} e^{-i(s_n-s) \Delta} T_{y_n}^{-1}(-V v_n+ B[z_n]v_n)(s)\,ds}_{H^1} \le \eps
\]

On the other hand, we may write
\begin{align*}
	\int_{s_n}^{s_n-\tau_0} & e^{-i(s_n-s) \Delta} T_{y_n}^{-1}(-V v_n+ B[z_n]v_n)(s)\,ds \\
	& = \int_{0}^{-\tau_0} e^{i s\Delta} T_{y_n}^{-1}(-V v_n+ B[z_n]v_n)(s+s_n) \,ds
\end{align*}
Now, Lemma~\ref{l:spacetranslation}, we have that
\[
\|v_n(s+s_n)-T_{y_n}e^{-it\Delta}\varphi\|_{L^\infty H^1((-\tau_0,0))}\to 0 \qtq{as}n\to\infty.
\]
Thus, since 
\[
\int_{0}^{-\tau_0} e^{i s\Delta} T_{y_n}^{-1}(-V T_{y_n} e^{-is\Delta} \varphi+ B[z_n(s+s_n)]T_{y_n} e^{-is\Delta} \varphi)\, ds \to 0\qtq{in}H^1,
\]
we deduce  
\[
\int_{s_n}^{s_n-\tau_0} e^{-i(s_n-s) \Delta} T_{y_n}^{-1}(-V v_n+ B[z_n]v_n)(s) ds  \to 0\qtq{in}H^1\qtq{as}n\to\infty.
\]

Returning to \eqref{e:tyn111}, we conclude that
\[
-i \int_{s_n}^{0} e^{-i(s_n-s) \Delta} T_{y_n}^{-1}(-V v_n+ B[z_n]v_n)(s) ds \to 0
\] 
in $H^1$ as $n\to\I$, which proves the claim \eqref{13claim}.

It remains to upgrade from $e^{-is_n\Delta}$ to $e^{is_nH}$ in \eqref{13claim}. First, using \eqref{13claim} with the choice $z_n\equiv 0$, we find that
\[
T_{y_n}^{-1} e^{- i s_n \Delta} e^{-is_n H}  P_c T_{y_n} \to \text{Id} \qtq{as}n\to\infty
\]
(as operators on $H^1$). As $e^{-it H}$ is unitary, this also implies the convergence of the adjoint, that is:
\[
T_{y_n}^{-1}  e^{is_n H}  P_c e^{ i s_n \Delta} T_{y_n} \to \text{Id}\qtq{as}n\to\infty.
\]
Thus,
\begin{align*}
T_{y_n}^{-1}&e^{i s_n H} \mathcal{L}(0,s_n;z_n) P_c T_{y_n}\\
& =(T_{y_n}^{-1}e^{i s_n H} P_c e^{ i s_n \Delta} T_{y_n})( T_{y_n}^{-1} e^{ -i s_n \Delta}   \mathcal{L}(0,s_n;z_n) P_c T_{y_n})\to \text{Id}
\end{align*}
as $n\to\I$.\end{proof}

With Lemma~\ref{l:orthremedy} in place, we can prove the following consequence of orthogonality. 

\begin{proposition}\label{p:orthremedy}
Let $\{z_n\}_n \subset \mathrm{SBC}$, and let
\[
\lambda_n^{j}=\mathcal{L}(0,s_n^j ; z_n)P_c T_{y_n^j} \varphi^j\qtq{and}\lambda_n^k= \mathcal{L}(0,s_n^k ; z_n) P_c T_{y_n^k}  \varphi^k
\]
be two linear profiles. If the parameters of shifts $(s_n^j,y_n^j)$ and $(s_n^k,y_n^k)$ obey
\[
|s_n^j-s_n^k| + |y_n^j-y_n^k| \to\I\qtq{as}n\to\infty,
\]
then
\[
|(\lambda_n^j, \lambda_n^k)| + |(H \lambda_n^j, \lambda_n^k)| \to 0\qtq{as}n\to\infty.
\]

Furthermore, if $\{r_n\}$ is a bounded sequence in $H^1$ which satisfies 
\[
T_{y_n^j}^{-1} \mathcal{L}(s_n^j,0;z_n) r_n \rightharpoonup 0\qtq{weakly in} H^1,
\]
then
\[
|(\lambda_n^j, r_n)| + |(H \lambda_n^j, r_n)| \to 0\qtq{as}n\to\infty.
\]
\end{proposition}

\begin{proof}  We focus on the $(H\cdot,\cdot)$ inner product.

We suppose  $|s_n^j-s_n^k|\to\I$ as $n\to\I$.  As $e^{itH}$ is unitary and commutes with $H$, we may write
\begin{align*}
	(H \lambda_n^j, \lambda_n^k) 
	= ((-\Delta + (T_{y_n^j}^{-1}V)) A_n \varphi^j, 
	A_n C_n\varphi^k),
\end{align*}
where
\[
A_n := T_{y_n^j}^{-1} e^{is_n^j H}\mathcal{L}(0,s_n^j ; z_n)P_c T_{y_n^j}\qtq{and}C_n := T_{y_n^j}^{-1} \mathcal{L}(s_n^j,s_n^k;z_n)P_c T_{y_n^k}.
\]
By Lemma~\ref{l:orthremedy}, we have that $A_n$ converges strongly (as operators on $H^1$).  Moreover, by Lemma~\ref{l:orth}, $C_n\varphi^k\rightharpoonup 0$ weakly in $H^1$.  Thus 
\[
((-\Delta) A_n \varphi^j, e^{is_nH} A_n C_n\varphi^k) \to 0\qtq{as}n\to\infty.
\]
As $|s_n^j-s_n^k|\to\I$ as $n\to\I$, we also see that $C_n \to 0$ strongly as an operator $H^1\to L^4$, which yields
\[
((T_{y_n^j}^{-1}V)) A_n \varphi^j, A_n C_n\varphi^k) \to 0\qtq{as}n\to\infty.
\]

We next suppose that $|s_n^j - s_n^k|$ is bounded, which necessarily implies $|y_n^j| + |y_n^k| \to \I$ as $n\to \I$.  Without loss of generality, we suppose that $|y_n^k|  \to \I$ as $n\to \I$.  Choosing a subsequence along which $s_n^j- s_n^k\to \tilde s_\infty$, Lemma~\ref{l:spacetranslation} and Lemma~\ref{l:orthremedy} together yield
\begin{align*}
	&T_{y_n^k}^{-1} T_{y_n^j} A_n C_n \\
	&= (T_{y_n^k}^{-1} e^{is_n^j H}\mathcal{L}(0,s_n^j ; z_n)P_c T_{y_n^k})
	(T_{y_n^k}^{-1} \mathcal{L}(s_n^j,s_n^k;z_n)P_c T_{y_n^k}) \to  e^{-i \tilde{s}_\I \Delta }
\end{align*}
strongly in $H^1$. Combining this with the convergence of $A_n$,  the problem reduces to showing that
\[
((-\Delta + (T_{y_n^j}^{-1}V)) A \varphi^j, T_{y_n^k-y_n^j}   e^{-i \tilde{s}_\I \Delta }\varphi^k) \to 0\qtq{as}n\to\infty.
\]
In fact, this convergence is straightforward, as both
\begin{align*}
&\lim_{n\to\I} ((-\Delta)  A \varphi^j, T_{y_n^k-y_n^j}   e^{-i \tilde{s}_\I \Delta }\varphi^k) = 0,\\
&\lim_{n\to\I} \norm{(A \varphi^j) (T_{y_n^k-y_n^j}   e^{-i \tilde{s}_\I \Delta }\varphi^k)}_{L^1 \cap L^2} = 0, 
\end{align*}
follow from the fact that $\lim_{n\to\I}|y_n^j-y_n^k|= \I$.

We turn to the other estimate. We let 
\[
\tilde{r}_n = T_{y_n^j}^{-1} \mathcal{L}(s_n^j,0;z_n) r_n.
\]
Using Lemma~\ref{l:orthremedy}, it suffices to show that
\[
	((-\Delta+(T_{y_n^j}^{-1}V))A \varphi^j, {A^*} \tilde{r}_n  ) \to 0\qtq{as}n\to\infty,
\]
where
\[
A:=\lim_{n\to\infty} T_{y_n^j}^{-1} e^{is_n^j H}\mathcal{L}(0,s_n^j ; z_n)P_c T_{y_n^j}
\]
(along some subsequence). The convergence
\[
	\lim_{n\to\I} (-\Delta A \varphi^j, {A^*} \tilde{r}_n  ) =0
\]
is immediate from the weak convergence of $\tilde{r}_n$. Let us therefore consider the potential term. As $A \tilde{r}_n$ converges to zero weakly in $H^1$, we have
\[
\sup_{R>0} \lim_{n\to\I} \norm{ {\bf 1}_{\{|x| \le R\}} A \tilde{r}_n  }_{L^4} =0.
\]
Then, given $\eps>0$, we choose $R_0>0$ such that
\[
\norm{{\bf 1}_{\{|x| \ge R_0 \}} A \varphi^j }_{} \le \eps
\]
and estimate
\begin{align*}
	\varlimsup_{n\to\I} |((T_{y_n^j}^{-1}V)A \varphi^j, A \tilde{r}_n  )|&\le \norm{V}_{L^2}\norm{{\bf 1}_{\{|x| \ge R_0 \}} A \varphi^j }_{L^4} \sup_n \norm{A\tilde{r}_n }_{H^1} \\
	&\quad + \norm{V}_{L^2} \norm{ A \varphi^j }_{H^1} \varlimsup_{n\to\I}\norm{ {\bf 1}_{\{|x| \le R_0 \}} A\tilde{r}_n }_{L^4} \\
	& \lesssim \eps.
\end{align*}
As $\eps$ was arbitrary, we obtain the result.
\end{proof}

We turn finally to the construction of the linear profile decomposition.  The main step is to show that if $\psi_n$ is a bounded sequence in $H^1$ with the property that $\mathcal{L}(\cdot,0;z_n)\psi_n$ is nontrivial in $L_t^\infty L_x^4$, then we can extract a `bubble of concentration'.  To achieve this, we proceed essentially as in \cite{Gerard} (see also \cite{BahouriGerard,CarlesKeraani,MasakiSegata}).  

\begin{definition} Given $\{z_n\}_n\subset \mathrm{SBC}$ and a bounded sequence $\{ \psi_n \}_n \subset P_c H^1$, we let $\mathcal{V}(\{z_n\},\{\psi_n\})$ denote the set of all $\varphi\in H^1$ such that there exist parameters of shifts $(s_n,y_n)$ and a sequence in $n$ along which
\[
T_{y_n}^{-1}\mathcal{L}(s_n,0;z_n)\psi_{n} \rightharpoonup \varphi \qtq{weakly in}H^1.
\]
We set
\[
\eta(\{z_n\},\{\psi_n\}) = \sup_{\varphi\in \mathcal{V}(\{z_n\},\{\psi_n\})}\|\varphi\|_{H^1}. 
\]
\end{definition}
%

\begin{remark}
In the non-radial setting, space translation must be taken into account. Because of the presence of translation, the elements of $\mathcal{V}$ do not necessarily belong to $P_c H^1$.
\end{remark}

\begin{proposition}[Control on vanishing]\label{p:vanishing}
Let $\{z_n\}_n\subset \mathrm{SBC}$ and let $\{\psi_n\}_n\subset P_c H^1$ be a bounded sequence.
If
\[
\sup_n \norm{\psi_n}_{H^1}\le M\qtq{and}\varlimsup_{n\to\I} \norm{\mathcal{L}(\cdot, 0;z_n) \psi_n }_{L^\I_t (\R, L^4_x)} \ge \eps,
\]
then there exists $\beta = \beta (M, \eps)>0$ such that
\[
		\eta (  \{z_n\}, \{ \psi_n \}) \ge \beta.
\]
\end{proposition}
\begin{proof} Extracting a subsequence if necessary, we may suppose that
\[
	\norm{\mathcal{L}(\cdot, 0;z_n) \psi_n }_{L^\I_t (\R, L^4_x)} \ge \eps/2 \qtq{for all}n.
\]

First, by Sobolev embedding,
\[
\norm{ P_{\ge N} \mathcal{L}(\cdot, 0;z_n) \psi_n }_{L^\I_t (\R, L^4_x)}\lesssim N^{-\frac14}	\norm{ P_{\ge N} \mathcal{L}(\cdot, 0;z_n) \psi_n }_{L^\I_t (\R, H^1)}\lesssim  N^{-\frac14} M.
\]
Thus there exists $N=N(M,\eps)$ such that
\[
\norm{ P_{< N} \mathcal{L}(\cdot, 0;z_n) \psi_n }_{L^\I_t (\R, L^4_x)} \ge \eps/4\qtq{for all}n.
\]
On the other hand, by H\"older's inequality, 
\[
\norm{ P_{< N} \mathcal{L}(\cdot, 0;z_n) \psi_n }_{L^\I_t (\R, L^4_x)}\lesssim M^\frac12\norm{ P_{< N} \mathcal{L}(\cdot, 0;z_n) \psi_n }_{L^\I_t (\R, L^\I_x)}^\frac12.
\]

It follows that 
\[
\norm{ P_{< N} \mathcal{L}(\cdot, 0;z_n) \psi_n }_{L^\I_t (\R, L^\I_x)}\gtrsim \eps^2/M\qtq{for all}n,
\]
and thus there exist $(s_n,y_n)\in\R\times\R^3$ such that
\[	
|( P_{< N} \mathcal{L}(s_n, 0;z_n) \psi_n)(y_n) | \gtrsim \eps^2/M.
\]

We now observe that
\[
|( P_{< N} \mathcal{L}(s_n, 0;z_n) \psi_n)(y_n)|= \left| \int_{\R^3}  g(y)   (\mathcal{L}(s_n, 0;z_n) \psi_n) (y+y_n)\,dy \right|
\]
for some Schwartz kernel $g$.  As $\{ T^{-1}_{y_n} \mathcal{L}(s_n, 0;z_n) \psi_n\}_n$ is a bounded sequence in $H^1$, it converges weakly along a subsequence to some $\varphi\in H^1$. Along this subsequence, we have
\[
\int_{\R^3}  g(y)   (\mathcal{L}(s_n, 0;z_n) \psi_n) (y+y_n)\,dy\to \int_{\R^3}  g(y)  \varphi (y)\,dy,
\]
which implies
\[
\norm{\varphi}_{H^1} \gtrsim \eps^2/M.
\]
Now, if $(s_n,y_n)$ is already a parameter of shifts (so that $s_n\equiv 0$ or $|s_n|\to\infty$ and $y_n\equiv 0$ or $|y_n|\to\infty$), then we obtain that $\varphi\in\mathcal{V}(\{z_n\},\{\psi_n\})$ and
\[
\eta(\{z_n\},\{\psi_n\})\geq\|\varphi\|_{H^1}\gtrsim \eps^2/M,
\]
and the proof is complete.  Otherwise, we argue as follows.

If $\{s_n\}$ and $\{y_n\}$ are both bounded, then we may pass to a subsequence in $n$ so that $s_n\to s_0$, $y_n\to y_0$, and $z_n\to z_\infty$ (locally uniformly). Then the convergence 
\[
T_{y_n}^{-1}\mathcal{L}(s_n,0;z_n)\psi_n\rightharpoonup\varphi
\]
yields 
\[
	(\psi_n,h) = (\mathcal{L}(0,s_n;z_n) P_cT_{y_n}(T_{y_n}^{-1}\mathcal{L}(s_n,0;z_n)\psi_n),h)
	\to (\mathcal{L}(0,s_0;z_\infty) P_cT_{y_0}\varphi,h)
\]
as $n\to\infty$ for all $h \in H^1$. This reads as
\[
\mathcal{L}(0,0;z_n)\psi_n = \psi_n\rightharpoonup \tilde\varphi:= (\mathcal{L}(0,s_0;z_\infty) P_cT_{y_0})^*\varphi\in\mathcal{V}(\{z_n\},\{\psi_n\}),
\]
with
\[
\eta(\{z_n\},\{\psi_n\}) \geq \|\tilde \varphi\|_{H^1} \gtrsim \|\varphi\|_{H^1}\gtrsim \eps^2/M,
\]
as desired. Similar arguments treat the case that only one of the sequences $\{s_n\},\{y_n\}$ is bounded. \end{proof}

We turn to the proof of Theorem~\ref{t:lpd}.  We follow the standard argument, removing one bubble of concentration at a time. 


\begin{proof}[Proof of Theorem \ref{t:lpd}]
Passing to a subsequence, we first find $\varphi_\infty^0\in H^1$ so that $\psi_n\rightharpoonup \varphi_\infty^0$ weakly in $H^1$.  We then define $\lambda^0_n=\varphi_\I^0$ and $\gamma_n^0= \psi_n - \varphi_\I^0$, so that \eqref{e:lpd} holds for $J=0$.  Moreover,
\begin{equation}\label{e:lpdpf1}
\gamma_n^0  \rightharpoonup 0\wIN H^1\qtq{as}n\to\infty.
\end{equation}

Now, if $\eta  (\{z_n\}, \{ \gamma_n^0 \})=0$, then the result holds with $J^*=0$.  In this case, we let $\varphi_\I^j=0$ for all $j\ge1$ and pick arbitrary  $\{(y_n^j,s_n^j)\}$ for $j\ge1$ so that the orthogonality condition is fulfilled.

Suppose instead that $\eta(\{z_n\}, \{ \gamma_n^0 \})>0$. Then, by definition, there exists $\varphi_\I^1 \in \mathcal{V}(\{z_n\}, \{ \psi_n \})$ such that
\[
\norm{\varphi_\I^1}_{H^1} \ge \tfrac12  \eta(\{z_n\}, \{ \gamma_n^0 \})>0
\]
and a parameter of shifts $\{(y_n^1,s_n^1)\}$ such that
\begin{equation}\label{e:lpdpf2}
T^{-1}_{y_n^1} \mathcal{L} (s_n^1,0;z_n) \gamma_n^0 \rightharpoonup \varphi_\I^1\wIN H^1
\end{equation}
as $n\to\I$ along a subsequence. 

We now set
\[
\lambda^1_n := \mathcal{L} (0, s_n^1;z_n) P_c T_{y_n^1} \varphi_\I^1  \qtq{and}\gamma^1_n := \psi_n - \lambda_n^1,
\]
which yields the decomposition \eqref{e:lpd} for $J=1$. In view of Lemma~\ref{l:orth}, we see from \eqref{e:lpdpf1} and \eqref{e:lpdpf2} that the property (ii) holds for $(j,k)=(0,1)$. Moreover, by \eqref{e:lpdpf2} 
and the definition of $\gamma_n^1$ and $\lambda_n^1$, we have
\begin{equation}\label{e:lpdpf3}
T^{-1}_{y_n^1} \mathcal{L} (s_n^1,0;z_n)\gamma_n^1  \rightharpoonup 0\wIN H^1\qtq{as}n\to\infty.
\end{equation}
Furthermore, by Lemma~\ref{l:orth}(ii), Proposition \ref{p:orthremedy}, and \eqref{e:lpd}, we have
\begin{equation}\label{e:lpdpf2.5}
\norm{\gamma_n^0}_{L^2}^2 = \norm{\lambda_n^1}_{L^2}^2 + \norm{\gamma_n^1}_{L^2}^2 + o(1)\qtq{as}n\to\infty,
\end{equation}
\begin{equation}\label{e:lpdpf2.55}
\mathbb{H}_V(\gamma_n^0) = \mathbb{H}_V(\lambda_n^1) + \mathbb{H}_V(\gamma_n^1) + o(1)\qtq{as}n\to\infty.
\end{equation}

If $\eta(\{z_n\}, \{ \gamma_n^1 \})=0$, the result holds with $J^*=1$.  We obtain the result by defining the remaining profiles suitably (as in the case of $J^*=0$).  If instead $\eta(\{z_n\}, \{ \gamma_n^1 \})>0$, then there exists $\varphi_\I^2 \in \mathcal{V}(\{z_n\}, \{ \gamma_n^1 \})$ such that
\[
\norm{\varphi_\I^2}_{H^1} \ge \tfrac12  \eta  (  \{z_n\}_n, \{ \gamma_n^1 \}_n)>0.
\]
Moreover, there exist $\{(y_n^2,s_n^2)\}_n$ such that
\begin{equation}\label{e:lpdpf4}
T^{-1}_{y_n^2} \mathcal{L} (s_n^2,0;z_n) \gamma_n^1 \rightharpoonup \varphi_\I^2 \wIN H^1
\end{equation}
as $n\to\I$ along a subsequence. We then set
\[
\lambda^2_n := \mathcal{L} (0, s_n^2;z_n) P_c T_{y_n^2} \varphi_\I^2\qtq{and}	\gamma^2_n := \gamma^2_n - \lambda_n^2,
\]
which yields the decomposition \eqref{e:lpd} for $J=2$. Using Lemma~\ref{l:orth}, we see from \eqref{e:lpdpf3} and \eqref{e:lpdpf4} that \ the property (ii) holds for $(j,k)=(1,2)$.  By \eqref{e:lpdpf4} and the definition of $\gamma_n^2$ and $\lambda_n^2$, we also have
\begin{equation}\label{e:lpdpf5}
T^{-1}_{y_n^2} \mathcal{L} (s_n^2,0;z_n)\gamma_n^2  \rightharpoonup 0\wIN H^1\qtq{as}n\to\infty.
\end{equation}

We repeat this procedure and define $\varphi_\I^m$, $\{ (y_n^m,s_n^m) \}$, and $\gamma_n^{m}$ inductively as long as $\eta(\{z_n\},\{ \gamma_n^{m-1} \})>0$. Indeed, if $\eta  (  \{z_n\}, \{ \gamma_n^{m-1} \})>0$ then one finds $\varphi^m_\I \in H^1$ so that
\begin{equation}\label{e:lpdpf5.5}
\norm{ \varphi_\I^m }_{H^1} \ge \tfrac12 \eta  (  \{z_n\}, \{ \gamma_n^{m-1} \})>0.
\end{equation}
We define $(y_n^m,s_n^m)$ so that
\begin{equation}\label{e:lpdpf6}
	T^{-1}_{y_n^m} \mathcal{L} (s_n^m,0;z_n) \gamma_n^{m-1} \rightharpoonup \varphi_\I^m \wIN H^1
\end{equation}
We then derive the decomposition \eqref{e:lpd} with $J=m$ from the decomposition with $J=m-1$ by defining $\gamma_n^{m}$ via $\gamma_n^{m-1} = \l_n^{m} + \gamma_n^{m}$.  In this case we obtain
\begin{equation}\label{e:lpdpf7}
T^{-1}_{y_n^m} \mathcal{L} (s_n^m,0;z_n)\gamma_n^{m}  \rightharpoonup 0\wIN H^1\qtq{as}n\to\infty.
\end{equation}
Using Lemma~\ref{l:orth} as above, we see from \eqref{e:lpdpf6} and \eqref{e:lpdpf7} that property (ii) holds for $(j,k)=(m-1,m)$. Moreover, by Proposition \ref{p:orthremedy} and \eqref{e:lpdpf7}, we have
\begin{equation}\label{e:lpdpf8.5}
\norm{\gamma^{m-1}_n}_{L^2}^2 = \norm{\lambda_n^m}_{L^2}^2 + \norm{\gamma_n^m}_{L^2}^2 + o(1)\qtq{as}n\to\infty,
\end{equation}
\begin{equation}\label{e:lpdpf8.55}
	\mathbb{H}( \gamma^{m-1}_n )  = \mathbb{H}( \lambda_n^m ) + \mathbb{H}( \gamma_n^m) + o(1)\qtq{as}n\to\infty.
\end{equation}

If $\eta(\{z_n\}, \{ \gamma_n^{J_0} \})=0$ holds for some $J_0 \in \N$ then we let $J^*=J_0$ and define the linear profiles for $j \ge J^*$ suitably as above. Otherwise, $J^* =\I$.

It remains to check the desired properties of the decomposition. By induction on $j-k$, we see that
\begin{equation}\label{e:lpdpf9}
	T^{-1}_{y_n^{k}} \mathcal{L} (s_n^{k},0;z_n)\gamma_n^{j}  \rightharpoonup 0\wIN H^1
\end{equation}
and that the property (ii) holds for all $0 \le k \le j \le J^*$.  The property (i) follows from (ii) and the definition of $(y_n^0,s_n^0)$.  The property (iii) follows from \eqref{e:lpd} for $J=j-1$, \eqref{e:lpdpf6}, the property (ii), and Lemma~\ref{l:orth}.  Finally, the property (iv) was shown in \eqref{e:lpdpf5.5} and the definition of $J^*$.

We now consider the first several decoupling statements.  The decomposition \eqref{e:Mdecomp} follows from \eqref{e:lpdpf2.5} and \eqref{e:lpdpf8.5}. Next, we see from \eqref{e:lpdpf2.55} and \eqref{e:lpdpf8.55} that
\begin{equation}\label{e:HVdecomp}
\mathbb{H}_V(\psi_n) = \sum_{j=0}^J \mathbb{H}_V(\lambda_n^j) + \mathbb{H}_V(\gamma_n^J) + o_n(1).
\end{equation}
Now, if $\{v_n\}$ is a bounded sequence in $H^1$ which converges to zero weakly in $H^1$ as $n\to\I$, then one sees from 
\[
\sup_{R>0} \norm{v_n {\bf 1}_{\{|x| \le R \}} }_{L^4} =0
\]
and the fact that $V\in L^2$ that
\[
	\lim_{n\to\I} \int V |v_n|^2 dx = \lim_{n\to\I} \int V v_n \overline{w} dx =0
\]
for any $w \in H^1$. As $(y^j,s^j)\neq(0,0)$ for all $j\ge1$, we see that
$\l_n^j \rightharpoonup 0$ weakly in $H^1$ as $n\to\I$ for all $j \ge 1$.  Moreover, we see from \eqref{e:lpd} that $\gamma_n^j \rightharpoonup 0$ weakly in $H^1$ as $n\to\I$ for all $j \ge 1$.  Thus, together with \eqref{e:lpd}, one has
\begin{equation}\label{e:Verrorpf1}
\int V |\l_n^j|^2 dx = o_n(1)\qtq{and} \int V |\gamma_n^j|^2 dx = o_n(1)\qtq{for all}j\ge 1,
\end{equation}
along with the desired equality \eqref{e:Vdecomp}.  Combining \eqref{e:Verrorpf1} and \eqref{e:Vdecomp} with \eqref{e:HVdecomp}, one then obtains \eqref{e:H0decomp}. 

We next establish the desired bounds on the remainder.  The uniform bound is an immediate consequence of the uniform Strichartz estimate (Proposition~\ref{Prop:Strichartz}), so we turn to \eqref{e:remainder}. We only consider the case $J^*=\I$.  In this case, we can use \eqref{e:Mdecomp}, \eqref{e:H0decomp}, and $\tnorm{\lambda_n^j }_{H^1}\sim \tnorm{\varphi_\I^j}_{H^1} $ to obtain 
\[
\sum_{j=0}^\I \norm{\varphi_\I^j}_{H^1}^2 \lesssim \sup_n \norm{\psi_n}_{H^1}^2 <\I.
\]
Thus, $0 < \eta  (  \{z_n\}, \{ \gamma_n^{J} \}) < 2 \norm{ \varphi_\I^J }_{H^1} \to 0$ as $J\to\I$. By Proposition~\ref{p:vanishing}, this implies \eqref{e:remainder}.

Finally, we address the remaining decoupling statements.  First, we deduce \eqref{e:Gdecomp2} from \eqref{e:remainder} as in the proof of \eqref{e:Gdecomp}, relying on the fact that $\lim_{n\to\I} \mathbb{G}(\l_n^j)=0$
if $s^j=\I$.  Finally, combining \eqref{e:H0decomp}, \eqref{e:Gdecomp2}, \eqref{e:Vdecomp}, and \eqref{e:spacenormflat2}, we obtain \eqref{e:energydecomp}. \end{proof}

Finally, we observe that by a suitable modification of the profiles (as in the last paragraph of the proof of Proposition \ref{p:vanishing}), we may choose the time-translation parameters so that for each $j,k$, we have either
\[
\lim_{n\to\infty} |s_n^j - s_n^k| = \infty\qtq{or}s_n^j\equiv s_n^k.
\]
We will work under this assumption in what follows.

\section{Proof of Proposition~\ref{p:key}}\label{S:key}\label{S:NPD}

In this section we prove the key convergence result, Proposition~\ref{p:key}. We first write the solutions $u_n$ as
\[
u_n = \Phi[z_n] + R[z_n] \xi_n
\]
We then have $\{ z_n \} \subset \mathrm{SBC}$ by mass conservation and  \emph{a priori} $H^1$ bounds.  We apply the linear profile decomposition (Theorem~\ref{t:lpd})  to the bounded sequence $\{ \xi_n(0) \}_n \subset P_c H^1$.  In particular, up to a subsequence in $n$, we obtain $J^* \in \N \cup \{0, \I \}$, sequences
\[
(s_n^j,y_n^j) \in \R \times \R^3\qtq{and}  \varphi_\I^j \in H^1,
\]
and, for $n\geq 1$ and $J\geq 0$, a decomposition of the form
\[
\xi_n(0) = \sum_{j=0}^J \lambda_n^j + \gamma_n^J,\qtq{where} \lambda_n^j := \mathcal{L} (0,s_n^j;z_n) P_c T_{y_n^j} \varphi_\I^j
\]
and the remainder term satisfies
\begin{equation}\label{e:linear_remainder_smallness0}
\lim_{J\to J^*} \varlimsup_{n\to\I} \norm{\mathcal{L} (\cdot,0;z_n) \gamma_n^J}_{L^\I_tL^4_x} =0.
\end{equation}

We denote
\[
\sigma^j:=  \lim_{n\to\I} s^j_n   \in \{ 0, \pm \I \}, \quad y^j := \lim_{n\to\I} |y^j_n|  \in \{ 0, \I \}
\]
for $j \ge 0$. Then (by the orthogonality of parameters) we have $(\sigma^0,y^0)=(0,0)$ and $(\sigma^j,y^j)\neq (0,0)$ for $j\ge1$.  Recall that for each pair $(j,k)$, either $s_n^j\equiv s_n^k$ or $\sign(s_n^j-s_n^k) \in\{\pm 1\}$ is independent of $n$.

If $J^*\ge 1$, we have $\varphi_\I^j \neq0$ for $j \in [1,J^*]$.  If $J^*$ is finite, we use the convention that $\varphi_\I^j = \gamma_n^j =0$ for $j \ge J^*+1$ and $n\ge1$, and the parameters $(y_n^j,\sigma_n^j)$
for $j \ge J^*+1$ are chosen so that $y^j=\I$ for all $j \ge J^*+1$ and
the orthogonality of $\{(y_n^j,\sigma_n^j)\}_{j\ge1}$ holds.

We now let $\eps>0$ to be determined below and choose $J^\dagger= J^\dagger(\eps) \le J^*$ such that
\begin{align}\label{e:linear_remainder_smallness1}
&\sup_{J \ge J^\dagger} \varlimsup_{n\to\I} \norm{ \mathcal{L} (t,0;z_n) \gamma_n^{J}}_{(L^\I_t \cap L^8_t) L^4_x (\R)} \le \eps,\\
\label{e:linear_remainder_smallness2}
& \sum_{j= J^\dagger + 1}^\I \norm{\varphi_\infty^j}_{H^1}^2 \le \eps^2.
\end{align}
Note that that smallness of the $L^8_tL^4_x$-norm follows from \eqref{e:linear_remainder_smallness0} combined with interpolation, Strichartz, and the uniform $H^1$ bound on $\gamma_n^J$.  The parameter $\eps>0$ will be chosen in a way that depends only on $\sup_{n} \norm{\xi_n(0)}_{H^1}$ and $\sup_n \norm{z_n}_{L^\I_t(\R)}$ (and several universal constants); see Remark \ref{r:choiceofeps} below.

In the construction above, we have labeled the profiles $\varphi_\infty^j$ using the index $j\in\mathbb{N}_0$.  In what follows, it will be convenient to have a different system for labeling profiles with $0\leq j\leq J^\dagger$.  In particular, each $j$ will correspond to a unique pair $(s,k)\in\mathbb{Z}\times[\mathbb{N}\cup\{0\}]$.

\begin{definition}[Relabeling via the mapping $j\mapsto(s,k)$]\label{def:relabel} By construction, for each distinct pair $j,k\in[0,\infty)$, we have
\[
\lim_{n\to\infty} s_n^j-s_n^k = \infty,\quad \lim_{n\to\infty} s_n^k-s_n^j = \infty,\qtq{or}s_n^j\equiv s_n^k.
\]
This allows us to put the set of sequences `in order'.  For example, $(s_n^j) > (s_n^k)$ if $s_n^j-s_n^k\to\infty$.  In particular, we can map the set of indices $j\in[0,J^\dagger]$ to $s(j)\in[s_{\min},s_{\max}]\cap\mathbb{Z}$ in such a way that $(s_n^j)>(s_n^k)$ if and only if $s(j)>s(k)$ (and similarly $s_n^j\equiv s_n^k$ if and only if $s(j)=s(k)$).  We choose $s(0)=0$, so that $s_{\min}\leq 0 \leq s_{\max}$ and $s_{\max}-s_{\min}\leq J^\dagger$.

We define $s(j)$ for $j>J^\dagger$ as follows:  First, if $s_n^j\equiv s_n^k$ for some $k\in[0,J^\dagger]$, we set $s(j)=s(k)$.  Otherwise, we set $s(j)=s_{\min}-1$. 

Next, we define $k(j)$. First, if $s(j)\in[s_{\min},s_{\max}]$ and $y^j=0$, we set $k(j)=0$.  If instead $y^j=\infty$, then we let $k(j)$ equal the number of $m\in[0,j]$ such that $s_n^s\equiv s_n^j$ and $y^m=\infty$.  Second, if $s(j)=s_{\min}-1$, then we let $k(j)$ equal the number of $m\in[0,j)$ such that $s_n^j\equiv s_n^m$. 

It follows that $j\mapsto (s(j),k(j))$ is injective. Furthermore, we have the following monotonicity property: if $j_1<j_2$,  $s(j_1)=s(j_2)\in[s_{\min},s_{\max}]$, and $y^{j_1}=y^{j_2}=\infty$, then $1\leq k(j_1)<k(j_2)$.

\end{definition}

With this new labeling system in place, we next introduce a modified remainder and a modified form of the profile decomposition.  We first define the following: 

\begin{definition}[Modified Remainder]  Define 
\[
S=\{j\in(J^\dagger,\infty):s(j)\in [s_{\min},s_{\max}]\}.
\]
For $J>J^\dagger$, let $S_J = \{j\in(J^\dagger,J]:s(j)\in[s_{\min},s_{\max}]\}$ and
\[
\tilde \gamma_n^J := \gamma_n^{J^\dagger} - \sum_{j\in S_J} \lambda_n^j. 
\]
\end{definition}
In particular, while $\gamma_n^J$ is defined by removing all profiles indexed by $j\in(J^\dagger+1,J]$ from $\gamma_n^{J^\dagger}$, the modified remainder $\tilde \gamma_n^J$ is obtained by removing only those profiles whose time-translation sequences are equivalent to one that already appears in the first $J^\dagger+1$ profiles.  Equivalently, we can say that we obtain $\tilde\gamma_n^J$ from $\gamma_n^J$ by adding back in those profiles that are not equivalent (with respect to the time-translation parameters) to any of the first $J^\dagger+1$ profiles.
%

 Next, let us define
\[
J_0 = \max\bigl\{ J^\dagger, \max \{ j \ge 0 \ |\ s( j) \in [s_{\min},s_{\max}]\qtq{and} k(j) = 0  \}\bigr\} .
\]
In particular, $J_0$ is finite since $s_{\max} - s_{\min}\leq J^\dagger$ and the map $j\mapsto(s(j),k(j))$ is injective.  Furthermore, if $J^*<\infty$ then $J_0 \le J^*$; this is due to our convention that $y^j=\infty$ for $j  \ge J^*+1$. 

Finally, for $s \in [s_{\min},s_{\max}]$ and $J \ge 1$, we let
\begin{equation}\label{KsJ}
K(s,J) = \# \{ m \in [1,J] \ |\ s(m)=s,\, y^m=\I \}.
\end{equation}

Then, for $J \ge J_0$, we may obtain the following modified profile decomposition:
\begin{equation}\label{e:mlpd}
	\xi_n(0) = \sum_{s = s_{\min}}^{s_{\max}} \( \l_n^{(s,0)} + \sum_{k=1}^{K(s,J)} \l_n^{(s,k)} \) + \tilde{\gamma}_n^J,
\end{equation}
where we have relabeled $\lambda_n^j$ as $\lambda_n^{(s,k)}$ (with $s=s(j)$ and $k=k(j)$), and we use the convention that $\l_n^{(s,0)}=0$ if there is no $j\ge 1$ such that $(s(j),k(j))=(s,0)$.
\begin{proof}[Proof of \eqref{e:mlpd}] By the profile decomposition at level $J^\dagger$ and the definition of $\tilde\gamma_n^J$, we have
\[
\xi_n(0)=\sum_{j=0}^{J^\dagger} \lambda_n^j + \gamma_n^{J^\dagger}  = \sum_{j=0}^{J^\dagger}\lambda_n^j + \sum_{j\in S_J} \lambda_n^j + \tilde\gamma_n^J. 
\]
Then, by the definition of $S_J$ and $K(s,J)$, we have 
\[
\{\lambda_n^j\}_{j=0}^{J^\dagger}\cup \{\lambda_n^j\}_{j\in S_J}=\{\lambda_n^\ell:s(\ell)\in [s_{\min},s_{\max}]\qtq{and} k(\ell)\in[0,K(s(\ell),J)]\},
\]
which yields the result. 
\end{proof}

Our next main step will be to define nonlinear profiles associated to each linear profile.  Before doing so, let us  record the following estimate on linear profiles.  We first introduce the notation
\[
\lambda_n^{(s,k)}(t):=\mathcal{L}(t,0;z_n)\lambda_n^{(s,k)}
\]
and remark that the time dependence may not always explicitly be indicated (and so will need to be understood from context). Essentially, the following estimate captures the fact that the $\lambda_n^{(s,k)}$ (for fixed $s$) concentrate the bulk of their norm around time $s_n^j$, where $s=s(j)$.  Here and below, we use the following notation:  Given $s\in[s_{\min},s_{\max}]$, we have $s=s(j)$ for at least one $j\geq 0$. We denote the time sequence $s_n^j$ by $s_n^s$, which is well-defined due to the fact that $s(j)=s(k)\in[s_{\min},s_{\max}]$ implies $s_n^j\equiv s_n^k$ (cf. Definition~\ref{def:relabel}).

\begin{lemma}\label{l:differentlinearprofile}
Let $s,s' \in [s_{\min},s_{\max}]$. If $s\neq s'$ then
\[
\lim_{\tau\to\I } \sup_{K\ge 0} \varlimsup_{n\to\I} \norm{ \sum_{k=0}^K \l_n^{(s,k)} }_{L^8L^4((-\I, -\tau+s_n^{s'} ]\cup [\tau+s_n^{s'},\I))} =0.
\]
Furthermore if two sequences $\{\tau_n^1\}$ and $\{\tau_n^2\}$ are such that $s_n^s -\tau_n^1 \to \I$
and $\tau_n^2 - s_n^{s}\to \I$ as $n\to\I$ then
\[
\sup_{K\ge 0} \lim_{n\to\I} \norm{ \sum_{k=0}^K \l_n^{(s,k)} }_{L^8L^4((-\I, \tau_n^1 ]\cup [\tau_n^2,\I))} =0.
\]
\end{lemma}

\begin{proof}
We only treat the first estimate, as the other is handled similarly.  Note that for any $\eps>0$ there exists $K_0$ such that
\[
	 \sup_{K\ge K_0} \varlimsup_{n\to\I} \norm{ \sum_{k=K_0}^K \l_n^{(s,k)} }_{L^8L^4(\R)}
	\lesssim \(\sum_{k=K_0}^{\I} \norm{\varphi^{(s,k)}}_{H^1}^2\)^{1/2} \le \eps,
\]
thanks to Strichartz estimate and orthogonality.\footnote{Here (and below) we write $\varphi^{(s,k)}$ for the profile $\varphi^j_\infty$ with $(s,k)=(s(j),s(k))$.}  Thus, it suffices to show that
\[
\lim_{\tau\to\I} \varlimsup_{n\to\I}\norm{ \l_n^{(s,k)} }_{L^8L^4(-\I, -\tau+s_n^{s'} ]\cup [\tau+s_n^{s'},\I))}=0
\]
for each $k\ge0$. When $k=0$, one has $\l_n^{(s,0)}(t)= \mathcal{L}(t;s_n^s;z_n)\varphi^{(s,0)}$.  Thus, using Proposition~\ref{p:linearunifdecaytau} and interpolation,
\[
	\lim_{\tau\to \I}\varlimsup_{n\to\I} \norm{\mathcal{L}(\cdot ;0;z_n(\cdot + s_n^{s'} ))\varphi^{(s,0)}}_{L^8L^4 ([\tau,\I))}=0.
\]
On the other hand, when $|y_n^{(s,k)}|\to\I$ as $n\to\I$, the estimate follows from the fact that
\[
	\lim_{\tau\to\I}  \norm{ e^{-i t \Delta } \varphi^{(s,k)} }_{L^8L^4 \cap L^2 H^{1/2}_6(-\I, -\tau]\cup [\tau,\I))} =0
\]
in view of \eqref{e:spacelargetimebound12} and Lemma \ref{l:spacetranslation}.
\end{proof}

\subsection{Definition of nonlinear profiles} The next step is to define corresponding \emph{nonlinear profiles} for each $j$, as well as the \emph{nonlinear remainder} for each $J$.  The construction depends on the behavior of the parameters $(s_n^j, y_n^j)$. 

\begin{definition}[Nonlinear profiles for $y^j=0$]\label{d:nprofile1}
Suppose that 
\[
(s(j),k(j))=(s,0)\qtq{for some}s\in [s_{\min},s_{\max}].
\]
Passing to  a subsequence, we have 
\[
(z_n(t+s_n^j), \xi_n(t+s_n^j)) \rightharpoonup (z_\I^j(t) ,\xi_\I^j (t) )
\]
weakly in $\C \times H^1$, with uniform convergence on any compact time interval.  We then define 
\[
\Lambda_n^{(s,0)}(t) := \xi_\I^j (t-s_n^j),
\]
which solves \eqref{def:PDEODE} with $z_\infty^j(t-s_n^j)$.  We then define
\[
u_\I^s(t):= \Phi[z_\I^j(t)] + R[z_\I^j(t)]\xi_\I^j(t),
\]
which is a solution to \eqref{e:nls}.
\end{definition}

\begin{definition}[Nonlinear profiles for $y^j=\I$ and $\sigma^j=0$]\label{d:nprofile2}
Suppose that 
\[
(s(j),k(j))=(0,k)\qtq{for some}k\ge 1.
\] 
Then we define $\Lambda_n^{(0,k)}$ to be the solution to 
\[
\begin{cases}
& (i \d_t + H) \Lambda_n^{(0,k)} = B[z_n] \Lambda_n^{(0,k)} + \tilde{N} (z_n , \Lambda_n^{(0,k)} ),\\
& \Lambda_n^{(0,k)} (0) = P_c T_{y_n^j}\varphi_\infty^j,
\end{cases}
\]
where $\tilde N$ is as in \eqref{def:tildeN}. 
\end{definition}

The nonlinear profiles for the remaining case will be defined in terms of a parameter $\tau>0$, although we will not express this dependence explicitly.  Later, we will need to take a limit as $\tau\to\infty$. 

\begin{definition}[Nonlinear profiles for $y^j=\I$ and $\sigma^j=\I$]\label{d:nprofile3}
Suppose that 
\[
(s(j),k(j))=(s,k)\qtq{for some}s\in [1,s_{\max}]\qtq{and}k \ge 1.
\]
Let $v_n^j$ be the solution to
\[
\begin{cases}
&(i \d_t + H) v_n^j = B[z_n(\cdot + s_n^j)] v_n^j + \tilde{N} (z_n(\cdot + s_n^j), v_n^j), \\
& v_n^j(-\tau) = P_c T_{y_n^j}e^{i \tau \Delta}\varphi_\infty^j,
\end{cases}
\]
where $\tilde N$ is as in \eqref{def:tildeN}.  Then we define $\Lambda^{(s,k)}_n$ by
\[
\Lambda_n^{(s,k)}(t) :=  v_n^j(t-s_n^j).
\]
\end{definition}

Nonlinear profiles corresponding to $(s,k)$ with $s<0$ and $k\ge 1$ are defined similarly; however, we can omit their definition here, as we will only consider the positive time direction in what follows. 

\begin{definition}[Nonlinear Remainder] We define the \emph{nonlinear remainder} 
\[
\Gamma_n^J \in C(I_n; P_c H^1)
\]
to be the solution to
\[
\begin{cases}
&(i \d_t + H) \Gamma_n^J = B[z_n]\Gamma_n^J + \tilde{N} (z_n, \Gamma^J_n),\\
& \Gamma_n^J (0) = \tilde{\gamma}_n^J,
\end{cases}
\]
where $\tilde N$ is as in \eqref{def:tildeN}. 
\end{definition}

The main ingredient to complete the proof of Proposition~\ref{p:key} will be to establish the following proposition by a contradiction argument (recall the definition of $(M_*,E_*)$ in \eqref{MstarEstar}): 

\begin{proposition}\label{p:keykey}
There exist $s_+ \in [0,s_{\max}]$ and $s_- \in [s_{\min},0]$ such that
\begin{equation}\label{e:keykey}
(\mathbb{M} (u_\I^{s_\pm}) , \mathbb{E}_V (u_\I^{s_\pm})) 
= (M_*,E_*).
\end{equation}
\end{proposition}

We note that Proposition~\ref{p:Lambdas0bound} below guarantees that 
\[
\mathbb{M} (u_\I^s) \le M_*\qtq{and}\mathbb{E}_V (u_\I^{s}) \le E_*\qtq{for}s\in[s_{\min},s_{\max}].
\]
Thus, by \eqref{e:MEinside}, the failure of \eqref{e:keykey} guarantees scattering for $u_{\infty}^s$. 

\subsection{Estimates on the nonlinear remainder}

We begin by estimating the nonlinear remainder term.  We first need some estimates for the modified linear remainder $\tilde\gamma_n^J$.  Recall that $\eps,J^\dagger$ were introduced above, but have not yet been specified. 

\begin{lemma} We have the following estimates on $\tilde\gamma_n^J$:
\begin{align}
&\sup_{J > J^\dagger}\varlimsup_{n\to\I} \norm{ \tilde{\gamma}_n^J}_{ H^1 }^2  \le \sup_{n} \norm{\xi_n(0)}_{H^1}^2 + C_1^2\eps^2, \label{e:linear_remainder_smallness3} \\
&\sup_{J > J^\dagger} \varlimsup_{n\to\I} \norm{\mathcal{L}(\cdot,0;z_n) \tilde{\gamma}_n^J}_{(L^8_t \cap L^\I_t) L^4_x (\R)}  \le (1+C_1(C_2+C_3))\eps, \label{e:linear_remainder_smallness4}
\end{align} 
where $C_1,C_2$ are constants arising from the uniform Strichartz estimate and $C_3$ is the implicit constant the $H^1 \hookrightarrow L^4$ Sobolev embedding.
\end{lemma}

\begin{proof}
By the orthogonality of the linear profiles, for all $J\ge J^\dagger$ we have
\begin{align*}
\norm{\tilde{\gamma}_n^J - \gamma_n^{J^\dagger}}_{H^1}^2 & =\sum_{j \in S_J } \norm{  \l_n^j }_{H^1}^2  + o_n(1) \\
& \le C_1^2 \sum_{j > J^\dagger} \norm{  \varphi_\I^j }_{H^1}^2 + o_n(1) \le (C_1 \eps)^2 + o_n(1)\qtq{as}n\to\infty.
\end{align*}
Thus
\[
\norm{\mathcal{L}(\cdot,0;z_n) \tilde{\gamma}_n^J}_{L^8L^4 (\R)} \le \norm{\mathcal{L}(\cdot,0;z_n) \gamma_n^{J^\dagger}}_{L^8 L^4 (\R)} + C_2C_1 \eps + o_n(1),
\]
which implies
\[
\varlimsup_{n\to\I}\norm{\mathcal{L}(\cdot,0;z_n) \tilde{\gamma}_n^J}_{L^8 L^4 (\R)} \le (1+C_1C_2)\eps\qtq{for}J>J^\dagger.
\]

Similarly, by Sobolev embedding,
\[
\norm{\mathcal{L}(\cdot,0;z_n) \tilde{\gamma}_n^J}_{L^\I L^4 (\R)} \le \norm{\mathcal{L}(\cdot,0;z_n) \gamma_n^{J^\dagger}}_{L^\I L^4 (\R)} + C_3C_1 \eps + o_n(1)
\]
by Sobolev embedding.  Thus \eqref{e:linear_remainder_smallness4} follows from \eqref{e:linear_remainder_smallness1}.

Finally, 
\[
\norm{ \tilde{\gamma}_n^J}_{ H^1 }^2 \le \norm{ \tilde{\gamma}_n^{J^\dagger}}_{H^1}^2 + C_1^2 \sum_{j>J^\dagger} \norm{\varphi^j_\I}_{H^1}^2 + o_n(1),
\]
which yields
\[
\sup_{J > J^\dagger}\varlimsup_{n\to\I} \norm{ \tilde{\gamma}_n^J}_{ H^1 }^2 \le \sup_{n} \norm{\xi_n(0)}_{H^1}^2 + C_1^2\eps^2.
\]
\end{proof}

\begin{corollary} There exists $J^\dagger$ sufficiently large such that for any $J > J^\dagger$, there exists $N(J)$ such that if $n \ge N(J)$ then the nonlinear remainder $\Gamma_n^J$ is well-defined on $\R$ and obeys the following bounds
\begin{align*}
&\norm{\Gamma_n^J}_{L^8_tL^4_x (\R)} \lesssim \eps,\\
&\norm{\Gamma_n^J}_{\mathrm{Stz}^1 (\R)}\lesssim \sup_n\|z_n\|_{L_t^\infty}+ \sup_{n} \norm{\xi_n(0)}_{H^1}, \\
&\norm{ \Gamma_n^J- \mathcal{L}(\cdot,0;z_n)\tilde{\gamma}_n^J  }_{L^8 L^4(\R)} + \norm{ \Gamma_n^J }_{[z_n;0,\infty;\infty]} \ll \norm{ \mathcal{L}(\cdot,0;z_n)\tilde{\gamma}_n^J }_{L^8L^4(I_n)},
\end{align*}
where the implicit constants are independent of $n$ and $J$ and the semi-norm refers to the notation introduced in \eqref{nakanishi-seminorms}. Furthermore, 
\begin{equation}\label{e:epschoice}
\sup_{J > J^\dagger} \varlimsup_{n\to\I} \norm{\Gamma^J_n}_{L^\I L^4(\R)}\lesssim \eps.
\end{equation}
\end{corollary}
\begin{proof} The existence and first set of estimates follow from Lemma~\ref{l:sdt}, so we turn to \eqref{e:epschoice}.  As in the proof of Lemma \ref{l:sdt}, the Sobolev embedding $H^{1} \hookrightarrow L^4$ and the
uniform Strichartz estimate give us
\begin{align*}
\|\Gamma^J_n\|_{L^\I L^4}&\le \norm{\mathcal{L}(\cdot,0;z_n) \tilde{\gamma}^J_n}_{L^\I L^4} \\
& \quad  +C \norm{\Gamma_n^J}_{\mathrm{Stz}^1 (\R)}\norm{\mathcal{L}(\cdot,0;z_n) \tilde{\gamma}^J_n}_{L^8 L^4}(\norm{z_n}_{L^\I} +\norm{\mathcal{L}(\cdot,0;z_n) \tilde{\gamma}^J_n}_{L^8 L^4} ).
\end{align*}
Thus, using \eqref{e:linear_remainder_smallness4} and the bound on $\norm{\Gamma_n^J}_{\mathrm{Stz}^1}$, we find
\[
\norm{\Gamma^J_n}_{L^\I L^4(\R)}\lesssim \eps + \eps(z_* + \sup_{n} \norm{\xi_n(0)}_{H^1})(\norm{z_n}_{L^\I(\R)} +\eps) 	\lesssim \eps.
\]
for all $J  > J^\dagger$ and $n \geq N(J)$.\end{proof}

We next consider the linear evolution of $\tilde \gamma_n^J$.  We recall the notation $s_n^s$ introduced above (i.e. $s_n^s$ denotes the common sequence $s_n^j$ for all $j$ such that $s(j)=s$).

\begin{proposition} For any $s \in [s_{\min},s_{\max}]$, 
\begin{equation}\label{e:linear_remainder_smallness5}
\lim_{J\to\I} \sup_{\tau>0} \varlimsup_{n\to\I} \norm{  \mathcal{L}(\cdot, 0 ; z_n) \tilde{\gamma}_n^J }_{L^\infty_t L^4_x ([-\tau+s_n^s,\tau+s_n^s])}  =0.
\end{equation}
\end{proposition}

\begin{proof} If instead we have
\[
\varlimsup_{J\to\I} \sup_{\tau>0} \varlimsup_{n\to\I} \norm{ \mathcal{L}(\cdot, 0 ; z_n) \tilde{\gamma}_n^J }_{L^\infty_t L^4_x ([-\tau+s_n^s,\tau+s_n^s])} \ge c_0
\]
for some $s \in [s_{\min},s_{\max}]$ and $c_0>0$, then we may find a subsequence of $J\to\infty$ and $\tau=\tau(J)>0$ such that 
\[
	\norm{ \mathcal{L}(\cdot, 0 ; z_n) \tilde{\gamma}_n^J }_{L^\infty_t L^4_x ([-\tau+s_n^s,\tau+s_n^s])} \ge \tfrac12{c_0}
\]
along a subsequence of $n$ (depending on $J$).

Now, given $N\in2^{\Z}$, we have 
\[
\norm{ P_{\ge N}\mathcal{L}(\cdot, 0 ; z_n) \tilde{\gamma}_n^J }_{L^4} \lesssim \norm{ P_{\ge N}\mathcal{L}(\cdot, 0 ; z_n) \tilde{\gamma}_n^J }_{H^{\frac34}(\R)} \lesssim N^{-\frac14} .
\]
Thus, for $N$ is sufficiently large, we see that
\[
\norm{ P_{<N}\mathcal{L}(\cdot, 0 ; z_n) \tilde{\gamma}_n^J }_{L^\infty_t L^4_x ([-\tau+s_n^s,\tau+s_n^s])} \ge \tfrac14{c_0}.
\]
On the other hand, by H\"older and Strichartz, we have
\begin{align*}
\| P_{<N}&\mathcal{L}(\cdot, 0 ; z_n) \tilde{\gamma}_n^J \|_{L^\infty_t L^4_x ([-\tau+s_n^s,\tau+s_n^s])}\\
&\le \norm{ \mathcal{L}(\cdot, 0 ; z_n) \tilde{\gamma}_n^J }_{L^\infty_t L^2_x(\R)}^{\frac12} \norm{ P_{<N}\mathcal{L}(\cdot, 0 ; z_n) \tilde{\gamma}_n^J }_{L^\I_{t ,x} ([-\tau+s_n^s,\tau+s_n^s]\times \R^3)}^{\frac12} \\
&\lesssim  \norm{ P_{<N}\mathcal{L}(\cdot, 0 ; z_n) \tilde{\gamma}_n^J }_{L^\I_{t ,x} ([-\tau+s_n^s,\tau+s_n^s]\times \R^3)}^{\frac12},
\end{align*}
and hence
\[
\norm{ P_{<N}\mathcal{L}(\cdot, 0 ; z_n) \tilde{\gamma}_n^J }_{L^\I_{t ,x} ([-\tau+s_n^s,\tau+s_n^s]\times \R^3)} \gtrsim_{c_0} 1.
\]

In particular, for each $J$ in the subsequence, there exist $x_n^J\in\R^3$ and 
\[
\tilde s_n^J\in [-\tau+s_n^s,\tau+s_n^s]
\]
such that $T_{x_n^J}^{-1}\mathcal{L}(\tilde s_n^J,0;z_n)\tilde \gamma_n^J\rightharpoonup \tilde\psi^J \in H^1$ weakly in $H^1$, where
\[
\inf_{J} \|\tilde\psi^J\|_{H^1}\gtrsim 1.
\]
We now write
\[
T_{x_n^J}^{-1}\mathcal{L}( s_n^s,0;z_n)\tilde\gamma_n^J  = [T_{x_n^J}^{-1}\mathcal{L}(s_n^s,\tilde s_n^J;z_n)T_{x_n^J}][T_{x_n^J}^{-1}\mathcal{L}(\tilde s_n^J,0;z_n)\tilde\gamma_n^J]
\]
and observe that since $\sup_n |s_n^s - \tilde s_n^J|\leq \tau<\infty$, the first operator above converges strongly to an operator on $H^1$ with norm $\sim 1$ (with precise value depending on $\sup_n \|z_n\|_{L^\infty}$).  Thus, we deduce that there exists $\psi^J \in H^1$ such that
\[
T_{x_n^J}^{-1}\mathcal{L}(s_n^s,0;z_n)\tilde\gamma_n^J \rightharpoonup\psi^J,\qtq{with}\inf_J \|\psi^J\|_{H^1}\gtrsim 1. 
\]
By the usual adjustments, we may also assume that $x_n^J\equiv 0$ or $|x_n^J|\to\infty$. 

We now claim that for each $J$ in the subsequence, there exists $c\in\R^3$ and $k_0\in[J+1,\infty)$ so that
\[
T_c \psi^J = \varphi_\I^{k_0},\qtq{with}s(k_0)\in[s_{\min},s_{\max}],
\]
where $\varphi_\I^{k_0}$ is the function given in the profile decomposition of $\{\xi_n(0)\}$.  With this claim, we thereby obtain the contradiction
\[
\tnorm{\psi^J}_{H^1} = \tnorm{\varphi_\I^{k_0}}_{H^1} 
\le \biggl(\sum_{j=J+1}^\infty \norm{\varphi_\infty^j}_{H^1}^2\biggr)^{\frac12}\to 0 \qtq{as}J\to\infty.
\]
%

By construction, we may write
\begin{align*}
\tilde\gamma_n^J & = \gamma_n^{J^\dagger} - \sum_{j\in S_J} \mathcal{L}(0,s_n^j;z_n)P_c T_{y_n^j}\varphi_\infty^j \\
& = \gamma_n^J + \sum_{j\in(J^\dagger,J]\backslash S_J} \mathcal{L}(0,s_n^j;z_n)P_c T_{y_n^j}\varphi_\infty^j.
\end{align*}
Now, for any $j\in(J^\dagger,J]\backslash S_J$, we have $|s_n^s-s_n^j|\to\infty$, and hence by Lemma~\ref{l:orth}(ii), 
\[
T_{x_n^J}^{-1}\mathcal{L}(s_n^s,0;z_n)\mathcal{L}(0,s_n^j;z_n)P_cT_{y_n^j}\varphi^j_\infty \rightharpoonup 0 \qtq{as}n\to\infty,
\]
which implies
\[
T_{x_n^J}^{-1}\mathcal{L}(s_n^s,0;z_n)\gamma_n^J\rightharpoonup \psi^J. 
\]
Next, we recall that by construction, we have
\[
\lim_{K\to\infty} \eta(\{z_n\},\{\gamma_n^K\}) = 0,
\]
and thus we may find $K_0>J$ so that $\eta(\{z_n\},\{\gamma_n^{K_0}\})<\tfrac{1}{10}\|\psi^J\|_{H^1}$.  Then, writing
\[
T_{x_n^J}^{-1}\mathcal{L}(s_n^s,0;z_n)\gamma_n^{K_0}=T_{x_n^J}^{-1}\mathcal{L}(s_n^s,0;z_n)\gamma_n^J - \sum_{j=J+1}^{K_0} T_{x_n^J}^{-1}\mathcal{L}(s_n^s,s_n^j;z_n)P_c T_{y_n^j} \varphi_\infty^j,
\]
it follows that there exists $k_0\in[J+1,K_0]$ and $g\in H^1\backslash\{0\}$ such that
\[
T_{x_n^J}^{-1}\mathcal{L}(s_n^s,s_n^{k_0};z_n)P_c T_{y_n^{k_0}}\varphi_\infty^{k_0} \rightharpoonup g. 
\]
By Lemma~\ref{l:orth}(i)--(ii) and construction, this implies that $s_n^s\equiv s_n^{k_0}$ and $\sup_n |x_n^J-y_n^{j_0}|<\infty$.  In particular, we deduce that $s(k_0)\in[s_{\min},s_{\max}]$, and that 
\[
T_{y_n^{k_0}}^{-1} \mathcal{L}(s_n^{k_0},0;z_n) \gamma_n^J \rightharpoonup T_c \psi^J \qtq{for some}c\in\R^3. 
\]
However, again by construction, we have that
\[
T_{y_n^{k_0}}^{-1}\mathcal{L}(s_n^{k_0},0;z_n)\gamma_n^J \rightharpoonup \varphi_\infty^{k_0}, 
\]
and thus $T_c\psi^J=\varphi_\infty^{k_0}$, as desired. 
\end{proof}

We now estimate the nonlinear remainder.  This estimate and its proof are closely related to \cite[Equation (7.18)]{Nakanishi} and the proof therein.

\begin{proposition}\label{p:7.18} If $\mathcal{N}_0=\sup_n  \norm{z_n}_{L^\I} + \mu_0$ is sufficiently small and $J^\dagger$ is sufficiently large, then for any $s \in [s_{\min},s_{\max}]$,
\begin{equation}\label{e:7.18b}
\sup_{\tau>0} \lim_{J\to\I}  \varlimsup_{n\to\I}\norm{ \Gamma_n^J }_{L^8_t L^4_x([-\tau+s_n^s, \tau + s_n^s])}=0.
\end{equation}
Furthermore, for any $s\in [1,s_{\max}]$,
\begin{equation}\label{e:7.18a}
	\sup_{\tau>0} \lim_{J\to\I}  \varlimsup_{n\to\I}
	\norm{ \mathcal{L}(\cdot,-\tau + s_n^s;z_n) \Gamma_n^J(-\tau+s_n^s) }_{L^8_t L^4_x([-\tau+s_n^s, \tau + s_n^s])}
=0.
\end{equation}

\end{proposition}

\begin{proof} Following \cite[Proof of (7.18)]{Nakanishi}, we introduce 
\[
\norm{ f }_{X^j_n} := \sup_{ t\in \R } \Jbr{ t- s^j_n }^{-\delta} \norm{ f(t) }_{(L^4+L^p)_x},
\]
for some $p>8$ and $\delta\in(\tfrac18,\tfrac12-\tfrac3p)$ (e.g. $(p,\delta)=(10,\tfrac16)$). 


Recalling \eqref{e:linear_remainder_smallness4}, we have that for any $\tau>0$ and $J \ge J^\dagger+1$,
\begin{align*}
\varlimsup_{n\to\I}  \norm{ \mathcal{L} (\cdot,0;z_n) \tilde{\gamma}_n^J  }_{X^j_n} &\lesssim \varlimsup_{n\to\I}\norm{  \mathcal{L}(\cdot, 0 ; z_n) \tilde{\gamma}_n^J }_{L^\infty_t L^4_x ([-\tau+s_n^j,\tau+s_n^j])}\\
&\quad + \tau^{-\delta}\varlimsup_{n\to\I} \norm{ \mathcal{L} (\cdot,0;z_n) \tilde{\gamma}_n^J  }_{L^\I L^4_x(\R)} \\
&\lesssim \eps \tau^{-\delta} + \varlimsup_{n\to\I}\norm{  \mathcal{L}(\cdot, 0 ; z_n) \tilde{\gamma}_n^J }_{L^\infty_t L^4_x ([-\tau+s_n^j,\tau+s_n^j])}.
\end{align*}
Thus, letting first $J\to \I$ and then $\tau \to\I$, we see from \eqref{e:linear_remainder_smallness5} that
\begin{equation}\label{e:7.18pf1}
	\lim_{J\to\I} \lim_{n\to\I}  \norm{ \mathcal{L} (\cdot,0;z_n) \tilde{\gamma}_n^J  }_{X^j_n} = 0.
\end{equation}

It suffices consider the case $\sigma^j=\infty$ (the other case is similar). Using the dispersive estimate for $\mathcal{L}$ (see Lemma~\ref{L:weighted-dispersive}), we obtain
\[
	\norm{ \Gamma_n^J(t) - \mathcal{L} (t,0;z_n) \gamma_n^J}_{L^4+L^p}
	\lesssim \int_0^t h(t-t') \norm{ \tilde{N} (z_n, \Gamma^J_n)}_{L^{4/3} \cap L^{p'}} dt',
\]
where $h(t) = \min\{|t|^{-(\frac32-\frac3p)}, |t|^{-\frac34}\}.$  We now observe that
\begin{align*}
&\norm{ \tilde{N} (z_n ,\Gamma^J_n)}_{L^{4/3} }\lesssim \norm{\Gamma^J_n}_{L^4 + L^p} \norm{\Gamma^J_n}_{L^4 \cap L^{\frac{2p}{p-2}}} (|z_n| + \norm{\Gamma^J_n}_{L^4}), \\
&\norm{ \tilde{N} (z_n ,\Gamma^J_n)}_{L^{p'} }\lesssim \norm{\Gamma^J_n}_{L^4 + L^p} \norm{\Gamma^J_n}_{L^4 \cap L^{\frac{2p}{p-2}}} (|z_n| + \norm{\Gamma^J_n}_{L^{\frac{2p}{p-2}}}),
\end{align*}
which (together with \eqref{e:epschoice}) imply
\begin{align*}
\| \tilde{N} (z_n ,\Gamma^J_n)\|_{L^{4/3} \cap L^{p'}} &\lesssim \bigl[|z_n|^2+\norm{\Gamma^J_n}_{L^4 \cap L^2}^2\bigr]\norm{\Gamma^J_n}_{L^4 + L^p}  \\
&\lesssim \bigl[\mathcal{N}_0^2 + \norm{\Gamma^J_n}_{L^\I L^4}^2\bigr] \norm{\Gamma^J_n}_{L^4 + L^p} \\
&\lesssim\bigl[\mathcal{N}_0^2 +\eps^2\bigr] \norm{\Gamma^J_n}_{L^4 + L^p}.
\end{align*}

We therefore obtain
\begin{align*}
\norm{ \Gamma_n^J(t) - \mathcal{L} (t,0;z_n) \tilde{\gamma}_n^J}_{L^4+L^p}&\lesssim (\mathcal{N}_0^2+\eps^2) \norm{\Gamma^J_n}_{X_n^j} \int_{\R}  \Jbr{t'-s_n^j}^\delta dt' \\
&\lesssim (\mathcal{N}_0^2+\eps^2) \norm{\Gamma^J_n}_{X_n^j}  \Jbr{t-s_n^j}^\delta,
\end{align*}
where we have used $-3(\frac12-\frac1p)+\delta <-1$. 
Thus, if $\mathcal{N}_0+\eps$ is sufficiently small, 
\begin{equation}\label{e:epschoice2}
\norm{ \Gamma_n^J - \mathcal{L} (\cdot,0;z_n) \tilde{\gamma}_n^J}_{X_n^j} \le \tfrac12\norm{ \Gamma_n^J}_{X_n^j}.
\end{equation}
In particular, we see from \eqref{e:7.18pf1} that
\begin{equation}\label{e:7.18pf2}
\varlimsup_{n\to\I} \norm{ \Gamma_n^J}_{X_n^j} \le 2\varlimsup_{n\to\I} \norm{ \mathcal{L} (t,0;z_n) \tilde{\gamma}_n^J}_{X_n^j}  \to 0\qtq{as}J\to\infty.
\end{equation}

Now, by H\"older's inequality, we have
\[
\norm{ \Gamma_n^J }_{L^4} \le \norm{ \Gamma_n^J }_{L^4 + L^p }^{\frac12} \norm{ \Gamma_n^J }_{L^4 \cap L^{\frac{2p}{p-2}}}^{\frac12}.
\]
Thus, for any $\tau >1$,
\begin{align*}
	\norm{ \Gamma_n^J }_{L^8 L^4 ([s_n^j -\tau,s_n^j + \tau])}
	&{}\le \tau^{\delta/2}  \norm{ \Gamma_n^J}_{X_n^j}^\frac12 \norm{\Gamma_n^J}_{L^{4}(L^4 \cap L^{\frac{2p}{p-2}})([s_n^j -\tau,s_n^j + \tau]) }^\frac12 \\
	&{}\lesssim \tau^{\delta/2}  \norm{ \Gamma_n^J}_{X_n^j}^\frac12 ( \norm{\Gamma_n^{J}}_{L^4 H^{\frac16}_{3} }+ \tau^{\frac{4p}{p-6}} \norm{\Gamma_n^{J}}_{L^{\frac{2p}{3}} L^{\frac{2p}{p-2}} })^\frac12 \\
	&{}\lesssim_\tau \norm{ \Gamma_n^J}_{X_n^j}^\frac12   \norm{\Gamma_n^{J}}_{\mathrm{Stz}^1(\R)}^\frac12.
\end{align*}
Using \eqref{e:7.18pf2} and boundedness in $\mathrm{Stz}^1$, we obtain the desired estimate \eqref{e:7.18b}.

Finally, since $\Gamma_n^J$ is a solution to the nonlinear equation, we see from Proposition~\ref{l:sdt} that \eqref{e:7.18b} implies \eqref{e:7.18a}. \end{proof}

With the estimates above in place, we can finally discuss our choice of $\eps$ (and hence $J^\dagger$) more precisely. 

\begin{remark}[Choice of $J^\dagger$]\label{r:choiceofeps}
We choose $\eps>0$ sufficiently small so that \eqref{e:epschoice2} is valid.  The implicit constants that are used to obtain this estimate depend on the choice of $p, \delta$ (which we may fix to be $(p,\delta)=(10,\tfrac16)$, say), the constants in uniform Strichartz estimate, and that in \eqref{e:epschoice}.  The implicit constant in \eqref{e:epschoice} depends on the constants in uniform Strichartz estimate, $\sup_n \norm{z_n}_{L^\I}$, and $\sup_n \norm{ \xi_n (0) }_{H^1}$. The constants in the uniform Strichartz estimate in turn depend on $\sup_n \norm{z_n}_{L^\I}$.  Thus, we find that the choice of $\eps$ depends only on $\sup_n \norm{ \xi_n (0) }_{H^1}$ and $\sup_n \norm{z_n}_{L^\I}$.
\end{remark}

\subsection{Estimates on nonlinear profiles} We will next derive estimates on the nonlinear profiles.

\subsubsection{Estimate on $\Lambda_n^{(s,0)}$}

Let us begin with the profile $\Lambda_n^{(s,0)}$, defined as in Definition~\ref{d:nprofile1}.  Recalling the notation there, we first note that by $H^1$-subcriticality of \eqref{e:nls} and uniform $H^1$-boundedness, we have that $(z_\I^s ,\xi_\I^s )$ is global in time.

\begin{proposition}\label{p:Lambdas0bound}
Suppose that \eqref{e:fakeassumption} holds.  Let $s \in [s_{\min},s_{\max}]$ and let $(z_\I^s,\xi_\I^s)$ and $u_\I^s$ be as in Definition \ref{d:nprofile1}. Then, for all $s \in [s_{\min},s_{\max}]$, one has
\begin{equation}\label{e:edecrease}
\mathbb{M}(u_\I^s) \le M_*, \quad \mathbb{E}_V (u_\I^s) \le E_*,
\end{equation}
and $\norm{u_\I^s}_{L^\I H^1 (\R)}^2 \lesssim \mathbb{E}_V (u_\I^s) + \mathbb{M}(u_\I^s).$

If $u^s_\I$ scatters to $\mathscr{S}_0$ forward in time, then
\[
\lim_{\tau \to \I} \bigl\{ \norm{\xi_\I^s}_{L^8L^4([\tau,\I)) } + \norm{\mathcal{L}(\cdot, \tau;z_\I^s) \xi_\I^s(\tau)}_{L^8L^4([\tau,\I)) }  \bigr\}= 0.
\]
The analogous statement holds backward in time. 

Finally, if 
\[
(\mathbb{M}(u_\I^s) , \mathbb{E}_V (u_\I^s)) \neq ( M_*,   E_*)
\]
then $u_\I^s(t)$ scatters to $\mathscr{S}_0$ in both time directions and $\norm{\xi_\I^s}_{L^8_tL^4_x(\R)} <\I$.
\end{proposition}

\begin{proof}  We begin with \eqref{e:edecrease}. We apply a profile decomposition with $L^4$ error to the bounded sequence $\{ u_n(s_n^s) \}_n \subset H^1$ to get a decomposition of the form
\[
 	u_n(s_n^s) = u_\I^s + \sum_{k=1}^{J} T_{y_n} \varphi^k + r_n^J.
\]
We have
\begin{align*}
&M_* = \lim_{n\to\I}\mathbb{M}(u_n(s_n^s)) \ge \mathbb{M}(u_\I^s) + \sum_{k\ge1} \mathbb{M}(\varphi^{k}),\\
&E_* = \lim_{n\to\I}\mathbb{E}_V(u_n) \ge \mathbb{E}_V(u_\I^s) + \sum_{k\ge1} \mathbb{E}_0(\varphi^{k}),\\
&\lim_{n\to\I}\mathbb{I}_V (u_n) = \mathbb{I}_V (u_\I^s) + \sum_{k\ge1} \mathbb{I}_0(\varphi^{k}) .
\end{align*}
Here the second and the third decomposition follows from the decomposition of $\mathbb{G}(u_n)$. The first implies $\mathbb{M}(u_\I^s) \le M_*$, yielding the former part of \eqref{e:edecrease}.

We next show that $\mathbb{E}_0(\varphi^1) \ge 0$.  It suffices to consider the case $\varphi^1\neq 0$.  By assumption, $\mathbb{K}_{V,2}(u_n)\ge \kappa >0$, which (together with \eqref{e:basicEMrelation} and the identity $\mathbb{I}_0(Q)=\mathbb{E}_0(Q)$) yields
\[
\lim_{n\to\I}\mathbb{I}_V (u_n) = \lim_{n\to\I} \bigl[\mathbb{E}_V (u_n) - \tfrac12\mathbb{K}_{V,2}(u_n)\bigr] \le E_* -\kappa < \tfrac1{M_*} \mathbb{M}(Q)  \mathbb{I}_0 (Q) -\kappa.
\]
Furthermore, 
\[
\mathbb{I}_V (u_\I^s) \ge - \tfrac12\norm{(x\cdot \nabla V + 2 V)_-}_{L^\I} \mathbb{M}(u_\I^s),
\]
so that
\[
 \mathbb{I}_0(\varphi^{1}) < \tfrac1{M_*} \mathbb{M}(Q)  \mathbb{I}_0 (Q) -\kappa + \tfrac12\norm{(x\cdot \nabla V + 2 V)_-}_{L^\I} \mathbb{M}(u_\I^s).
\]
As $\mathbb{M}(\varphi^{1}) \le M_* - \mathbb{M}(u_\I^s)$, we obtain
\begin{align*}
\mathbb{M}(\varphi^1) \mathbb{I}_0(\varphi^{1})<{}& \tfrac{ M_* - \mathbb{M}(u_\I^s)}{M_*} \mathbb{M}(Q)  \mathbb{I}_0 (Q) \\
&+ \tfrac12\norm{(x\cdot \nabla V + 2 V)_-}_{L^\I}(M_* - \mathbb{M}(u_\I^s)) \mathbb{M}(u_\I^s).
\end{align*}
Now observe that if $M_* \lesssim_{Q,V} 1$, then the right hand side is decreasing in $\mathbb{M}(u_\I^s) \in [0, M_*]$. Hence we obtain the bound
\[
\mathbb{M}(\varphi^{1}) \mathbb{I}_0(\varphi^{1}) < \mathbb{M}(Q)  \mathbb{I}_0 (Q).
\]
As 
\[
\mathbb{M}(Q)  \mathbb{I}_0 (Q) = \inf \{ \mathbb{M}(\varphi)  \mathbb{I}_0 (\varphi) \ |\ \mathbb{K}_{0,2} (\varphi) \le 0, \,\varphi \neq0\},
\]
we see that 
\begin{equation}\label{e:K2phi1}
\mathbb{K}_{0,2}(\varphi^{1})>0.
\end{equation}
Thus
\begin{equation}\label{e:Ephi10}
\mathbb{E}_{0}(\varphi^{1}) = \mathbb{H}_{0}(\varphi^{1})-  \mathbb{G}(\varphi^{1}) = \tfrac13 \mathbb{H}_{0}(\varphi^{1}) + \tfrac13 \mathbb{K}_{0,2}(\varphi^{1}) >0,
\end{equation}
and so we obtain $\mathbb{E}_{0}(\varphi^{1}) \ge 0$ also in the case $\varphi^1 \neq 0$.

By the same argument, we see that $\mathbb{E}_{0}(\varphi^{k})\ge0$ for all $k\ge1$, which yields $
{\mathbb{E}_V (u_{\infty}^{s})} \le E_*$ (i.e. the latter part of \eqref{e:edecrease}).  By the variational characterization of $\mathscr{E}_1(\mu)$, we see that $\mathbb{K}_{V,2}(u_\I^s)>0$, and hence $u_\I^s$ is a global solution obeying the bound $\norm{u}_{L^\I H^1}^2 \lesssim \mathbb{E}_V(u_\I^s) + \mathbb{M}(u_\I^s).$  Scattering {of $u_{\infty}^{s}$} in the case $( \mathbb{M}(u_\I^s),\mathbb{E}_V(u_\I^s) )\neq (M_*,E_*)$ follows from \eqref{e:MEinside}.
\end{proof}

\begin{remark}
In the above proof,
the upper bound for $\mathbb{E}_0(\varphi^{1} )$ is obtained similarly. Indeed, using the facts that $\mathbb{E}_V(u_\I) \ge \mathscr{E}_0(\mathbb{M}(u_\I)) $ and $\mathbb{E}_0(\varphi^{k})\ge 0$ for all $k\ge 2$, one has
\[
\mathbb{E}_0(\varphi^{1}) \le E_* - \mathscr{E}_0(\mathbb{M}(u_\I)).
\]
Hence, we obtain the bound
\begin{equation}\label{e:Ephi11}
\mathbb{E}_0(\varphi^{1}) \le E_* - \mathscr{E}_0(M_*).
\end{equation}
By using \eqref{e:basicEMrelation}, we may also obtain
\[
\mathbb{M}(\varphi^{1})  \mathbb{E}_0(\varphi^{1}) < \tfrac{M_* - \mathbb{M}(u_\I)}{M_*} \mathbb{M}(Q)  \mathbb{E}_0 (Q) -  \mathcal{E}_0(\mathbb{M}(u_\I))(M_* - \mathbb{M}(u_\I)) -\delta_0\mathbb{M}(\varphi^{1}) .
\]
As $\mathcal{E}_0(\mathbb{M}(u_\I)) \sim e_0 \mathbb{M}(u_\I)$, one sees that the right hand side is decreasing in $\mathbb{M}(u_\I) \in [0,M_*]$ if $M_* \lesssim_{Q,e_0} 1$.  Thus we obtain
\begin{equation}\label{e:Ephi12}
\mathbb{E}_0(\varphi^{1}) < \frac1{\mathbb{M}(\varphi^{1})} \mathbb{M}(Q)  \mathbb{E}_0 (Q) - \delta_0.
\end{equation}
We will use these estimates in the estimate of $\Lambda_n^{(0,k)}$.
\end{remark}

\subsubsection{Estimate on $\Lambda_n^{(0,k)}$}

To estimate the $\Lambda_n^{(0,k)}$, the idea is to compare the solution to the translation of the solution to the cubic NLS without potential:
\begin{equation}\label{e:fnls_np}
\begin{cases}
 i \d_t w_\I^{(0,k)} - \Delta w_\I^{(0,k)} = \sigma |w_\I^{(0,k)}|^2 w_\I^{(0,k)},\\
 w_\I^{(0,k)}(0)= \varphi^{(0,k)},
\end{cases}
\end{equation}
where we have written $\varphi^{(0,k)}$ to denote the profile $\varphi_\infty^j$ appearing in Definition~\ref{d:nprofile2}.  We first collect some relevant properties of $w_\I^{(0,k)}$.

\begin{lemma}\label{l:wkest}
Suppose \eqref{e:fakeassumption} holds.  Let $k \ge 1$ and let
$\Lambda_n^{(0,k)}(t)$ be as in Definition~\ref{d:nprofile2}.
If $\varphi^{(0,k)}\neq 0$ then $\mathbb{M}(\varphi^{(0,k)}) \le M_*$,
\begin{equation}\label{e:wkenergy}
	0 < \mathbb{E}_0 (\varphi^{(0,k)}) \le \min \bigl[ E_* - \mathscr{E}_0(M_*), \tfrac1{\mathbb{M}(\varphi^{(0,k)})} \mathbb{M}(Q)  \mathbb{E}_0 (Q) - \delta_0 \bigr],
\end{equation}
and
\[
\mathbb{K}_{0,2} (\varphi^{(0,k)}) >0.
\]
In particular, there exists a unique global solution $w_\I^{(0,k)}(t) \in C(\R,H^1)$ to \eqref{e:fnls_np}
such that
\[
\norm{w_\I^{(0,k)}}_{L^8_tL^4_x (\R)} \lesssim_{\delta_0} 1, \qtq{and} \norm{w_\I^{(0,k)}}_{L^\I_t H^1 (\R)}^2 \lesssim  \mathbb{E}_0 (\varphi^{(0,k)}) +  \mathbb{M}(\varphi^{(0,k)}).
\]
\end{lemma}

\begin{proof} We apply the $H^1\hookrightarrow L^4$ profile decomposition to obtain
\[
u_n(0) = u_\I + \sum_{\ell=1}^L T_{\hat{y}_n^\ell} \hat{\varphi}^{\ell} + \hat{r}_n^L.
\]

As
\begin{align*}
 (T_{y_n^{(0,k)}})^{-1} u_n ={}& (T_{y_n^{(0,k)}})^{-1}\Phi[z_n] + (T_{y_n^{(0,k)}})^{-1}\xi_n + (T_{y_n^{(0,k)}})^{-1}(R[z_n]-1)\xi_n \\
 \rightharpoonup{}& 0+ \varphi^{(0,k)} +0 = \varphi^{(0,k)} \neq0
\end{align*}
weakly in $H^1$ as $n\to\I$, it follows that $T_{y_n^{(0,k)}} \varphi^{(0,k)}$ appears in the above decomposition.  That is, there exists $\ell_0$ and $y_\delta:=\lim_{n\to\infty}[\hat y_n^{\ell_0}-y_n^{(0,k)}]\in\R^3$ such that $\varphi^{(0,k)} = T_{y_\delta } \hat{\varphi}^{\ell_0}.$ Thus, arguing as in the proof of \eqref{e:K2phi1}, we obtain $\mathbb{K}_0 (\varphi^{(0,k)}) >0$.  Similarly, we obtain the energy bound from \eqref{e:Ephi10}, \eqref{e:Ephi11}, and \eqref{e:Ephi12}.

The global existence of $w_\I^{(0,k)}$ therefore follows from the well-posedness theory for the standard cubic NLS (see e.g. \cite{DHR}). \end{proof}

For the $\Lambda_n^{(0,k)}$, we will prove the following: 

\begin{proposition}\label{p:Lambda0kbound} Suppose that \eqref{e:fakeassumption} holds and fix $k\geq 1$.  Let $\Lambda_n^{(0,k)}$ be as in Definition \ref{d:nprofile2}, and let $w_\I^{(0,k)}\in C(\R,H^1)$ be the unique global solution to \eqref{e:fnls_np} given in Lemma~\ref{l:wkest}. Writing $y_n$ for the sequence $y_n^{(0,k)}$, we have 
\begin{align}
\lim_{n\to\I} \norm{\Lambda_n^{(0,k)} - T_{y_n} w_\I^{(0,k)}}_{L^8L^4(\R)}&=0,\label{e:Lambda0kbound1}\\
 \lim_{R\to\I} \varlimsup_{n\to\I} \norm{ {\bf 1}_{ \{|x-y_n|\ge R \}} \Lambda_n^{(0,k)} }_{L^8_t L^4 (\R)} &=0,\label{e:Lambda0kbound2}
\end{align}
and
\begin{multline}\label{e:Lambda0kbound3}
	\lim_{\tau \to\I} \varlimsup_{n\to\I} \biggl\{ \norm{ \mathcal{L}(\cdot,-\tau;z_n)\Lambda_n^{(0,k)}(-\tau) }_{L^8_tL^4_x ((-\I,-\tau))} \\
+ \norm{ \mathcal{L}(\cdot,\tau;z_n)\Lambda_n^{(0,k)}(\tau) }_{L^8_tL^4_x ((\tau,\I))} + \norm{\Lambda_n^{(0,k)}}_{L^8_tL^4_x ((-\I,-\tau)\cup(\tau,\I))}\biggr\}=0.
\end{multline}

Furthermore, for any $\tau>0$,
\begin{equation}\label{e:Lambda0kbound4}
\lim_{n\to\I} \norm{ \Lambda_n^{(0,k)} }_{L^8_t L^4 ([-\tau,\tau] )} = \norm{ w_\I^{(0,k)} }_{L^8_t L^4 ([-\tau,\tau] )}.
\end{equation}
\end{proposition}

We begin with the following lemma.
\begin{lemma}\label{L:818}
Let $\varphi \in H^1$ and let $\{y_n\}\subset\R^3$ be a sequence such that $|y_n|\to\I$ as $n\to\I$, and let $\{z_n\}\subset\mathrm{SBC}$.  Let $v_n$ be the solution to the first equation in \eqref{def:PDEODE} with $v_n(0)= P_c T_{y_n}  \varphi$. 

Let $A>0$ and $\frac12 < \theta_1 < \theta_2 \le 1$.
If
\begin{equation}\label{e:unifbddasmp}
\sup_{n} \norm{v_n}_{\mathrm{Stz}^{\theta_2}([-\tau,\tau]))} \le A
\end{equation}
for some $\tau>0$, then $v_n$ satisfies the following:
\begin{align}\label{e:trial2.claim1}
\lim_{R\to\I} \varlimsup_{n\to\I}  \norm{ v_n }_{L^\I _tL^2_x ([-\tau,\tau]\times \{ |x-y_n| \ge R \})} &=0,\\
\label{e:trial2.claim2}
\lim_{n\to\I} \norm{ \int_0^t e^{i(t-s)\Delta} (i \d_t v_n - \Delta v_n - \sigma |v_n|^2 v_n)\,ds  }_{\mathrm{Stz}^{\theta_1} ([-\tau,\tau])} &= 0.
\end{align}
\end{lemma}

\begin{proof}[Proof of Lemma~\ref{L:818}] As in the proof of Lemma~\ref{l:spacetranslation}, we let $\chi_{n,r}(x)=\chi(\tfrac{x-y_n}{r})$ be a smooth function obeying
\[
\chi_{n,r}(x) = \begin{cases} 0 & |x-y_n|<r \\ 1 & |x-y_n|>2r \end{cases}\qtq{and} \|\nabla \chi_{n,r}\|_{L^\infty}\lesssim \tfrac{1}{r}.
\]

We first note that
\begin{align*}
\d_t |v_n|^2 = 2\nabla \cdot \Im (\overline{v_n} \nabla v_n) + 2\Im ( \overline{v_n} B[z_n] v_n )+ 2\Im ( \overline{v_n} \tilde{N}(z_n,v_n) ),
\end{align*}
which yields
\begin{equation}\label{e:trial2.2}
\begin{aligned}
\d_t \int \chi_{n,r}  |v_n|^2 dx ={}& - 
2\int \nabla \chi_{n,r} \cdot \Im (\overline{v_n} \nabla v_n)\,dx
+2\int \chi_{n,r} \Im ( \overline{v_n} B[z_n] v_n )\,dx\\
& +2\int \chi_{n,r} \Im ( \overline{v_n} \tilde{N}(z_n,v_n))\,dx.
\end{aligned}
\end{equation}

For the first term in \eqref{e:trial2.2}, we observe that
\[
\abs{ \int \nabla \chi_{n,r} \cdot \Im (\overline{v_n} \nabla v_n) dx } \lesssim \tfrac1r \norm{v_n}_{H^{\frac12}}^2,
\]
which we prove by interpolation: It suffices to prove that $|\nabla|^s [\nabla\chi_{n,r}] |\nabla|^{-s}$ maps $L^2\to L^2$ with norm $\lesssim r^{-1}$ for $s=\tfrac12$, which in turn follows provided we verify this bound for $s\in\{0,1\}$.  The case $s=0$ is straightforward, while for $s=1$ we use the product rule, Sobolev embedding, and the fact that $\|\Delta \chi\|_{L^3}\lesssim r^{-1}$. 

The second term in \eqref{e:trial2.2} is estimated as in the proof of Lemma \ref{l:spacetranslation}.  In particular, one has
\[
\norm{ B[z_n(t)]v_n(t) }_{L^2} \lesssim \( \int \chi_{n,r} |v_n|^2 dx \)^\frac12 + o_n(1)
\]
as $n\to\I$, and hence
\[
\abs{\int \chi_{n,r} \Im ( \overline{v_n} B[z_n] v_n ) dx} \le  \int \chi_{n,r} |v_n|^2 dx  + o_n(1).
\]

We turn to the third term in \eqref{e:trial2.2}. We observe the following bounds:
\begin{align*}
&|\chi_{n,r}^{\frac12} P_c N(z_n,R[z_n] v_n)| \lesssim  |\chi_{n,r}^{\frac12} R[z_n]v_n| (|\Phi[z]|^2 + |R[z_n]v_n|^2), \\
&|\chi_{n,r}^{\frac12} R[z_n]v_n| \le |\chi_{n,r}^{\frac12}v_n| + |\phi_0| \biggl\{ \biggl[\int \chi_{n,r} |v_n|^2 dx \biggr]^\frac12 + o_n(1) \biggr\}.
\end{align*}
Thus 
\[
\norm{ \chi_{n,r}^{\frac12} P_c N(z_n,R[z_n] v_n) }_{L^2} \le C \(\int \chi_{n,r} |v_n|^2 dx \)^\frac12 + o_n(1).
\]
Inserting these estimates into \eqref{e:trial2.2}, we see that 
\[
f(t)=\int \chi_{n,r} |v_n(t)|^2\,dx\qtq{satisfies}
f(t) \le f(0) + \tfrac{\tau}{r}[C+o_n(1)]+ C \int_0^t f(s) \,ds 
\]
for some constant $C>0$ (independent of $t,r,n$).  Thus, by Gronwall's inequality, 
\[
\sup_{t\in [0,\tau]} f(t) \le e^{C\tau} \( \int \chi_{n,r} |\varphi|^2 dx  +\tfrac{\tau}{r}(C+o_n(1)\).
\]
In particular, given $\eps,\tau>0$, we may find $r>0$ sufficiently large and $N$ so that
\[
n\geq N \implies \sup_{t\in [0,\tau]} \int \chi_{n,r} |v_n(t)|^2 dx \le \eps.		
\]
As we can derive a similar estimate for negative times, we obtain \eqref{e:trial2.claim1}.

It remains to prove \eqref{e:trial2.claim2}. Defining 
\[
w_n(t) := T_{y_n}^{-1} v_n(t),
\]
we may write
\begin{align*}
-i T_{y_n}^{-1}& \int_0^t  e^{-i(t-s)\Delta} ( i \d_t v_n - \Delta v_n - \sigma |v_n|^2 v_n )\,ds \\
& = 	w_n(t) - e^{-it\Delta}\varphi +i  \int_0^t e^{i(t-s)\Delta} (\sigma |w_n|^2 w_n)\,ds= N_1 + N_2 + N_3,
\end{align*}
where
\begin{align*}
&N_1 =  (T_{y_n}^{-1} \mathcal{L}(t,0;z_n)P_c T_{y_n} - e^{-it\Delta}  ) \varphi,\\
&N_2 = - i  \int_0^t (T_{y_n}^{-1} \mathcal{L}(t,s;z_n)P_cT_{y_n}- e^{-i(t-s)\Delta}) (\sigma |w_n|^2 w_n)(s)\,ds,\\
&N_3 = - i  \int_0^t (T_{y_n}^{-1} \mathcal{L}(t,s;z_n)( \tilde{N}(z_n, T_{y_n} w_n) - \sigma P_c T_{y_n}(| w_n|^2 w_n))\,ds.
\end{align*}

By Lemma~\ref{l:spacetranslation}, we first observe that
\[
\lim_{n\to\infty}\norm{N_1}_{\mathrm{Stz}^1 ([0,\tau])}=0,
\]
giving an acceptable estimate for $N_1$.

For $N_2$, we will use the interpolation estimate
\[
\norm{N_2}_{\mathrm{Stz}^{\theta_1} ([0,\tau])} \lesssim \norm{N_2}_{\mathrm{Stz}^0 ([0,\tau])}^{1-\frac{\theta_1}{\theta_2}} \norm{N_2}_{\mathrm{Stz}^{\theta_2} ([0,\tau])}^{\frac{\theta_1}{\theta_2}}.
\]
Using the Strichartz estimates for $\mathcal{L}(t,0;z_n)$ and $e^{-it\Delta}$, we first observe that
\[
\norm{N_2}_{\mathrm{Stz}^{\theta_2} ([0,\tau])} \lesssim \norm{|w_n|^2w_n}_{L^2 H^{\theta_1}_{6/5}([0,\tau])} \lesssim \norm{w_n}_{L^4 L^6 ([0,\tau])}^2 \norm{w_n}_{L^\I H^{\theta_2} ([0,\tau])}\lesssim A^3.
\]
We next write
\begin{align*}
N_2 = i T_{y_n}^{-1} \int_0^t& \biggl[\mathcal{L}(t,s;z_n) (V  (T_{y_n} d_n) \\
& \quad  - B[z_n] (T_{y_n}d_n) +(1- P_c) T_{y_n}|w_n|^2 w_n ))(s)\biggr]\,ds,
\end{align*}
where 
\[
d_n(t) = -i \sigma \int_0^t e^{-i(t-s)\Delta} (|w_n|^2 w_n) (s)\,ds
\]
solves
\[
(i \d_t -\Delta ) d_n = \sigma |w_n|^2 w_n,\quad d_n(0)=0.
\]
We consider each term in $N_2$ separately.  First, by Strichartz, we have that
\[
\norm{d_n}_{\mathrm{Stz}^{\theta_2}} \lesssim \norm{|w_n|^2 w_n}_{L^2 H^{\theta_2}_{6/5}} \lesssim \norm{w_n}_{L^4 L^6}^2 \norm{w_n}_{L^\I H^{\theta_2}}
\lesssim A^3.
\]
Next, we observe that
\begin{align*}
\abs{\d_t \int \chi(\tfrac{\cdot}{R}) |d_n|^2 \,dx}  &{}\lesssim R^{-1}\norm{d_n}^2_{H^{1/2}} + \norm{d_n}_{L^6} \norm{ \chi(\tfrac{\cdot}{R}) w_n}_{L^2} \norm{w_n}_{L^6}^2 \\
&{}\lesssim R^{-1} + \norm{d_n}_{L^6} \norm{ \chi(\tfrac{\cdot}{R}) w_n}_{L^2} \norm{w_n}_{L^6}^2.
\end{align*}
In particular, using \eqref{e:trial2.claim1}, given $\eps>0$, we have that
\[
\varlimsup_{n\to\I} \sup_{t \in [0,\tau]}\int \chi(\tfrac{\cdot}{R}) |d_n|^2\,dx \lesssim\tau R^{-1} + \tau^\frac14 A^5 \varlimsup_{n\to\I }\norm{ \chi(\tfrac{\cdot}{R}) w_n}_{L^\I L^2([0,\tau])} \lesssim \eps
\]
for $R$ sufficiently large. From this estimate, we can deduce that
\[
\lim_{n\to\I} \norm{ V  T_{y_n} d_n }_{(L^1L^2 + L^2 L^{\frac65})([0,\tau])} =0,
\]
which suffices to treat the first term appearing in the formula for $N_2$ above.  Here the $L^1L^2$ and $L^2 L^{\frac65}$ norms are used to estimate the $L^2$ and $L^\I$ parts of $V$, respectively.

{For the second term appearing in the formula for $N_2$, we note that
\[
\lim_{n\to\I} \norm{ B[z_n] ( T_{y_n} d_n) }_{L^1L^2([0,\tau])} =0,
\]
which follows from the fact that $z_n$ takes values in a compact set of $\C$.}

Finally, for the third term in the formula for $N_2$, we note that by \eqref{e:trial2.claim1} implies
\[
\norm{ (1- P_c) T_{y_n}|w_n|^2 w_n }_{L^1L^2([0,\tau])}\le \tau \sup_{t\in [0,\tau]} |(T_{y_n}^{-1}\phi_0, |w_n|^2 w_n )| \to 0\qtq{as}n\to\infty.
\]

Combining the estimates above, we obtain
\[
\lim_{n\to\I} \norm{N_2}_{\mathrm{Stz}^0 ([0,\tau])} =0,\qtq{and hence}	\lim_{n\to\I} \norm{N_2}_{\mathrm{Stz}^{\theta_1} ([0,\tau])} =0.
\]

Finally let us estimate $N_3$.  We begin by writing
\begin{equation}\label{N3-stuff}
\begin{aligned}
\tilde{N} (z_n,T_{y_n} w_n)& - \sigma P_c  T_{y_n} (|w_n|^2 w_n) \\
&= \sigma P_c ( 2\Phi[z_n]|R[z_n]T_{y_n} w_n |^2 + \overline{\Phi[z_n]}(R[z_n]T_{y_n} w_n)^2 )\\
& \quad+ \sigma P_c ( |R[z_n]T_{y_n} w_n |^2R[z_n]T_{y_n} w_n- |T_{y_n} w_n |^2 T_{y_n} w_n  )\\
& \quad- i P_c D\Phi[z_n]  \underline{N} (z_n,R[z_n]T_{y_n} w_n))
\end{aligned}
\end{equation}
and observing that this is a sum of multilinear functions with respect to $w_n$.  By a density argument, we may approximate $w_n$ on $[0,\tau]\times\R^3$ by a smooth function in the $\mathrm{Stz}^{\theta_1}([0,\tau])$ norm.  In particular, the error is small in the $L^1 H^{\theta_1}$-norm, so thatit suffices to estimate the term above in $L^1 H^1$ assuming that  $w_n$ is smooth.

We next observe that
\[
|(R[z_n]-1) T_{y_n} w_n| \le |c_n||\phi_0| 
\]
with $|c_n|\lesssim | (\Phi[z_n] ,T_{y_n} w_n )| \to 0$ as $n\to\I$ uniformly in $t$.  Thus $R[z_n]$ may be replaced by the identity with an acceptable error.

The quadratic part (the first line on the right-hand side of \eqref{N3-stuff}) is small in $L^1 H^1$ since each term (along with their derivatives) contains a product $\Phi[z_n]\cdot T_{y_n}w_n$ or $\nabla\Phi[z_n]\cdot T_{y_n}w_n$, which tend to zero in $L_t^\infty(L^1\cap L^2)_x([0,\tau])$ as $n\to\infty$ (cf. \eqref{e:trial2.claim1}). 

The cubic part (the second line on the right-hand side of \eqref{N3-stuff}) is acceptable, as it is essentially $\mathcal{O}(|(R[z_n]-1) T_{y_n} w_n| |T_{y_n} w_n|^2)$.  The last line on the right-hand side of \eqref{N3-stuff} is handled similarly.

Thus we obtain that as $n\to\infty$, we have
\[
\norm{N_3}_{\mathrm{Stz}^{\theta_1} ([0,\tau])} \lesssim \norm{\tilde{N} (z_n,T_{y_n} w_n) - \sigma P_c  T_{y_n} (|w_n|^2 w_n)}_{L^1 H^{\theta_1} ([0,\tau])} \to 0.
\] 
As the same argument handles negative times, we complete the proof of \eqref{e:trial2.claim2}.\end{proof}

\begin{proof}[Proof of Proposition \ref{p:Lambda0kbound}]
We prove the bound for positive times.  We fix $\tau>0$ and let $\eps_0,\delta_0>0$ be constants to be chosen later.  To simplify notation, we omit superscripts and denote $v_n = \Lambda_n^{(0,k)}$ and $w_\I=w_\I^{(0,k)}$.  We also define
\[
w_n = (T_{y_n})^{-1} v_n,\qtq{where}y_n\qtq{denotes}y_n^{(0,k)}.
\]

As $M := \norm{w_\I}_{L^8L^4(\R)} <\I$, there exists $L_0\le 1+(M/\delta_0)^{1/8}$ and a partition $\{t_\ell\}_{\ell=0}^{L_0}$ of $[0,\tau]$ such that
\[
\sup_{ \ell \in [1, L_0] } \norm{ w_\I }_{L^8 L^4 ([t_{\ell-1}, t_\ell])} \le \delta_0.
\]

We now choose $\theta_\ell$ so that
\[
1 = \theta_1 > \theta_2 > \dots > \theta_{L_0} \ge \tfrac34.
\]
We will show by induction on $k\in [0,L_0]$ that there exists $N_k\in \N$, increasing $C_k>0$ (increasing in $k$), and $\eps_k >0$ (decreasing in $k$) such that if
\[
\norm{ w_n(0) - w_\I(0) }_{H^1} \le \eps_k,
\]
then
\begin{equation}\label{e:Lambda0kboundpf}
\norm{ w_n - w_\I }_{ \mathrm{Stz}^{\theta_k}([0,t_{k}]) } \le \min\{C_k\norm{ w_n(0) - w_\I(0) }_{H^1}, \eps_0\}\qtq{for}n\ge N_k.
\end{equation}

The base case $k=0$ is straightforward for any choice of $\eps_0>0$. Indeed, since 
\[
w_n(0)-w_\I(0) = (P_c-1)T_{y_n}  \varphi^{(0,k) } = \phi_0 ( T_{y_n} \varphi^{(0,k)} ,\phi_0)\to 0
\]
in $H^1$ as $n\to\I$, the assumption is satisfied if $n$ is sufficiently large.

Now suppose that we have \eqref{e:Lambda0kboundpf} at some $k_0 \in [0, L_0-1]$.  Then, provided $\delta_0>0$ is sufficiently small, we have that
\[
\norm{ e^{-i(t-t_{k_0}) \Delta} w(t_{k_0}) }_{L^8L^4 ([t_{k_0},t_{k_0+1}])} \le 2\norm{w}_{L^8L^4 ([t_{k_0},t_{k_0+1}])} \le 2 \delta_0.
\]
By Lemma~\ref{l:spacetranslation}, this implies
\[
\norm{ \mathcal{L} (\cdot, t_{k_0};z_n) P_c T_{y_n} w_\I(t_{k_0}) }_{L^8L^4 ([t_{k_0},t_{k_0+1}])} \le 3\delta_0
\]
for large $n$. Consequently, by Strichartz and the inductive hypothesis, we have 
\begin{align*}
\| \mathcal{L} (\cdot,& t_{k_0};z_n) {w_n} (t_{k_0}) \|_{L^8L^4 ([t_{k_0},t_{k_0+1}])} \\
&\le \norm{ \mathcal{L} (\cdot, t_{k_0};z_n) P_c T_{y_n} w_\I(t_{k_0}) }_{L^8L^4 ([t_{k_0},t_{k_0+1}])} \\
&\quad + \norm{ \mathcal{L} (\cdot, t_{k_0};z_n) P_c T_{y_n} (w_\I - w_n)(t_{k_0}) }_{L^8L^4 ([t_{k_0},t_{k_0+1}])} \\
&\lesssim \delta_0 + \norm{(w_\I - w_n)(t_{k_0})}_{H^{\frac12}} \lesssim \delta_0+ \eps_0
\end{align*}
for large $n$. We now fix $\delta_0$ and $\eps_0$ so small that Lemma~\ref{l:sdt} implies
\begin{align*}
\norm{v_n}_{L^8L^4([t_{k_0},t_{k_0+1}])} &\le 2 \norm{ \mathcal{L} (\cdot, t_{k_0};z_n) v_n (t_{k_0}) }_{L^8L^4 ([t_{k_0},t_{k_0+1}])}\lesssim \delta_0 + \eps_0, \\
\norm{v_n}_{\mathrm{Stz}^{\theta_{k_0}}([t_{k_0},t_{k_0+1}])}&{}\le C( \norm{v_n(t_{k_0})}_{H^{\theta_{k_0}} } + \sup_{n} \norm{z_n}_{L^\I}) \\
&\le C ( \norm{w_\I (t_{k_0})}_{H^{\theta_{k_0}} }+ \delta_0  + \sup_{n} \norm{z_n}_{L^\I}) \\
&\le 2 C (\norm{w_\I }_{L^\I H^{1} (\R)}  + \sup_{n} \norm{z_n}_{L^\I}) =:A.
\end{align*}
Furthermore, since $v_n(0) \in H^1$, we have $v_n(t_{k_0}) \in H^1$.  Then, by Lemma~\ref{L:818} (i.e. \eqref{e:trial2.claim1} and \eqref{e:trial2.claim2}), we have
\begin{align*}
\lim_{R\to \I} \varlimsup_{n\to\I} \norm{v_n}_{L^\I L^2 ([t_{k_0},t_{k_0+1}] \times \{ |x-y_n| \ge R \} )} &=0,\\
\lim_{n\to\I} \norm{ \int_0^t e^{i(t-s)\Delta} (i \d_t v_n - \Delta v_n - \sigma |v_n|^2 v_n)\,ds  }_{\mathrm{Stz}^{\theta_{k_0+1}} ([t_{k_0},t_{k_0+1}])}& = 0.
\end{align*}
Then, together with 
\begin{align*}
&\norm{(w_\I - w_n)(t_{k_0})}_{H^{\theta_{k_0+1}}}\le  \norm{(w_\I - w_n)(t_{k_0})}_{H^{\theta_{k_0} } } \le C_{k_0} \norm{ w_n(0) - w_\I(0) }_{H^1},\\
&\norm{ w_\I }_{L^8L^4 ([t_{k_0},t_{k_0+1}])} \le \delta_0,
\end{align*}
we see from perturbation around $w_\I$ that
\[
\norm{ w_\I - w_n }_{\mathrm{Stz}^{\theta_{k_0+1} } ([t_{k_0},t_{k_0+1}]) }\lesssim \norm{ w_n(0) - w_\I(0) }_{H^1}.
\]
Hence, there exists $C_{k_0+1} \ge C_{k_0}$ such that
\[
\norm{ w_\I - w_n }_{\mathrm{Stz}^{\theta_{k_0+1} } ([0,t_{k_0+1}]) }\le C_{k_0+1} \norm{ w_n(0) - w_\I(0) }_{H^1}.
\]
Setting $\eps_{k_0+1} := \eps_0/C_{k_0+1}$, we obtain \eqref{e:Lambda0kboundpf} for large $n$ under the assumption that $\norm{ w_n(0) - w_\I(0) }_{H^1} \le \eps_{k_0+1}$.

By induction, we have that \eqref{e:Lambda0kboundpf} holds for $k=L_0$. Thus
\begin{equation}\label{e:Lambda0kboundpf2}
\lim_{n\to\I}  \norm{ w_n - w_\I }_{\mathrm{Stz}^{3/4} ([0,\tau])} =0
\end{equation}
for any $\tau>0$, which yields \eqref{e:Lambda0kbound4}.  Furthermore, by \eqref{e:trial2.claim1},
\begin{equation}\label{e:Lambda0kboundpf3}
\lim_{R\to\I} \varlimsup_{n\to\I} \norm{ w_n }_{L^\I_t L^2 ([0,\tau] \times \{|x|\ge R \}  )}=0
\end{equation}
for any $\tau>0$. Interpolating \eqref{e:Lambda0kboundpf2} with \eqref{e:Lambda0kboundpf3}, we obtain
\begin{equation}\label{e:Lambda0kboundpf4}
\lim_{R\to\I} \varlimsup_{n\to\I} \norm{ w_n }_{L^8_t L^4 ([-\tau,\tau] \times \{|x|\ge R \}  )}=0.
\end{equation}

To complete the proof, it suffices to show \eqref{e:Lambda0kbound3}. Indeed, \eqref{e:Lambda0kbound2} follows from \eqref{e:Lambda0kbound3} and \eqref{e:Lambda0kboundpf4}.  Similarly, \eqref{e:Lambda0kbound1} follows from \eqref{e:Lambda0kbound3}, \eqref{e:Lambda0kboundpf2}, and the fact that $w_\I \in L^8L^4(\R)$.

We turn to \eqref{e:Lambda0kbound3} and let $\eps>0$.  Using \eqref{e:spacelargetimebound12}, we find that there exists $\tau_0=\tau_0(\eps)>0$ such that if $\tau \ge \tau_0$, then
\begin{align*}
\norm{ \mathcal{L} (\cdot, \tau; z_n) w_\I (\tau)}_{L^8 L^4 ([\tau,\I))}&=\norm{ \mathcal{L} (\cdot, 0; z_n(\cdot + \tau)) w_\I (\tau)}_{L^8 L^4 ([0,\I))} \\
&\lesssim \norm{e^{-it \Delta} w_\I(\tau) }_{L^8 L^4 ([0,\I))} \\
&\lesssim \norm{e^{-i(t-\tau) \Delta} w_\I(\tau) }_{L^8 L^4 ([\tau,\I))} \\
&\lesssim \norm{ w_\I }_{L^8 L^4 ([\tau,\I))}\lesssim \eps.
\end{align*}
On the other hand, using \eqref{e:Lambda0kboundpf2}, we have that for each $\tau\ge \tau_0$, there exists $N=N(\tau)$ such that if $n \ge N$, then
\[
	\norm{ \mathcal{L} (\cdot, \tau; z_n) (w_n-w_\I (\tau))}_{L^8 L^4 ([\tau,\I))}
	\lesssim  \norm{ w_n - w_\I }_{\mathrm{Stz}^{3/4} ([0,\tau])} \le \eps.
\]
For such $\tau$ and $n$, we therefore have that
\[
\norm{ \mathcal{L} (\cdot, \tau; z_n) w_n (\tau)}_{L^8 L^4 ([\tau,\I))} \lesssim \eps.
\]
If $\eps$ is small then by stability (Lemma \ref{l:6.3}), we have that $w_n$ is forward global and satisfies
\[
\norm{ w_n }_{L^8 L^4 ([\tau,\I))} \le 2 \norm{ \mathcal{L} (\cdot, \tau; z_n) w_n (\tau)}_{L^8 L^4 ([\tau,\I))} \lesssim \eps.
\]
As the same argument applies for negative times, we derive \eqref{e:Lambda0kbound3}.
\end{proof}

We next bound the \emph{sum} of nonlinear profiles of the form $\Lambda_n^{(0,k)}$:

\begin{lemma}\label{l:Lambda0ksum} We have the following bounds:
\begin{align}
&\sup_{K\ge1} \varlimsup_{n\to\I} \norm{ \sum_{k=1}^K \Lambda_n^{(0,k)}}_{L^8L^4(\R)} <\I,\label{e:Lambda0ksum1}\\
&\lim_{\tau\to\I }\sup_{K\ge1} \varlimsup_{n\to\I} \norm{ \sum_{k=1}^K \Lambda_n^{(0,k)}}_{L^8L^4((-\I,-\tau]\cup[\tau,\I))} =0.\label{e:Lambda0ksum4}
\end{align}
\end{lemma}

\begin{proof} We begin with \eqref{e:Lambda0ksum1}.  If $k_1 \neq k_2$ then we see from \eqref{e:Lambda0kbound1}, \eqref{e:Lambda0kbound2}, and orthogonality of the profiles that
\[
\lim_{n\to\I} \tnorm{ \Lambda_n^{(0,k_1)} \Lambda_n^{(0,k_2)}  }_{L^4_tL^2_x (\R)} =0.
\]
Thus
\begin{align*}
\varlimsup_{n\to\I} \norm{ \sum_{k=1}^K \Lambda_n^{(0,k)}}_{L^8L^4(\R)}^4 
&=\varlimsup_{n\to\I} \norm{ \sum_{ k_1,k_2,k_3,k_4 \in [1,K] }\Lambda_n^{(0,k_1)}\Lambda_n^{(0,k_2)}\Lambda_n^{(0,k_3)}\Lambda_n^{(0,k_4)} }_{L^2_tL^1_x (\R)} \\
&{}\le \varlimsup_{n\to\I} \sum_{k=1}^K  \tnorm{ \Lambda_n^{(0,k)}}_{L^8L^4(\R)}^4= \sum_{k=1}^K  \tnorm{ w_\I^k}_{L^8L^4(\R)}^4.
\end{align*}
As $\norm{\varphi_\I^k}_{H^1} \to 0$ as $k\to\I$, the small-data theory yields sufficiently large $K_*$ that 
\[
k\geq K_*\implies \tnorm{ w_\I^k}_{L^8L^4(\R)} \le 2 \tnorm{\varphi^{(0,k)}}_{H^1}.
\]
It follows that
\[
\sup_{K \ge 1}  \sum_{k=1}^K  \tnorm{ w_\I^k}_{L^8L^4(\R)}^4\le \sum_{k=1}^{K_*}  \tnorm{ w_\I^k}_{L^8L^4(\R)}^4+ 2^4 \sum_{k=K_*+1}^\I \tnorm{\varphi^{(0,k)}}_{H^1}^4<\I,
\]
which gives \eqref{e:Lambda0ksum1}.

We turn to \eqref{e:Lambda0ksum4}.  We let $\eps>0$ and argue as above to deduce the existence of $K_0\ge K_*$ such that
\begin{align*}
\sup_{K\ge K_0} & \varlimsup_{n\to\I} \norm{ \sum_{k=K_0}^K \Lambda_n^{(0,k)}}_{L^8L^4(\R)}^4\\
&\le  \sup_{K\ge K_0} \sum_{k=K_0}^K  \tnorm{ w_\I^k}_{L^8L^4(\R)}^4\le 2^4 \sum_{k=K_0}^\I \tnorm{\varphi^{(0,k)}}_{H^1}^4\le \eps.
\end{align*}
For this $K_0$, we see from \eqref{e:Lambda0kbound3} that there exists $\tau_0$ such that $\tau\ge\tau_0$ implies
\begin{align*}
 \varlimsup_{n\to\I}& \norm{ \sum_{k=1}^{K_0} \Lambda_n^{(0,k)}}_{L^8L^4((-\I,-\tau]\cup[\tau,\I)))}\\
&  \le K_0 \sup_{k \in [1,K_0]} \varlimsup_{n\to\I} \norm{ \Lambda_n^{(0,k)}}_{L^8L^4((-\I,-\tau]\cup[\tau,\I)))}\le \eps.
\end{align*}
Combining these two estimates, we obtain the result.
\end{proof}

\subsubsection{Estimates on $\Lambda_n^{(s,k)}$} We next consider nonlinear profiles $\Lambda_n^{(s,k)}$ with $s\in[1,s_{\max}]$.  We let $w_\I^{(s,k)}$ be the solution to the following cubic NLS: 

\begin{equation}\label{e:fnls_np2}
\begin{cases}
& i \d_t w_\I^{(s,k)} - \Delta w_\I^{(s,k)} = \sigma \bigl|w_\I^{(s,k)}\bigr|^2 w_\I^{(s,k)}, \\
& w_\I^{(s,k)}(-\tau)= e^{i\tau \Delta }{\varphi^{(s,k)}},
\end{cases}
\end{equation}
where we denote the profile $\varphi_\infty^j$ appearing in Definition~\ref{d:nprofile3} by $\varphi^{(s,k)}.$  We first collect some properties of $w_\I^{(s,k)}$.

\begin{lemma}\label{l:wskest} Suppose that \eqref{e:fakeassumption} holds.  Let $s \in [1,s_{\max}]$ and $k \ge 1$.  If $\varphi^{(s,k)} \neq0$ then  $\mathbb{M}(\varphi^{(s,k)}) \le M_*$ and
\begin{align*}
 \mathbb{H}_0 (\varphi^{(s,k)}) \le \min \bigl\{ E_* - \mathscr{E}_0(M_*), \tfrac1{\mathbb{M}(\varphi^{(s,k)})} \mathbb{M}(Q)  \mathbb{E}_0 (Q) - \delta_0 \bigr\}.
\end{align*}
In particular, for any $\tau >0$ there exists a unique global solution $w_\I^{(s,k)}\in C(\R,H^1)$ to \eqref{e:fnls_np2} such that
\begin{align*}
&\sup_{\tau>0}\norm{w_\I^{(s,k)}}_{L^8_tL^4_x (\R)} \lesssim_{\delta_0} 1, \\ 
& \sup_{\tau>0}\norm{w_\I^{(s,k)}}_{L^\I_t H^1 (\R)}^2 \lesssim  \mathbb{H}_0 (\varphi^{(s,k)}) +  \mathbb{M}(\varphi^{(s,k)}).
\end{align*}
Furthermore, writing $w_{\infty,-}^{(s,k)}$ for the unique global solution to
\begin{equation}\label{e:fnls_np3}
\begin{cases}
i \d_t w_{\I,-}^{(s,k)} - \Delta w_{\I,-}^{(s,k)} = \sigma \bigl|w_{\I,-}^{(s,k)}\bigr|^2 w_{\I,-}^{(s,k)}, \\
 \displaystyle\lim_{t\to-\I} e^{it\Delta }w_{\I,-}^{(s,k)}(t)= \varphi^{(s,k)},\end{cases}
\end{equation}
we have
\begin{equation}\label{e:wIskconv}
\lim_{\tau\to\I} \norm{ w_\I^{(s,k)} - w_{\I,-}^{(s,k)}}_{\mathrm{Stz}^1(\R)}=0.
\end{equation}
\end{lemma}

\begin{proof}
By \eqref{e:energydecomp} and \eqref{e:wkenergy}, we have
\[
E_* \ge \mathbb{E}_V (u_\I^0) + \mathbb{H}_0 (\varphi^{(s,k)}).
\]
In particular, $\mathbb{H}_0 (\varphi^{(s,k)}) \le E_* - \mathscr{E}_0(M_*)$.

Next, by \eqref{e:Mdecomp} and \eqref{e:spacenormflat1}, one obtains
\[
M_* \ge \mathbb{M} (u_\I^0) + \mathbb{M} (\varphi^{(s,k)})
\]
Thus, mimicking the proof of \eqref{e:Ephi12}, we obtain
\[
\mathbb{H}_0 (\varphi^{(s,k)}) \le  \frac1{\mathbb{M}(\varphi^{(s,k)})} \mathbb{M}(Q)  \mathbb{E}_0 (Q) - \delta_0 .
\]

Now, for any $\tau >0$, we have
\begin{align*}
&\mathbb{M}( e^{i\tau \Delta } \varphi^{(s,k)}) \mathbb{E}_0 (e^{i\tau \Delta }\varphi^{(s,k)})< \mathbb{M}( e^{i\tau \Delta } \varphi^{(s,k)}) \mathbb{H}_0 (e^{i\tau \Delta }\varphi^{(s,k)})< \mathbb{M}(Q)  \mathbb{E}_0 (Q),\\
& \mathbb{M}( e^{i\tau \Delta } \varphi^{(s,k)}) \mathbb{H}_0 (e^{i\tau \Delta }\varphi^{(s,k)})< \mathbb{M}(Q)  \mathbb{E}_0 (Q) < \mathbb{M}(Q)  \mathbb{H}_0 (Q).
\end{align*}
Thus, one sees from the scattering theory for the cubic NLS (cf. \cite{DHR})  theory that $w_\I^{(s,k)}$ exists globally in time and satisfies the bounds
\begin{align*}
&\norm{w_\I^{(s,k)}}_{L^8_tL^4_x (\R)} \lesssim_{\delta_0} 1, \\
&\norm{w_\I^{(s,k)}}_{L^\I_t H^1 (\R)}^2 \lesssim  \mathbb{E}_0 (e^{i\tau \Delta }\varphi^{(s,k)}) +  \mathbb{M}(\varphi^{(s,k)}).
\end{align*}

Let $w_{\I,-}^{(s,k)}$ be the solution to \eqref{e:fnls_np3}. Since this solution scatters to $\varphi^{(s,k)}$ backward in time, we have
\begin{align*}
&\mathbb{M}(w_{\I,-}^{(s,k)}) =\mathbb{M}( \varphi^{(s,k)}), \\
&\mathbb{E}_0 (w_{\I,-}^{(s,k)}) = \mathbb{H}_0 (\varphi^{(s,k)}),\\
& \lim_{t\to-\I} \mathbb{K}_{0,2} (w_{\I,-}^{(s,k)}) = 2\mathbb{H}_0 (\varphi^{(s,k)})>0.
\end{align*}
These imply that 
\[
\mathbb{M}(w_{\I,-}^{(s,k)}) \mathbb{E}_0 (w_{\I,-}^{(s,k)}) < \mathbb{M}(Q)  \mathbb{H}_0 (Q),
\]
which guarantees scattering forward in time as well, with the bounds
\begin{align*}
&\norm{w_{\I,-}^{(s,k)}}_{\mathrm{Stz}^1 (\R)} \lesssim 1, \\
&\norm{w_{\I,-}^{(s,k)}}_{L^\I_t H^1 (\R)}^2 \lesssim  \mathbb{E}_0 (\varphi^{(s,k)}) +  \mathbb{M}(\varphi^{(s,k)}).	
\end{align*}
The convergence $w_\I^{(s,k)} \to w_{\I,-}^{(s,k)}$ then follows from a standard stability argument.  In particular, the convergence implies that the bound on $w_\I^{(s,k)}$ is uniform in $\tau>0$.
\end{proof}

We can then prove the following properties for the nonlinear profiles $\Lambda_n^{(s,k)}$. As above, we write $y_n$ for the sequence of translations $y_n^{(s,k)}$. 

\begin{proposition}\label{p:Lambdaskbound}
Suppose that \eqref{e:fakeassumption} holds.  Let $k\ge1$, and define $\Lambda_n^{(s,k)}$ be as in Definition \ref{d:nprofile3}. Let $w_\I^{(s,k)} \in C(\R,H^1)$ be the unique global solution to \eqref{e:fnls_np2} given in Lemma~\ref{l:wskest}. Then, for any $\tau>0$, we have
\begin{align}
&\lim_{n\to\I} \norm{\Lambda_n^{(s,k)}(\cdot + s_n^s) - T_{y_n} w_\I^{(s,k)}}_{L^8L^4(\R)}=0, \label{e:Lambdaskbound1} \\
&\lim_{R\to\I} \varlimsup_{n\to\I} \norm{ {\bf 1}_{ \{|x-y_n|\ge R \}} \Lambda_n^{(s,k)} }_{L^8_t L^4 (\R)}=0. \label{e:Lambdaskbound2}
\end{align}
In particular,
\begin{equation}\label{e:Lambdaskbound2.5}
\sup_{\tau>0} \varlimsup_{n\to\I} \norm{\Lambda_n^{(s,k)} }_{L^8L^4(\R)} < \I.
\end{equation}
Furthermore, 
\begin{equation}\label{e:Lambdaskbound3}
\begin{aligned}
\lim_{\tau \to\I}& \varlimsup_{n\to\I}\biggl\{ \norm{ \mathcal{L}(\cdot,-\tau+s_n^s;z_n)\Lambda_n^{(s,k)}(-\tau+s_n^s) }_{L^8_tL^4_x ((-\I,-\tau+s_n^s))} \\
&\quad\quad \quad +  \norm{ \mathcal{L}(\cdot,\tau+s_n^s;z_n)\Lambda_n^{(s,k)}(\tau+s_n^s) }_{L^8_tL^4_x ((\tau,\I))} \\
& \quad\quad \quad + \norm{\Lambda_n^{(s,k)}}_{L^8_tL^4_x ((-\I,-\tau+s_n^s)\cup(\tau+s_n^s,\I))}\biggr\}=0.
\end{aligned}
\end{equation}
\end{proposition}

\begin{proof} Using the fact that $\tnorm{w_{\I}^{(s,k)}}_{L^8L^4(\R)} <\I$, the proof is essentially the same as that of \eqref{p:Lambda0kbound}. Next, define
\[
w_n^{(s,k)}(t)= T_{y_n^{(s,k)}}^{-1} \Lambda_n^{(s,k)}(t + s_n^s).
\]
Then, mimicking the proof of \eqref{e:Lambda0kboundpf2} and \eqref{e:Lambda0kboundpf4}, we may obtain
\begin{equation}\label{e:Lambdaskboundpf0}
\lim_{n\to\I}\biggl\{ \norm{w_n^{(s,k)} - w_\I^{(s,k)}}_{\mathrm{Stz}^{3/4} ([-\tau',-\tau])} + \norm{w_n^{(s,k)} - w_\I^{(s,k)}}_{\mathrm{Stz}^{3/4} ([-\tau,\tau'])}\biggr\} =0
\end{equation}
and
\[
\lim_{R\to\I} \varlimsup_{n\to\I} \norm{w_n^{(s,k)}}_{L^8L^4 ([-\tau',\tau'] \times \{|x| \ge R\})}=0
\]
for any $\tau' \ge \tau>0$.  Then, \eqref{e:Lambdaskbound1} and \eqref{e:Lambdaskbound2} follow from 
\begin{multline}\label{e:Lambdaskboundpf1}
	\lim_{\tau' \to\I} \varlimsup_{n\to\I} \biggl\{\norm{ \mathcal{L}(\cdot,-\tau';z_n(\cdot+s_n^s))\Lambda_n^{(s,k)}(-\tau') }_{L^8_tL^4_x ((-\I,-\tau'))} 
	\\
	+ \norm{ \mathcal{L}(\cdot,\tau';z_n(\cdot+s_n^s))\Lambda_n^{(0,k)}(\tau') }_{L^8_tL^4_x ((\tau',\I))}\biggr\}=0,
\end{multline}
which we deduce from $\tnorm{w_{\I}^{(s,k)}}_{L^8L^4(\R)} <\I$ as in the proof of \eqref{e:Lambda0kbound3}. Note that $\tau'$ is not the same parameter as $\tau$ in \eqref{e:Lambdaskboundpf1} (i.e. $\tau$ is fixed in \eqref{e:Lambdaskboundpf1}).  We then obtain \eqref{e:Lambdaskbound2.5} from \eqref{e:Lambdaskbound1} and the bound
\[
\sup_{\tau>0} \tnorm{w_\I^{(s,k)}} _{L^8L^4 (\R)}<\I.
\]

We turn to \eqref{e:Lambdaskbound3}. By the small data theory, it suffices to establish the following:
\begin{align}
&\lim_{\tau \to\I} \varlimsup_{n\to\I}  \norm{ \mathcal{L}(\cdot,-\tau;z_n(\cdot+s_n^s))\Lambda_n^{(s,k)}(-\tau) }_{L^8_tL^4_x ((-\I,-\tau))} =0, \label{e:Lambdaskboundpf2}\\
&\lim_{\tau \to\I} \varlimsup_{n\to\I}  \norm{ \mathcal{L}(\cdot,\tau;z_n(\cdot+s_n^s))\Lambda_n^{(0,k)}(\tau) }_{L^8_tL^4_x ((\tau,\I))}=0.\label{e:Lambdaskboundpf3}
\end{align}

Note that by definition, we have
\[
\mathcal{L}(t-\tau,-\tau;z_n(\cdot+s_n^s))\Lambda_n^{(s,k)}(-\tau) =  \mathcal{L} (t; 0 ;z_n(\cdot-\tau+s_n^s)) P_c T_{y_n} e^{i\tau \Delta}\varphi^{(s,k)}.
\]
Thus, one sees from \eqref{e:spacelargetimebound12} that
\begin{align*}
&\norm{\mathcal{L}(t-\tau,-\tau;z_n(\cdot+s_n^s))\Lambda_n^{(s,k)}(-\tau) }_{L^8L^4 \cap L^2 H^{1/2}_6 ((-\I,0])} \\
&{} \lesssim \norm{ e^{-i \cdot \Delta} (e^{i\tau \Delta}\varphi^{(s,k)}) }_{L^8L^4 \cap L^2 H^{1/2}_6 ((-\I,0])} \\
&{} = \norm{ e^{-i \cdot \Delta}\varphi^{(s,k)} }_{L^8L^4 \cap L^2 H^{1/2}_6 ((-\I,-\tau])}.
\end{align*}
As this term tends to zero as $\tau\to\I$ (and is independent of $n$), we obtain \eqref{e:Lambdaskboundpf2}.

For \eqref{e:Lambdaskboundpf3}, we pick $\eps>0$ and choose $\tau_0$ such that if $\tau'> \tau_0$, then
\[
\norm{ w_{\I,-}^{(s,k)}}_{L^8L^4\cap L^2 H^{1/2}_6 ([\tau,\I))}\le \eps.
\]
Then we see from \eqref{e:wIskconv} that there exists $\tau_1\ge \tau_0$ such that if $\tau>\tau_1$, then
\begin{align*}
&\norm{ w_{\I}^{(s,k)}}_{L^8L^4\cap L^2 H^{1/2}_6 ([\tau,\I))} \\
&\le \norm{ w_{\I,-}^{(s,k)}}_{L^8L^4\cap L^2 H^{1/2}_6 ([\tau,\I))}+ \norm{ w_{\I}^{(s,k)}- w_{\I,-}^{(s,k)} }_{\mathrm{Stz}^1 (\R)}\le 2 \eps.
\end{align*}
Thus, if $\eps$ is sufficiently small, the small-data theory implies that
\[
\norm{  e^{-i (\cdot -\tau)\Delta} w_{\I}^{(s,k)}(\tau)}_{L^8L^4\cap L^2 H^{1/2}_6 ([\tau,\I))} \le 4 \eps.
\]
Using \eqref{e:spacelargetimebound12}, we obtain
\[
\norm{ \mathcal{L}(\cdot,0;z_n(\cdot+\tau+s_n^s)) w_{\I}^{(s,k)}(\tau)}_{L^8L^4\cap L^2 H^{1/2}_6 ([0,\I))}\lesssim \eps,
\]
which implies
\begin{equation}\label{e:Lambdaskboundpf4}
\norm{ \mathcal{L}(\cdot,\tau;z_n(\cdot+s_n^s)) w_{\I}^{(s,k)}(\tau)}_{L^8L^4 ([\tau,\I))}\lesssim \eps.
\end{equation}

On the other hand, by the uniform Strichartz estimate (Proposition~\ref{Prop:Strichartz}),
\[
\norm{ \mathcal{L}(\cdot,\tau;z_n(\cdot+s_n^s)) (w_{\I}^{(s,k)} -w_n^{(s,k)})(\tau) }_{L^8L^4 ([\tau,\I))}\lesssim  \norm{ (w_{\I}^{(s,k)} -w_n^{(s,k)})(\tau)  }_{H^{\frac12}}.
\]
Thus, in light of \eqref{e:Lambdaskboundpf0}, there exists $N(\tau)$ such that if $n \ge N$, we have
\begin{equation}\label{e:Lambdaskboundpf5}
\norm{ \mathcal{L}(\cdot,\tau;z_n(\cdot+s_n^s)) (w_{\I}^{(s,k)} -w_n^{(s,k)})(\tau) }_{L^8L^4 ([\tau,\I))}\le \eps.
\end{equation}
Combining \eqref{e:Lambdaskboundpf4} and \eqref{e:Lambdaskboundpf5}, we see that
\[
\varlimsup_{n\to\I} \norm{ \mathcal{L}(\cdot,\tau;z_n(\cdot+s_n^s)) w_{n}^{(s,k)}(\tau)}_{L^8L^4 ([\tau,\I))}\lesssim \eps
\]
for any $\tau \ge \tau_1$, which yields \eqref{e:Lambdaskboundpf3}.\end{proof}

Once again, we can derive bounds on the \emph{sum} of nonlinear profiles of the form $\Lambda_n^{(s,k)}$.  As the proof is essentially the same as that of Lemma~\ref{l:Lambda0ksum}, we omit it. 

\begin{lemma}\label{l:Lambdasksum} We have the following bounds:
\begin{align}\label{e:Lambdasksum1}
&\sup_{\tau>0} \sup_{K\ge1} \varlimsup_{n\to\I} \norm{ \sum_{k=1}^K \Lambda_n^{(s,k)}}_{L^8L^4(\R)} <\I,\\
\label{e:Lambdasksum4}
&\lim_{\tau\to\I }\sup_{K\ge1} \varlimsup_{n\to\I} \norm{ \sum_{k=1}^K \Lambda_n^{(s,k)}}_{L^8L^4((-\I,-\tau+s_n^s]\cup[\tau+s_n^s,\I)))}=0.
\end{align}
\end{lemma}

\subsection{Completion of the proof}

Finally, we turn to the proof of Proposition~\ref{p:keykey}, which we will then use to complete the proof of Proposition~\ref{p:key}.

\begin{proof}[Proof of Proposition~\ref{p:keykey}] We consider the positive time direction only. We suppose towards a contradiction that 
\[
(\mathbb{M}(u_\I^s ), \mathbb{E}_V (u_\I^s) ) \neq (M_*,E_*)
\]
for all $s\in[0,s_{\max}]$.  Ultimately, we will prove that
\begin{equation}\label{ultimately}
\|\mathcal{L}(\cdot,\tau+s_n^{s_{\max}};z_n)[(\xi_n-\Gamma_n^J)(\tau+s_n^{s_{\max}})]\|_{L^8L^4([\tau+s_n^{s_{\max}},\infty))}\to 0
\end{equation}
as $n,J,\tau\to\infty$, from which one can deduce scattering of $\xi_n$ as $t\to\infty$ (and hence reach a contradiction).  The proof will rely on an induction argument.  In particular, for each $s_0\in[0,s_{\max}]$, we define approximations
\begin{align}
\tilde \xi_n^{s_0,J} &:= \sum_{k=0}^{K(s_0,J)}\Lambda_n^{(s_0,k)}+\sum_{s=s_0+1}^{s_{\max}}\sum_{k=0}^{K(s,J)}\lambda_n^{(s,k)}+\Gamma_n^J, \label{def-tilde}\\
\hat\xi_n^{s_0,J} &:= \sum_{s=s_0}^{s_{\max}} \sum_{k=0}^{K(s,J)}\lambda_n^{(s,k)}+\Gamma_n^J,\label{def-hat}
\end{align}
(see \eqref{KsJ} for the definition of $K(s,J$)) and introduce the notation
\begin{equation}\label{tau-notation}
\tilde\tau_n^{s_0}=\tau+s_n^{s_0} \qtq{and} \hat\tau_n^{s_0}=-\tau+s_n^{s_0}.
\end{equation}
Note that for fixed $\tau>0$, we have $\hat\tau_n^{s_0}<\tilde\tau_n^{s_0}<\hat\tau_n^{s_0+1}$ for all $n$ large.  We will prove by induction that for each $s_0\in[0,s_{\max}]$, we have
\begin{align}\label{k-induction-dominant}
&\sup_{\tau>0} \lim_{J\to\I} \varlimsup_{n\to\I} \norm{\mathcal{L}(\cdot,\tilde\tau_n^{s_0} ;z_n)[(\xi_n - \tilde{\xi}_n^{s_0,J} )(\tilde\tau_n^{s_0})]}_{ L^8L^4 ([\tilde\tau_n^{s_0},\I]) } =0,\\
\label{k-induction-perturbative}
&\lim_{\tau\to\I} \varlimsup_{J\to\I} \varlimsup_{n\to\I} 
\norm{\mathcal{L}(\cdot,\hat\tau_n^{s_0};z_n) [( \xi_n  - \hat{\xi}_n^{s_0,J})(\hat\tau_n^{s_0})]  }_{ L^8L^4 ([\hat\tau_n^{s_0},\I)) } =0.
\end{align}
More precisely, we will prove:
\begin{itemize}
\item[(i)] \eqref{k-induction-dominant} holds for $s_0=0$; 
\item[(ii)] \eqref{k-induction-dominant} at level $s_0<s_{\max}$ implies \eqref{k-induction-perturbative} at level $s_0+1$; 
\item[(iii)] \eqref{k-induction-perturbative} at level $s_0\in[1,s_{\max}]$ implies \eqref{k-induction-dominant} at level $s_0$; and  
\item[(iv)] \eqref{k-induction-dominant} at level $s_{\max}$ implies \eqref{ultimately}.
\end{itemize}
Recall that once we have \eqref{ultimately}, we complete the proof of Proposition~\ref{p:keykey}. In particular, if $s_{\max}=0$, then it is only necessary to prove (i) and (iv).  Thus, in what follows we assume that $s_{\max}\geq 1$.
 
 \smallskip
 
\underline{(i) Estimate \eqref{k-induction-dominant} for $s_0=0$.} Define $\tilde\xi_n^{0,J}$ as in \eqref{def-tilde} (with $s_0=0$), as well as the corresponding error 
\begin{align*}
\mathcal{E}_{0,n}&{}:=(i\d_t + H - B[z_n])\tilde{\xi}_n^{0,J}- \tilde{N} (z_n , \tilde{\xi}_n^{0,J}) \\
&{} =  -(B[z_n]-B[z_\I^0])\xi_\I^0  + [\tilde{N} (z_\I^0, \xi_\I^0) - \tilde{N} (z_n, \xi_\I^0) ] \\
&{}\quad - \bigl[\tilde{N} (z_n , \tilde{\xi}_n^{0,J})  
-\sum_{k=0}^{K(0,J)} \tilde{N} (z_n,\Lambda_n^{(0,k)}) - \tilde{N}(z_n,\Gamma_n^J)  \bigr].
\end{align*}
We claim (to be proven below) that
\begin{equation}\label{e:approxpf1}
\sup_{\tau>0}\lim_{J\to\I} \varlimsup_{n\to\I} \bigl\| \tilde{N} (z_n , \tilde{\xi}_n^{0,J}) -\!\!\sum_{k=0}^{K(0,J)} \tilde{N} (z_n,\Lambda_n^{(0,k)}) - \tilde{N}(z_n,\Gamma_n^J)\bigr\|_{L^\frac83 L^\frac43 ([-\tau,\tau])} = 0.
\end{equation}
By the local uniform convergence $z_n \to z_\I^0$, we also have that
\[
\norm{ (B[z_n]-B[z_\I^0])\xi_\I^0 }_{L^1 H^1([-\tau,\tau])} + \norm{\tilde{N} (z_\I^0, \xi_\I^0) - \tilde{N} (z_n, \xi_\I^0) }_{L^\frac83 L^\frac43 ([-\tau,\tau])} \to 0
\]
as $n\to\I$, for each $\tau>0$. Thus,
\[
\sup_{\tau>0} \lim_{J\to\I} \varlimsup_{n\to\I} \norm{\mathcal{D}(\cdot,0;z_n)({\bf 1}_{[0,\tau]}\mathcal{E}_{0,n})}_{L^8L^4 ([0,\I)) } = 0
\]
(cf. \eqref{def:D} for the definition of $\mathcal{D}$).

Now, by definition, we have
\[
\xi_n(0)- \tilde{\xi}_n^{0,J}(0)=\sum_{s=s_{\min}}^{-1} \sum_{k=0}^{K(s,J)} \lambda_n^{(s,k)},
\]
so that Lemma~\ref{l:differentlinearprofile} (along with the fact that $s_n^0\equiv 0$) implies that
\begin{equation}\label{e:approxpf2}
\sup_{\tau>0} \lim_{J\to\I} \varlimsup_{n\to\I} \norm{\mathcal{L}(\cdot,0;z_n) (\xi_n(0)- \tilde{\xi}_n^{0,J}(0))  }_{L^8L^4 ([0,\tau]) } =0.
\end{equation}
Thus, by stability, we have
\begin{align}
\sup_{\tau>0} \lim_{J\to\I} \varlimsup_{n\to\I} \norm{\xi_n - \tilde{\xi}_n^{0,J}  }_{L^8L^4 ([0,\tau]) } =0, \label{e:approxpf3}\\
\sup_{\tau>0} \lim_{J\to\I} \varlimsup_{n\to\I} \norm{\xi_n - \tilde{\xi}_n^{0,J}  }_{[z_n;0,\tau;\I] } =0 \label{e:approxpf4}
\end{align}
(see \eqref{nakanishi-seminorms} for the semi-norm notation).  Using this together with  \eqref{e:approxpf2}, we also obtain
\begin{equation}\nonumber
\sup_{\tau>0} \lim_{J\to\I} \varlimsup_{n\to\I} \norm{\mathcal{L}(\cdot,\tau;z_n) (\xi_n(\tau) - \tilde{\xi}_n^{0,J}(\tau) )  }_{ L^8L^4 ([\tau,\I]) } =0,
\end{equation}
which is the desired estimate \eqref{k-induction-dominant} for $s_0=0$.  It remains to establish \eqref{e:approxpf1}. 

\begin{proof}[Proof of \eqref{e:approxpf1}] Let $\tau>0$.  We prove the estimate on $[0,\tau]$; similar arguments treat the interval $[-\tau,0]$.  Noting that by construction, we have
\[
\tilde{\xi}_n^{0,J}  - \sum_{s=1}^{s_{\max}} \sum_{k=0}^{K(s,J)} \lambda_n^{(s,k)} -\Gamma_n^J  = \sum_{k=0}^{K(0,J)} \Lambda_n^{(0,k)},
\]
we first estimate
\begin{align*}
&\varlimsup_{n\to\I}\biggl\| \tilde{N} (z_n , \tilde{\xi}_n^{0,J})  - \tilde{N}\bigl(z_n, \sum_{k=0}^{K(0,J)} \Lambda_n^{(0,k)} \bigr)  \biggr\|_{L^\frac83 L^\frac43 ({[0,\tau]})} \\
&\lesssim \biggl[\varlimsup_{n\to\I} \bigl\| \sum_{k=0}^{K(0,J)} \Lambda_n^{(0,k)} \bigr\|_{L^8L^4([0,\tau])}\biggr]^2 \varlimsup_{n\to\I} \bigl\| \sum_{s= 1}^{s_{\max}} \sum_{k=0}^{K(s,J)} \lambda_n^{(s,k)} + \Gamma_n^J \bigr\|_{L^8L^4([0,\tau])}.
\end{align*}

Now, by Proposition~\ref{p:Lambdas0bound}, Lemma~\ref{l:Lambda0ksum}, and the assumption that $\mathbb{M}(u_\I^0)<M_*$, we have
\begin{align*}
\sup_{J \ge J^\dagger + 1} \varlimsup_{n\to\I} \biggl\| &\sum_{k=0}^{K(0,J)} \Lambda_n^{(0,k)}\biggr\|_{L^8L^4([0,\tau])} \\
&\le \norm{\xi_\I^0}_{L^8L^4(\R)} + \sup_{K \ge 1} \varlimsup_{n\to\I} \norm{ \sum_{k=1}^{K} \Lambda_n^{(0,k)}}_{L^8L^4([0,\tau])}<\I.
\end{align*}
On the other hand, we see from Lemma~\ref{l:differentlinearprofile} that
\[
\varlimsup_{n\to\I} \norm{  \sum_{s= 1}^{s_{\max}} \sum_{k=0}^{K(s,J)} \lambda_n^{(s,k)} + \Gamma_n^J }_{L^8L^4([0,\tau])}=\varlimsup_{n\to\I} \norm{  \Gamma_n^J }_{L^8L^4([0,\tau])}.
\]
Thus, in view of \eqref{e:7.18b}, we have
\[
\sup_{\tau>0} \lim_{J\to\I} \varlimsup_{n\to\I} \norm{ \tilde{N} (z_n , \tilde{\xi}_n^{0,J})  - \tilde{N}(z_n, 
{\sum_{k=0}^{K(0,J)}\Lambda_{n}^{(0,k )})}
 }_{L^\frac83 L^\frac43 ([0,\tau])} =0,
\]
and similarly
\[
\sup_{\tau>0} \lim_{J\to\I} \varlimsup_{n\to\I} \norm{ \tilde{N}(z_n, \Gamma_n^J)  }_{L^\frac83 L^\frac43 ([0,\tau])} =0.
\]

Finally, using  \eqref{e:Lambda0kbound1}, \eqref{e:Lambda0kbound2}, and
the fact that $\Lambda_n^{(0,0)}=\xi_\I^0 \in L^8L^4(\R)$ satisfies
\[
	\lim_{R\to\I} \sup_n \norm{ {\bf 1}_{\{|x|\ge R\}} \Lambda_n^{(0,0)} }_{L^8L^4(\R)} =0,
\]
we can estimate
\[
\lim_{J\to\I} \varlimsup_{n\to\I} \bigl\|\tilde{N}\bigl(z_n, \sum_{k=0}^{K(0,J)} \Lambda_n^{(0,k)} \bigr)	 - \sum_{k=0}^{K(0,J)} \tilde{N} (z_n,\Lambda_n^{(0,k)})  \bigr\|_{L^\frac83 L^\frac43 ([0,\tau])} =0.
\]
Collecting the estimates, we obtain \eqref{e:approxpf1}.\end{proof}

\smallskip

\underline{(ii) Estimate \eqref{k-induction-dominant} at level $s_0$ implies estimate \eqref{k-induction-perturbative} at level $s_0+1$.} Suppose that \eqref{k-induction-dominant} holds at level $s_0<s_{\max}$.  As every profile such that $s(j)=s_0$ scatters forward in time, we have
\begin{equation}\label{e:approxpf27}
\varlimsup_{n\to\I} \(\norm{ \Lambda_n^{(s_0,k)} }_{ L^8L^4 ([\tilde\tau_n^{s_0},\I)) } + \norm{ \mathcal{L}(\cdot,\tau;z_n)\Lambda_n^{(s_0,k)}(\tilde\tau_n^{s_0}) }_{ L^8L^4 ([\tilde\tau_n^{s_0},\I)) } \)= 0
\end{equation}
as $\tau \to\I$ for all $k$.  Using \eqref{e:Lambda0ksum4}, we also obtain
\begin{equation}\label{e:approxpf28}
\lim_{\tau \to \I} \sup_{J \ge J^\dagger+1} \varlimsup_{n\to\I} \norm{ \sum_{k=1}^{K(s_0,J)} \mathcal{L}(\cdot,\tilde\tau_n^{s_0};z_n) \Lambda_n^{(s_0,k)}(\tilde\tau_n^{s_0}) }_{ L^8L^4 ([\tilde\tau_n^{s_0},\I)) } = 0.
\end{equation}
Inserting \eqref{e:approxpf27} and \eqref{e:approxpf28} into \eqref{k-induction-dominant}, we derive 
\begin{equation}\label{e:approxpf29}
\lim_{\tau\to\I} \varlimsup_{J\to\I} \varlimsup_{n\to\I} \bigl\|\mathcal{L}(\cdot,\tilde\tau_n^{s_0}; z_n) \bigl(\xi_n -\hat\xi_n^{s_0+1,J}\bigr)(\tilde\tau_n^{s_0}) \bigr\|_{ L^8L^4 ([\tilde\tau_n^{s_0},\I)) }=0,
\end{equation}
where $\hat\xi_n^{s_0+1,J}$ is as in \eqref{def-hat}. 

We need to upgrade \eqref{e:approxpf29} to obtain \eqref{k-induction-perturbative} at level $s_0+1$.  To this end, we first observe that 
\[
\mathcal{L}(t,\tilde\tau_n^{s_0} ;z_n)  \lambda_n^{(s,k)}(\tilde\tau_n^{s_0})=\lambda_n^{(s,k)}(t).
\]
Then, by Lemma~\ref{l:differentlinearprofile}, we have
\begin{equation}\label{e:approxpf30}
\lim_{\tau\to\I} \sup_{J\ge J^\dagger+1 } \varlimsup_{n\to\I} \norm{ \sum_{s\ge s_0 + 1}\sum_{k=0}^{K(s,J)} \lambda_n^{(s,k)} }_{ L^8L^4 ([\tilde\tau_n^{s_0},\hat\tau_n^{s_0+1}]) } =0,
\end{equation}
which (together with \eqref{e:approxpf29}) yields
\begin{equation}\label{e:approxpf31}
\lim_{\tau\to\I} \varlimsup_{J\to\I} \varlimsup_{n\to\I} \norm{\mathcal{L}(\cdot,\tilde\tau_n^{s_0} ;z_n) \(\xi_n - \Gamma_n^J \)(\tilde\tau_n^{s_0})  }_{ L^8L^4 ([\tilde\tau_n^{s_0},\hat\tau_n^{s_0+1}]) } =0.
\end{equation}
Since $\xi_n$ and $\Gamma_n^j$ solve the nonlinear equation
\[
i\partial_{t}v+Hv=B[z_{n}]v+\tilde{N}(z_{n},v),
\]
the stability theory (Lemma \ref{l:6.3}) yields
\[
\lim_{\tau\to\I} \varlimsup_{J\to\I} \varlimsup_{n\to\I} \norm{ \xi_n - \Gamma_n^J }_{ [z_n;\tilde\tau_n^{s_0},\hat\tau_n^{s_0+1};\I] } =0.
\]
In particular,
\begin{multline}\label{e:approxpf32}
\lim_{\tau\to\I} \varlimsup_{J\to\I} \varlimsup_{n\to\I} \| \mathcal{L}(\cdot,\hat\tau_n^{s_0+1}; z_n)(\xi_n - \Gamma_n^J)(\hat\tau_n^{s_0+1})\\
 -\mathcal{L}(\cdot,\tilde\tau_n^{s_0} ;z_n)(\xi_n - \Gamma_n^J)(\tilde\tau_n^{s_0}) \|_{ L^8L^4 ([\hat\tau_n^{s_0+1} ,\I)) } =0.
\end{multline}
By \eqref{e:approxpf30} and \eqref{e:approxpf32}, we have
\begin{equation}\label{e:approxpf33}
\lim_{\tau\to\I} \varlimsup_{J\to\I} \varlimsup_{n\to\I} \norm{\mathcal{L}(\cdot,\hat\tau_n^{s_0+1};z_n) (\xi_n - \hat{\xi}_n^{s_0+1,J})(\hat\tau_n^{s_0+1})   }_{ L^8L^4 ([\hat\tau_n^{s_0+1},\I)) } =0,
\end{equation}
which is \eqref{k-induction-perturbative} at level $s_0+1$, as desired.

\smallskip

\underline{(iii) Estimate \eqref{k-induction-perturbative} at level $s_0$ implies estimate \eqref{k-induction-dominant} at level $s_0$.}  Suppose that \eqref{k-induction-perturbative} holds at level $s_0\in[1,s_{\max}]$, and let $\tau>0$.

Note that by construction, we may write
\begin{equation}\label{e:approxpf15.5}
\tilde\xi_n^{s_0,J} = \hat{\xi}_n^{s_0,J} + \sum_{k=0}^{K(s_0,J)} (\Lambda_n^{(s_0,k)}-\lambda_n^{(s_0,k)})
\end{equation}
(cf. \eqref{def-tilde} and \eqref{def-hat}).  We now let $\chi\in C_0^\I$ be a smooth radial cutoff to the region $1\leq|x|\leq 2$ and set $\chi_R(x) = \chi(\frac{x}{R})$. Using \eqref{k-induction-perturbative} at level $s_0$, we have
\begin{equation}\label{e:approxpf16}
\lim_{\tau\to\I} \varlimsup_{J\to\I} \sup_{R>0}\varlimsup_{n\to\I} \norm{\chi_R \cdot\mathcal{L}(\cdot,\hat\tau_n^{s_0};z_n) ( \xi_n  - \hat{\xi}_n^{s_0,J}) (\hat\tau_n^{s_0})  }_{ L^8L^4 ([-\tau ,\tau]) } =0.
\end{equation}
By the local uniform convergence 
\[
(z_n(\cdot + s_n^{s_0} ),\xi_n(\cdot + s_n^{s_0} ))\to (z_\I^{s_0},\xi_\I^{s_0})\qtq{in}\C \times \text{w-}H^1,
\]
we can also show
\begin{equation}\label{e:approxpf17}
\begin{aligned}
\lim_{n\to\I} \|&\chi_R\cdot\mathcal{L}(\cdot,-\tau;z_n(\cdot+s_n^{s_0})) ({\xi}_n- \lambda_n^{(s_0,0)})(\hat\tau_n^{s_0}) \bigr\|_{ L^8L^4 ([-\tau ,\tau]) }  \\
&=\norm{\chi_R\cdot\mathcal{L}(\cdot,-\tau;z_\I^{s_0}) {\xi}_\I^{s_0}(-\tau) -  \mathcal{L} (\cdot ,0;z_\I^{s_0}) \varphi^{(s_0,0)}  }_{ L^8L^4 ([-\tau ,\tau]) }
\end{aligned}
\end{equation}
for all $R>0$.

\begin{proof}[Proof of \eqref{e:approxpf17}] First observe that for each $t \in [-\tau,\tau]$, we have
\[
\mathcal{L}(t,-\tau;z_n(\cdot+s_n^{s_0})) {\xi}_n(-\tau +s_n^{s_0}) \rightharpoonup \mathcal{L}(t,-\tau;z_\I^{s_0}) {\xi}_\I^{s_0}(-\tau)
\]
weakly in $H^1$ as $n\to\I$. Furthermore, 
\begin{align*}
\mathcal{L}(t,-\tau;z_n(\cdot+s_n^{s_0})) \lambda_n^{(s_0,0)}(-\tau +s_n^{s_0})&{}= \mathcal{L}(t,0;z_n(\cdot+s_n^{s_0})) \varphi^{(s_0,0)} \\
&{}\to \mathcal{L}(t,0;z_\I^{s_0})) \varphi^{(s_0,0)}
\end{align*}
in $L^\I ([-\tau,\tau], H^1)$ as $n\to\I$. Hence, for each $R>0$ and $t \in [-\tau,\tau]$,
\begin{align*}
\lim_{n\to\I} &\norm{\chi_R \cdot\mathcal{L}(\cdot,-\tau;z_n(\cdot+s_n^{s_0})) ({\xi}_n- \lambda_n^{(s_0,0)})(\hat\tau_n^{s_0}) }_{ L^4_x }  \\
&=\norm{\chi_R\cdot\mathcal{L}(\cdot,-\tau;z_\I^{s_0}) {\xi}_\I^{s_0}(-\tau) -  \mathcal{L} (\cdot ,0;z_\I^{s_0}) \varphi^{(s_0,0)}  }_{ L^4_x }.
\end{align*}
Furthermore,
\[
\sup_n \norm{\chi_R (\mathcal{L}(\cdot,-\tau;z_n(\cdot+s_n^{s_0})) ({\xi}_n- \lambda_n^{(s_0,0)})(\hat\tau_n^{s_0}))  }_{ L^4_x }\lesssim \sup_n \norm{\xi_n}_{L^\I H^1}.
\]
Since a constant function belongs to $L^8_t([-\tau,\tau])$, \eqref{e:approxpf17} therefore follows from Lebesgue's convergence theorem.
\end{proof}

Next, observe that as $y^{(s_0,k)}=\I$ for $k\ge 1$, we have
\begin{equation}\label{e:approxpf18}
\lim_{n\to\I} \norm{\chi_R \lambda_n^{(s_0,k)}  }_{ L^8L^4 ([-\tau,\tau]) } =0
\end{equation}
for any $k \ge 0$ and $\tau >0$. 

Inserting \eqref{e:7.18a}, \eqref{e:approxpf17}, and \eqref{e:approxpf18} into \eqref{e:approxpf16}, we see that
\begin{equation}\label{e:approxpf19}
\lim_{\tau\to\I} \norm{\mathcal{L}(\cdot,-\tau;z_\I^{s_0}) {\xi}_\I^{s_0}(-\tau) -  \mathcal{L} (\cdot ,0;z_\I^{s_0}) \varphi^{(s_0,0)} }_{ L^8L^4 ([-\tau ,\tau]) } =0.
\end{equation}
Using this, we will prove below that ${\xi}_\I^{s_0}$ scatters to $\varphi^{(s_0,0)}$ backward in time, i.e.
\begin{equation}\label{e:approxpf19.5}
\lim_{t\to-\I} \norm{ {\xi}_\I^{s_0}(t) - \mathcal{L} (t ,0;z_\I^{s_0}) \varphi^{(s_0,0)} }_{H^1} =0.
\end{equation}

On the other hand, we have
\begin{align*}
	\lim_{n\to\I} &\norm{ \mathcal{L}(\cdot, \hat\tau_n^{s_0};z_n)(\Lambda_{n}^{(s_0,0)} - \lambda_n^{(s_0,0)})(\hat\tau_n^{s_0})  }_{ L^8L^4 ([\hat\tau_n^{s_0},\tilde\tau_n^{s_0}]) } \\
	& = \lim_{n\to\I} \norm{ \mathcal{L}(\cdot, -\tau ;z_n(\cdot+ s_n^{s_0}))\xi_\I^{s_0}(-\tau) - \mathcal{L}(\cdot, 0 ;z_n(\cdot+ s_n^{s_0}))\varphi^{(s_0,0)}  }_{ L^8L^4 ([-\tau ,\tau]) } \\
& = \norm{ \mathcal{L}(\cdot, -\tau ;z_\I^{s_0})\xi_\I^{s_0}(-\tau) - \mathcal{L}(\cdot, 0 ;z_\I^{s_0})\varphi^{(s_0,0)} }_{ L^8L^4 ([-\tau ,\tau]) },
\end{align*}
where we have relied on the fact that 
\[
z_n(\cdot+ s_n^{s_0}) \to z_\I^{s_0}\qtq{in}L^\I_t ([-\tau,\tau])\qtq{as}n\to\I.
\]
Thus, using \eqref{e:approxpf19}, one has
\begin{equation}\label{e:approxpf20}
\lim_{\tau\to\I}\lim_{n\to\I} \norm{ \mathcal{L}(\cdot, \hat\tau_n^{s_0};z_n)(\Lambda_{n}^{(s_0,0)} - \lambda_n^{(s_0,0)})(-\tau + s_n^{s_0})  }_{ L^8L^4 ([\hat\tau_n^{s_0} ,\tilde\tau_n^{s_0}]) } =0.
\end{equation}

Now, for $k\ge1$, we write 
\begin{align*}
&(\Lambda_{n}^{(s_0,k)} - \lambda_n^{(s_0,k)})(\hat\tau_n^{s_0}) \\
&=P_c T_{y_n^{(s_0,k)}} e^{i\tau \Delta} \varphi^{(s_0,k)} - \mathcal{L}(-\tau,0;z_n) P_c T_{y_n^{(s_0,k)}}  \varphi^{(s_0,k)}\\
&= P_cT_{y_n^{(s_0,k)}} (e^{i\tau \Delta}- (T_{y_n^{(s_0,k)}})^{-1}\mathcal{L}(-\tau,0;z_n) P_c T_{y_n^{(s_0,k)}})  \varphi^{(s_0,k)}.
\end{align*}
Thus, by Lemma~\ref{l:spacetranslation},
\begin{equation}\label{e:approxpf21}
\lim_{n\to\I} \norm{(\Lambda_{n}^{(s_0,k)} - \lambda_n^{(s_0,k)})(\hat\tau_n^{s_0})}_{H^1} = 0
\end{equation}
Combining \eqref{k-induction-perturbative} at level $s_0$, \eqref{e:approxpf15.5}, \eqref{e:approxpf20}, and \eqref{e:approxpf21}, we obtain
\begin{equation}\label{e:approxpf22}
\lim_{\tau\to\I} \varlimsup_{J\to\I} \varlimsup_{n\to\I} \norm{\mathcal{L}(\cdot,\hat\tau_n^{s_0};z_n) ( \xi_n  - \tilde{\xi}_n^{s_0,J}) (-\tau +s_n^{s_0})   }_{ L^8L^4 ([\hat\tau_n^{s_0},\I)) } =0.
\end{equation}

We now define the error
\begin{align*}
\mathcal{E}_{s_0,n}(t) &{}:= ((i\d_t + H- B[z_n])\tilde{\xi}_n^{J} - \tilde{N} (z_n , \tilde{\xi}_n^{J}))(t+s_n^{s_0}) \\
&{} =  \left(-(B[z_n(\cdot + s_n^{s_0})]-B[z_\I^{s_0}])\xi_\I^{s_0}  + (\tilde{N} (z_\I^{s_0}, \xi_\I^{s_0}) - \tilde{N} (z_n(\cdot + s_n^{s_0}), \xi_\I^{s_0}))\right)(t)\\
&{}\quad - \bigl(\tilde{N} (z_n , \tilde{\xi}_n^J)  -\sum_{k=0}^{K(s_0,J)} \tilde{N} (z_n,\Lambda_n^{(s_0,k)}) - \tilde{N}(z_n,\Gamma_n^J)  \bigr)(t+s_n^{s_0}).
\end{align*}
We will prove below that
\begin{equation}\label{e:approxpf23}
\sup_{\tau>0}\lim_{J\to\I} \varlimsup_{n\to\I} \norm{ \tilde{N} (z_n , \tilde{\xi}_n^{s_0,J}) -\!\!\!\sum_{k=0}^{K(s_0,J)}\!\!\!\tilde{N} (z_n,\Lambda_n^{(s_0,k)}) - \tilde{N}(z_n,\Gamma_n^J)  }_{L^\frac83 L^\frac43 ([\hat\tau_n^{s_0},\tilde\tau_n^{s_0}])} = 0.
\end{equation}
By the local uniform convergence $z_n(\cdot + s_n^{s_0}) \to z_\I^{s_0}$, we also obtain
\begin{align*}
\| (B&[z_n(\cdot + s_n^{s_0})]-B[z_\I^{s_0}])\xi_\I^0 \|_{L^1 H^1([-\tau,\tau])}  \\
&+\norm{\tilde{N} (z_\I^{s_0}, \xi_\I^{s_0}) - \tilde{N} (z_n(\cdot + s_n^{s_0}), \xi_\I^{s_0}) }_{L^\frac83 L^\frac43 ([-\tau,\tau])} \to 0
\end{align*}
as $n\to\I$, for each $\tau>0$. Thus, by stability (Lemma \ref{l:6.3}), we have
\begin{align}
&\lim_{\tau\to\I} \lim_{J\to\I} \varlimsup_{n\to\I} \norm{\xi_n - \tilde{\xi}_n^{s_0,J}  }_{L^8L^4 ([\hat\tau_n^{s_0},\tilde\tau_n^{s_0}]) } =0,\label{e:approxpf24} \\
&\lim_{\tau\to\I} \lim_{J\to\I} \varlimsup_{n\to\I} \norm{\xi_n - \tilde{\xi}_n^{s_0,J}  }_{[z_n;\hat\tau_n^{s_0},\tilde\tau_n^{s_0};\I] } =0. \label{e:approxpf25}
\end{align}
Using this together with \eqref{e:approxpf20}, we then obtain
\begin{equation}\nonumber
\sup_{\tau>0} \lim_{J\to\I} \varlimsup_{n\to\I} \norm{\mathcal{L}(\cdot,\tilde\tau_n^{s_0} ;z_n) (\xi_n - \tilde{\xi}_n^{s_0,J} ) (\tilde\tau_n^{s_0}) }_{ L^8L^4 ([\tilde\tau_n^{s_0},\I]) } =0,
\end{equation}
which is \eqref{k-induction-dominant} at level $s_0$, as desired. 

It remains to prove \eqref{e:approxpf19.5} and \eqref{e:approxpf23}.

\begin{proof}[Proof of \eqref{e:approxpf19.5}] Let $\eps>0$ to be determined below, and choose $\tau_*>0$ such that
\[
\norm{ \mathcal{L}(\cdot,0;z_\I^{s_0}) \varphi^{(s_0,0)} }_{L^8L^4((-\I,-\tau_*])} \le \eps.
\]
Using \eqref{e:approxpf19}, we may also find $\tau_{**} \ge \tau_*$ such that
\[
\sup_{\tau > \tau_{**}} \norm{\mathcal{L}(\cdot,-\tau;z_\I^{s_0}) {\xi}_\I^{s_0}(-\tau) -  \mathcal{L} (\cdot ,0;z_\I^{s_0}) \varphi^{(s_0,0)}  }_{ L^8L^4 ([-\tau ,\tau]) } \le \eps.
\]
Thus, for $\tau \ge \tau_{**}$ we have
\begin{equation*}
 \norm{\mathcal{L}(\cdot,-\tau;z_\I^{s_0}) {\xi}_\I^{s_0}(-\tau)  }_{ L^8L^4 ([-\tau ,-\tau_{*}]) } \le 2\eps.
\end{equation*}
By Lemma \ref{l:sdt}, this implies
\[
 \norm{ {\xi}_\I^{s_0}  }_{ L^8L^4 ([-\tau ,-\tau_{*}]) } \le 4\eps.
\]
As this holds for arbitrary $\tau > \tau_{**}$, we deduce
\[
\norm{ {\xi}_\I^{s_0}  }_{ L^8L^4 ((-\I ,-\tau_{*}]) } \le 4\eps,
\]
which implies that ${\xi}_\I^{s_0}$ scatters backward in time.  That is, there exists $\varphi \in H^1$ such that
\[
\lim_{t\to-\I} \norm{ {\xi}_\I^{s_0}(t) - \mathcal{L} (t ,0;z_\I^{s_0}) \varphi }_{H^1} =0.
\]

It remains to show that $\varphi = \varphi^{(s_0,0)} $. To this end, we apply Strichartz to obtain
\begin{align*}
\bigl\| \mathcal{L} &(\cdot ,0;z_\I^{s_0})( \varphi - \varphi^{(s_0,0)} ) \bigr\|_{ L^8L^4 ([-\tau ,\tau]) } \\
&\le \norm{ \mathcal{L} (\cdot ,0;z_\I^{s_0}) \varphi - \mathcal{L}(\cdot,-\tau;z_\I^{s_0}) {\xi}_\I^{s_0}(-\tau) }_{ L^8L^4 (\R) }  \\
&\quad + \norm{\mathcal{L}(\cdot,-\tau;z_\I^{s_0}) {\xi}_\I^{s_0}(-\tau) - \mathcal{L} (\cdot ,0;z_\I^{s_0}) \varphi^{(s_0,0)}  }_{ L^8L^4 ([-\tau ,\tau]) } \\
&\lesssim\norm{ \mathcal{L} (-\tau ,0;z_\I^{s_0}) \varphi - {\xi}_\I^{s_0}(-\tau) }_{ H^1 }  \\
&\quad + \norm{\mathcal{L}(\cdot,-\tau;z_\I^{s_0}) {\xi}_\I^{s_0}(-\tau) - \mathcal{L} (\cdot ,0;z_\I^{s_0}) \varphi^{(s_0,0)}  }_{ L^8L^4 ([-\tau ,\tau]) }.
\end{align*}
Sending $\tau\to\I$, we obtain
\[
\norm{ \mathcal{L} (\cdot ,0;z_\I^{s_0})( \varphi - \varphi^{(s_0,0)} ) }_{ L^8L^4 (\R) } =0,
\]
so that $\varphi = \varphi^{(s_0,0)} $, as desired. \end{proof}

\begin{proof}[Proof of \eqref{e:approxpf23}] Fix $\tau>0$ and recall that
\[
\tilde{\xi}_n^{s_0,J}  - \sum_{s= s_0+1}^{s_{\max}} \sum_{k=0}^{K(s,J)} \lambda_n^{(s,k)} -\Gamma_n^J  = \sum_{k=0}^{K(s_0,J)} \Lambda_n^{(s_0,k)}.
\]
We begin by estimating
\begin{align*}
\varlimsup_{n\to\I}& \bigl\| \tilde{N} (z_n , \tilde{\xi}_n^J) - \tilde{N}\bigl(z_n, \sum_{k=0}^{K(s_0,J)} \Lambda_n^{(s_0,k)} \bigr)  \bigr\|_{L^\frac83 L^\frac43 ([\hat\tau_n^{s_0},\tilde\tau_n^{s_0}])} \\
&\lesssim\bigl(\varlimsup_{n\to\I} \bigl\| \sum_{k=0}^{K(s_0,J)} \Lambda_n^{(s_0,k)} \bigr\|_{L^8L^4([\hat\tau_n^{s_0},\tilde\tau_n^{s_0}])}\bigr)^2 
\\ &\quad\quad\quad\times \varlimsup_{n\to\I} \bigl\|  \sum_{s= s_0+1}^{s_{\max}} \sum_{k=0}^{K(s,J)} \lambda_n^{(s,k)} + \Gamma_n^J \bigr\|_{L^8L^4([\hat\tau_n^{s_0},\tilde\tau_n^{s_0}])}.
\end{align*}

Now, by Proposition~\ref{p:Lambdas0bound} and Lemma~\ref{l:Lambdasksum},  we have
\begin{align*}
\sup_{J \ge J^\dagger + 1}& \varlimsup_{n\to\I} \norm{ \sum_{k=0}^{K(s_0,J)} \Lambda_n^{(s_0,k)}}_{L^8L^4([\hat\tau_n^{s_0},\tilde\tau_n^{s_0}])} \\
&\le \norm{\xi_\I^{s_0}}_{L^8L^4(\R)} + \sup_{K \ge 1} \varlimsup_{n\to\I} \norm{ \sum_{k=0}^{K} \Lambda_n^{(s_0,k)}}_{L^8L^4([\hat\tau_n^{s_0},\tilde\tau_n^{s_0}])}<\I,
\end{align*}
where we have used the assumption $\mathbb{M}(u^{s_0}_\infty)<M_*$.  Note that by \eqref{e:Lambdasksum1}, the supremum of the right-hand side with respect to $\tau>0$ is bounded.  On the other hand, we see from Lemma~\ref{l:differentlinearprofile} that
\[
\varlimsup_{n\to\I} \norm{  \sum_{s= s_0+1}^{s_{\max}} \sum_{k=0}^{K(s,J)} \lambda_n^{(s,k)} + \Gamma_n^J }_{L^8L^4([\hat\tau_n^{s_0},\tilde\tau_n^{s_0}])}=\varlimsup_{n\to\I} \norm{  \Gamma_n^J }_{L^8L^4([\hat\tau_n^{s_0},\tilde\tau_n^{s_0}])}.
\]
Thus, recalling \eqref{e:7.18b}, we have
\[
\sup_{\tau>0} \lim_{J\to\I} \varlimsup_{n\to\I} \norm{ \tilde{N} (z_n , \tilde{\xi}_n^{s_0,J})  - \tilde{N}(z_n, \tilde{\xi}_n^{s_0,J} -\Gamma_n^J)  }_{L^\frac83 L^\frac43 ([\hat\tau_n^{s_0},\tilde\tau_n^{s_0}])} =0.
\]

Similarly,
\[
\sup_{\tau>0} \lim_{J\to\I} \varlimsup_{n\to\I} \norm{ \tilde{N}(z_n, \Gamma_n^J)  }_{L^\frac83 L^\frac43 ([\hat\tau_n^{s_0},\tilde\tau_n^{s_0}])} =0.
\]

Finally, we use \eqref{e:Lambdaskbound2}, \eqref{e:Lambdaskbound2.5}, and
the fact that $\Lambda_n^{(s_0,0)}(\cdot-s_n^{s_0})=\xi_\I^{s_0} \in L^8L^4(\R)$ satisfies
\[
\lim_{R\to\I} \sup_n \norm{ {\bf 1}_{\{|x|\ge R\}} \Lambda_n^{(s_0,0)} }_{L^8L^4(\R)} =0
\]
to see that for any $\tau>0$, we have
\[
\lim_{J\to\I} \varlimsup_{n\to\I} \bigl\| \tilde{N}\bigl(z_n, \sum_{k=0}^{K(s_0,J)} \Lambda_n^{(s_0,k)} \bigr)- \sum_{k=0}^{K(s_0,J)} \tilde{N} (z_n,\Lambda_n^{(s_0,k)})  \bigr\|_{L^\frac83 L^\frac43 ([\hat\tau_n^{s_0},\tilde\tau_n^{s_0}])} =0.
\]
\end{proof}

\underline{(iv) Derivation of \eqref{ultimately}.} Finally, we use \eqref{k-induction-dominant} at level $s_{\max}$ to derive \eqref{ultimately} and complete the proof of Proposition~\ref{p:keykey}.  The argument is the same as part (ii), in which we used \eqref{k-induction-dominant} at level $s_0$ to deduce \eqref{k-induction-perturbative} at level $s_0+1$.  In particular, arguing as we did to obtain \eqref{e:approxpf29}, we find that
\[
\|\mathcal{L}(\cdot,\tau_n^{s_{\max}};z_n)(\xi_n - \Gamma_n^J)(\tau_n^{s_{\max}})\|_{L^8 L^4([\tau_n^{s_{\max}},\infty))} \to 0
\]
as $n,J,\tau\to\infty$, which is \eqref{ultimately} (Essentially,  $\hat\xi_n^{J,s_{\max}+1}$ just reduces to $\Gamma_n^J$.)  This completes the proof of Proposition~\ref{p:keykey}.\end{proof}

With Proposition~\ref{p:keykey} in hand, we are finally in a position to complete the proof of Proposition~\ref{p:key}. 

\begin{proof}[Completion of the proof of Proposition~\ref{p:key}] Let $s_+ \in [0,s_{\max}]$ be the number defined in Proposition~\ref{p:keykey}.

If $s_+=0$, then the weak $H^1$-convergence $u_n(0) \rightharpoonup u_\I^0(0)$ and the norm convergence
\[
\lim_{n\to\I} \mathbb{M}(u_n(0)) = \mathrm{M}(u_\I^0(0))
\]
imply that the convergence $u_n(0) \to u_\I^0(0)$ is strong in $H^1$.  This implies the desired result with $u_{0,\I}=u_\I^0(0)$.  Indeed, if $u_\I^0$ scatters to $\mathscr{S}_0$ forward in time, that is, $\norm{\xi_\I^0}_{L^8L^4[0,\I) }<\I,$ then the stability theory implies
\[
\sup_n \norm{\xi_n}_{L^8L^4([0,\I)) }<\I,
\]
which is a contradiction. Thus $u_\I^0$ does not scatter to $\mathscr{S}_0$ forward in time, and by the same argument, we obtain the failure of scattering to $\mathscr{S}_0$ backward in time.

Suppose instead that $s_+ \neq0$. Without loss of generality, we may suppose that $s_+$ is minimal in the sense that
\[
(\mathbb{M} (u_\I^{s}) , \mathbb{E}_V (u_\I^{s}))  \neq (M_*,E_*)\qtq{for}s\in[0,s_+-1].
\]
In this case, the induction argument in the proof of Proposition~\ref{p:keykey} works up to $s_+$; in particular, we have \eqref{e:approxpf19.5} for $s_0=s_+$. Then, since $\mathbb{E}_V (u_\I^{s_+})=E_*$, we see that
\[
\lim_{t\to-\I}\mathbb{H}_0 (\mathcal{L}(t,0;z_\I^{s_+}) \varphi^{(s_+,0)} )= \lim_{t\to-\I}\mathbb{H}_V (\xi_\I^{s_+}(t)) \ge\mathbb{E}_V (u_\I^{s_+})-C\mu=E_*- C\mu.
\]
Since $s_- \le -1$, the same argument shows
\[
\lim_{t\to-\I}\mathbb{H}_0 (\mathcal{L}(t,0;z_\I^{s_+}) \varphi^{(s_+,0)} ) \ge E_* - C \mu.
\]
Combining these bounds with the energy decomposition in the linear profile decomposition, we conclude that
\[
E_* \ge 2 (1+ o(\mu))(E_*- C \mu). 
\]
This implies $E_* \lesssim \mu$, which is a contradiction.\end{proof}


%

\end{document}